\documentclass[11pt]{amsart}

\usepackage[
paper=a4paper,
text={147mm,230mm},centering
]{geometry}

\iftrue
\makeatletter
\def\@settitle{%
  \vspace*{-20pt}
  \begin{flushleft}%
    \baselineskip14\p@\relax
    \normalfont\bfseries\LARGE
    \@title
  \end{flushleft}%
}
\def\@setauthors{
  \begingroup
  \def\thanks{\protect\thanks@warning}%
  \trivlist
  \large \@topsep30\p@\relax
  \advance\@topsep by -\baselineskip
  \item\relax
  \author@andify\authors
  \def\\{\protect\linebreak}%

  \authors
  \ifx\@empty\contribs
  \else
    ,\penalty-3 \space \@setcontribs
    \@closetoccontribs
  \fi
  \normalfont
  \endtrivlist
  \endgroup
}
\def\@setabstracta{%
    \ifvoid\abstractbox
  \else
    \skip@25\p@ \advance\skip@-\lastskip
    \advance\skip@-\baselineskip \vskip\skip@
    \box\abstractbox
    \prevdepth\z@ 
    \vskip-10pt
  \fi
}
\renewenvironment{abstract}{%
  \ifx\maketitle\relax
    \ClassWarning{\@classname}{Abstract should precede
      \protect\maketitle\space in AMS document classes; reported}%
  \fi
  \global\setbox\abstractbox=\vtop \bgroup
    \normalfont\small
    \list{}{\labelwidth\z@
      \leftmargin0pc \rightmargin\leftmargin
      \listparindent\normalparindent \itemindent\z@
      \parsep\z@ \@plus\p@
      
    }%
    \item[\hskip\labelsep\bfseries\abstractname.]%
}{%
  \endlist\egroup
  \ifx\@setabstract\relax \@setabstracta \fi
}
\def\section{\@startsection{section}{1}%
  \z@{-1.2\linespacing\@plus-.5\linespacing}{.8\linespacing}%
  {\normalfont\bfseries\large}}
\def\subsection{\@startsection{subsection}{2}%
  \z@{-.8\linespacing\@plus-.3\linespacing}{.3\linespacing\@plus.2\linespacing}%
  {\normalfont\bfseries}}
\def\subsubsection{\@startsection{subsubsection}{3}%
  \z@{.7\linespacing\@plus.1\linespacing}{-1.5ex}%
  {\normalfont\itshape}}
\def\@secnumfont{\bfseries}
\makeatother
\fi 

\usepackage{amssymb}
\usepackage{mathrsfs}
\usepackage[all]{xy}
\usepackage{graphicx,color,float}
\usepackage[bookmarks]{hyperref}
\usepackage{pinlabel}
\usepackage[textwidth=1.3in,color=yellow]{todonotes}
\usepackage{amsmath}
\usepackage{pb-diagram,pb-xy}
\usepackage{amsmath}
\usepackage{amsfonts}
\usepackage{graphicx}
\usepackage{amscd}
\usepackage{mathtools}
\usepackage{url}
\usepackage{caption}
\usepackage{subcaption}
\usepackage{fancybox}
\usepackage{wrapfig}
\usepackage{color}
\usepackage{multirow, url}
\usepackage{tikz}
\usetikzlibrary{shapes.geometric, arrows}
\usetikzlibrary{shapes}
\usepackage{xcolor}
\usepackage[normalem]{ulem}  

\usetikzlibrary{calc}
\textwidth=\paperwidth \advance\textwidth by-2\oddsidemargin
\advance\oddsidemargin by-1in
\evensidemargin=\oddsidemargin
\calclayout

\newcommand{\bigslant}[2]{{\raisebox{.2em}{$#1$}\left/\raisebox{-.2em}{$#2$}\right.}}

\renewcommand{\tilde}{\widetilde}

\newcommand{\C}{\mathbb{C}}

\newcommand{\N}{\mathbb{N}}
\newcommand{\Q}{\mathbb{Q}}
\newcommand{\R}{\mathbb{R}}

\newcommand{\Z}{\mathbb{Z}}
\newcommand{\p}{\mathbb{P}}

\newcommand{\CP}{\mathbb{C}P}

\tikzstyle{startstop} = [rectangle, rounded corners, 
minimum width=3cm, 
minimum height=1cm,
text centered, 
draw=black]

\tikzstyle{io} = [trapezium, 
trapezium stretches=true, 
trapezium left angle=70, 
trapezium right angle=110, 
minimum width=3cm, 
minimum height=1cm, text centered, 
draw=black, fill=blue!30]

\tikzstyle{process} = [rectangle, rounded corners, 
minimum width=4cm, 
minimum height=1cm,
text centered, 
draw=black]

\tikzstyle{decision} = [diamond, 
minimum width=3cm, 
minimum height=1cm, 
text centered, 
draw=black, 
fill=green!30]
\tikzstyle{arrow} = [thick,->,>=stealth]

\def\mcal{\mathcal}
\def\frak{\mathfrak}
\def\scr{\mathscr}

\numberwithin{equation}{section}

\theoremstyle{plain}
\newtheorem{theorem}{Theorem}[section]
\newtheorem{thmx}{Theorem}
\newtheorem{propx}[thmx]{Proposition}

\newtheorem{proposition}[theorem]{Proposition}
\newtheorem{corollary}[theorem]{Corollary}

\newtheorem{lemma}[theorem]{Lemma}

\theoremstyle{definition}
\newtheorem{definition}[theorem]{Definition}
\newtheorem{assumption}[theorem]{Assumption}

\newtheorem{example}[theorem]{Example}
\newtheorem{remark}[theorem]{Remark}
\newtheorem*{remarkA}{Remark}

\def\to{\mathchoice{\longrightarrow}{\rightarrow}{\rightarrow}{\rightarrow}}
\makeatletter
\newcommand{\shortxra}[2][]{\ext@arrow 0359\rightarrowfill@{#1}{#2}}
\def\longrightarrowfill@{\arrowfill@\relbar\relbar\longrightarrow}
\newcommand{\longxra}[2][]{\ext@arrow 0359\longrightarrowfill@{#1}{#2}}

\makeatother
\numberwithin{equation}{section}

\newcommand{\g}{\mathfrak{g}}

\newcommand{\Gup}{\mathsf{G}^{\rm up}}
\newcommand{\Gupq}{\mathsf{G}_{q=1}^{\rm up}}
\newcommand{\rxw}{\underline{w}}
\newcommand{\fr}{\mathrm{fr}}
\newcommand{\uf}{\mathrm{uf}}
\newcommand{\cmA}{\mathsf{A}}
\newcommand{\vep}{\varepsilon}
\newcommand{\seed}{\mathsf{s}}
\newcommand{\bfa}{\mathbf{a}}
\newcommand{\bfb}{\mathbf{b}}

\newcommand{\Lb}{\mathcal{L}}

\newcommand{\Span}{\mathrm{Span}}
\newcommand{\supp}{\mathrm{supp}}
\newcommand{\bfi}{\mathbf{i}}
\newcommand{\bfe}{\mathbf{e}}
\newcommand{\bfu}{\mathbf{u}}

\usepackage{marginnote}

\newenvironment{Red}
{\relax\color{red}}
{\hspace*{.5ex}\relax}

\newenvironment{Yel}
{\relax\color{yellow}}
{\hspace*{.5ex}\relax}

\newcommand{\ber}{\begin{Red}}
\newcommand{\er}{\end{Red}}

\newcommand{\beb}{\begin{Yel}}
\newcommand{\eb}{\end{Yel}}
	
\newcommand{\berE}{\begin{Red}{}\marginnote{\fbox{\scshape\lowercase{E}}}{}}

\newcommand{\berMH}{\begin{Red}{}\marginnote{\fbox{\scshape\lowercase{MH}}}{}}

\newcommand{\berYS}{\begin{Red}{}\marginnote{\fbox{\scshape\lowercase{YS}}}{}}

\newcommand{\berYH}{\begin{Red}{}\marginnote{\fbox{\scshape\lowercase{YH}}}{}}

\begin{document}

\title{Cluster algebras and monotone Lagrangian tori }

\author{Yunhyung Cho}
\address{Department of Mathematics Education, Sungkyunkwan University, Seoul, Republic of Korea}
\email{yunhyung@skku.edu}
\thanks{The research of Y.\ Cho was supported by the  National Research Foundation of Korea(NRF) grant funded by the Korea government (MSIT) (NRF-2020R1C1C1A01010972 and NRF-2020R1A5A1016126).}

\author{Myungho Kim}
\address{Department of Mathematics, Kyung Hee University, Seoul, Republic of Korea}
\email{mkim@khu.ac.kr}
\thanks{The research of M.  \ Kim was supported by the  National Research Foundation of Korea(NRF) grant funded by the Korea government (MSIT) (NRF-2022R1F1A1076214 and NRF-2020R1A5A1016126).}

\author{Yoosik Kim}
\address{Department of Mathematics and Institute of Mathematical Science, Pusan National University, Busan, Republic of Korea}
\email{yoosik@pusan.ac.kr}
\thanks{The research of Y.  \ Kim was supported by the  National Research Foundation of Korea(NRF) grant funded by the Korea government (MSIT) (NRF-2021R1F1A1057739 and NRF-2020R1A5A1016126).}

\author{Euiyong Park}
\address{Department of Mathematics, University of Seoul, Seoul, Republic of Korea}
\email{epark@uos.ac.kr}
\thanks{The research of E.\ Park was supported by the  National Research Foundation of Korea(NRF) grant funded by the Korea government (MSIT) (RS-2023-00273425 and NRF-2020R1A5A1016126).}


\begin{abstract}
Motivated by the construction of Newton--Okounkov bodies and toric degenerations via cluster algebras in \cite{GrossHackingKeelKontsevich, FO20}, we consider a family of Newton--Okounkov polytopes of a complex smooth Fano variety $X$ related by a composition of tropicalized cluster mutations. According to the work of \cite{HaradaKaveh}, the toric degeneration associated with each Newton--Okounkov polytope $\Delta$ in the family produces a completely integrable system of $X$ over $\Delta$. We investigate circumstances in which each completely integrable system possesses a monotone Lagrangian torus fiber. We provide a sufficient condition, based on the data of tropical integer points and exchange matrices, for the family of constructed monotone Lagrangian tori to contain infinitely many monotone Lagrangian tori, no two of which are related by any symplectomorphism. By employing this criterion and exploiting the correspondence between the tropical integer points and the dual canonical basis elements, we generate infinitely many distinct monotone Lagrangian tori on flag manifolds of arbitrary type except in a few cases.
\end{abstract}
\maketitle
\setcounter{tocdepth}{1} 
\tableofcontents

\section{Introduction}\label{secIntroduction}

\subsection{Background}

Cluster algebras, introduced by Fomin--Zelevinsky \cite{FZ02}, are a class of $\Z$-subalgebras of a rational function field with a special combinatorial structure. This algebra possesses a generating set divided into subsets called \emph{clusters}, and these clusters are connected by special combination procedures called \emph{mutations}. The generating set can be recursively constructed from a \emph{seed} consisting of a cluster together with an \emph{exchange matrix}, a skew-symmetrizable matrix that determines mutations. Cluster algebras have naturally appeared in many research areas including geometry, representation theory, and combinatorics, numerous applications have been studied, and they now have become an active area of mathematical research. 

The algebro-geometric counterpart of a cluster algebra is a \emph{cluster variety}. By gluing split algebraic tori via two different types of mutations, one produces a pair of cluster varieties; an $\mathcal{A}$-cluster variety and an $\mathcal{X}$-cluster variety. The full Fock--Goncharov (FG) conjecture \cite{FockGoncharov06} asserts that the ring of regular functions of one cluster variety admits a basis whose elements are parametrized by tropical integer points of the other cluster variety.  A foundational work of Gross--Hacking--Keel--Kontsevich (GHKK) \cite{GrossHackingKeelKontsevich} reformulates and resolves the conjecture by constructing a \emph{theta basis}, whose elements correspond to a broken line in a scattering diagram associated with the cluster structure. The pair of cluster varieties is conceived as a mirror pair, see \cite{GrossHackingKeelMSCY1, HackingKeel} for a precise formulation. Additionally, it is expected that the mirror of a compactified $\mathcal{A}$-cluster variety is given by the \emph{disk potential function} on the dual cluster $\mathcal{X}$-variety, a generating function for invariants counting holomorphic disks in \cite{ChoOh, AurouxTduality, FOOOToric1}.

This paper focuses on a situation where the FG conjecture holds and explores an application of the theory of cluster algebras to symplectic topology. To put our work in the above context, we formulate the conjecture on mirror symmetry of the dual pair of cluster varieties more concretely. A work of Fujita--Oya \cite{FO20} revealed that a refinement of the Qin's dominance order in \cite{Qin17} on a lattice for an $\mathcal{A}$-cluster variety produces a Newton--Okounkov body $\Delta$. It is a rational polytope and generates GHKK's toric degeneration on a compactification $X$ of the $\mathcal{A}$-cluster variety. The tropical integer points of $\Delta$ parametrize a theta basis respecting the $\mathcal{X}$-cluster structure. Moreover, the polytope  $\Delta$ produces a completely integrable system on $X$ over $\Delta$ by Harada--Kaveh \cite{HaradaKaveh}. In particular, we obtain a Lagrangian torus fibration on (an open dense part of) $X$ associated with each seed. Conjecturally, by gluing the (Floer theoretical) SYZ mirrors together with disk potential functions of those completely integrable systems, one should (partially) recover the dual $\mathcal{X}$-cluster variety and the superpotential given by the sum of theta basis elements corresponding to an irreducible component of the anticanonical divisor for $X$ given by the frozen variables.

In this paper, we are concerned with completely integrable systems on the compactified $\mathcal{A}$-cluster variety $X$, constructed from a cluster algebra as described above. Specifically, our focus lies on \emph{monotone} Lagrangian tori, refer to Definition~\ref{def_monotoneLag} for a precise definition. However, \emph{not} every toric degeneration can yield a monotone Lagrangian torus. To describe situations where a monotone Lagrangian torus is constructed, we recall some notions. A full-dimensional rational polytope $\Delta \subset \R^m$ is called \emph{$\mathbb{Q}$-Gorenstein Fano polytope} if there exists a vector $\mathbf{u}_0 \in \R^m$ and $\nu \in \mathbb{N}$ such that $\nu \cdot (\Delta - {\bf u}_0)^\circ$ is a lattice polytope each of which vertex is a \emph{primitive lattice} vector. Such a point $\mathbf{u}_0$ uniquely exists and is called the \emph{center} of $\Delta$. The following proposition specifies a circumstance in which the toric degeneration produces a monotone Lagrangian torus, see Proposition~\ref{proposition_toricdegenerationsmonotone} for a more precise statement.

\begin{propx}[Proposition~\ref{proposition_toricdegenerationsmonotone}]\label{propA}
Let $X$ be a smooth Fano variety and let $\mathcal{L}$ be a very ample line bundle, given by a positive power of its anticanonical line bundle $K_X^{-1}$. Let us equip $X$ with the K\"{a}hler form inherited from the projective embedding induced by $\mathcal{L}$. Suppose that the Newton--Okounkov polytope $\Delta$ of $\mathcal{L}$ is a $\Q$-Gorenstein Fano polytope and the central fiber of the associated toric degeneration is a normal toric variety. We denote a completely integrable system constructed from this toric degeneration by $\Phi \colon X \to \Delta$.  Then the fiber of $\Phi$ located at the center of $\Delta$ is a monotone Lagrangian torus.
\end{propx}

\begin{remarkA}
Galkin--Mikhalkin in \cite{GalkinMikhalkin} proved an analogous statement, see Remark~\ref{remark_GM22}.
\end{remarkA}

The monotonicity condition ensures that disks bounded by a monotone Lagrangian torus $L$ do \emph{not} pose any serious obstruction to produce mirror spaces via Lagrangian Floer theory of $L$. Moreover, each monotone Lagrangian torus $L$ gives rise to a complex algebraic torus as mirror space, which fits into our interest in cluster varieties.

Besides the aspects of cluster duality and Floer theoretical SYZ mirror symmetry, there are additional motivations for considering monotone Lagrangian tori. One of the versions of mirror symmetry conjecture for Fano varieties asserts that for each $\Q$-factorial Fano variety $X$ having at worst terminal singularities, there exists a Laurent polynomial $W \colon (\C^*)^m \rightarrow \C$ called a {\em weak Landau--Ginzburg mirror} such that its period integral 
            \[
                \pi_W(t) \coloneqq \int_{|x_1| = \dots = |x_m| = 1} \frac{1}{1 - tW} \Omega, \quad \quad \Omega = \frac{dx_1}{x_1}
                \wedge \dots \wedge \frac{dx_m}{x_m}
            \]
         is equal to a {\em regularized quantum period}, i.e., a generating function of gravitational descendant one-point Gromov--Witten invariants of $X$, see \cite{Prz07, KP11} for instance.
         In \cite{Ton18}, Tonkonog proved that $W$ can be obtained as the disk potential function of a monotone Lagrangian torus in $X$. Thus the mirror symmetry conjecture will be proved if one can succeed in finding a monotone Lagrangian torus in a given Fano manifold.
         Proposition~\ref{propA} extends this result to an arbitrary smooth Fano variety admitting a {\em $\Q$-Gorenstein Fano} and \emph{normal toric degeneration}. 

The notion of monotonicity was introduced and used to construct Lagrangian Floer homology by Oh \cite{Ohmonotone} and Biran--Cornea \cite{BiranCornea}. In symplectic topology, it is an interesting problem to construct a new monotone Lagrangian submanifold not related to any pre-existing monotone Lagrangian submanifold through any symplectomorphism. In \cite{Chekanov}, Chekanov first constructed a monotone Lagrangian torus (called a \emph{Chekanov torus}) that is not Hamiltonian isotopic to any standard product torus in $\R^{2m}$. By suitably embedding a Chekanov torus into $\CP^m$, Chekanov and Schlenk in \cite{CS10} constructed a monotone Lagrangian torus distinct from the Clifford torus, the most standard monotone torus in $\CP^m$ for $m \geq 2$. Later, in \cite{Vianna1, Via16}, Vianna constructed an infinite family of new monotone tori in $\CP^2$, no two of which are related by any symplectomorphism. Also, in \cite{Aurouxinfinitely}, Auroux constructed infinitely many distinct monotone tori with the same monotonicity constant in the Euclidean space $\R^{2m}$ for $m \geq 3$. The following list attempts to include developments of construction exotic tori in different directions, see \cite{Via17, Cas, GalkinMikhalkin, CHW, Bre}. Also, there has been an interesting application of the theory of cluster algebras to construct and distinguish infinitely many Lagrangian fillings, see \cite{CG22, CZ22, CC23}. 

One of our motivations is to produce a general framework to produce {infinitely} many Lagrangian tori in a complex smooth projective variety with a monotone K\"{a}hler form obtained from different limits in the deformation space of $X$ and to find novel classes of Fano varieties admitting infinitely many distinct monotone Lagrangian tori.

\subsection{Main results}

The first half of this paper discusses a sufficient condition for a family of monotone Lagrangian tori constructed from a cluster structure to have \emph{infinitely} many monotone Lagrangian tori, no two of which are related by any symplectomorphism. 

Throughout this paper, all varieties are defined over the field $\C$ of complex numbers. Let $X$ be a smooth projective variety of complex dimension $m$. Consider an oriented rooted regular tree $\mathbb{T}$ whose vertices parametrize the seeds of a cluster algebra with outgoing edges indicating the direction of mutation from each vertex. Let $t_0$ be the root corresponding to the initial seed. Suppose that for each vertex $t \in \mathbb{T}$, there is a valuation $v_t \colon \C(X) \backslash \{0\} \to \Z^m$ with one-dimensional leaves associated with $t$. We denote by $S_t$ (resp. $\Delta_t$) the semigroup (resp. the Newton--Okounkov polytope) constructed from $v_t$. 

The family $\{ \Delta_t \mid t \in \mathbb{T} \}$ is said to have a \emph{tropical cluster structure} if, whenever two vertices $t$ and $t^\prime$ are connected by an oriented edge, the Newton--Okounkov polytopes $\Delta_t$ and $\Delta_{t^\prime}$ are related by a tropicalized cluster mutation in the direction corresponding to the oriented edge. Similarly, the family $\{S_t \mid t \in \mathbb{T} \}$ is said to have a \emph{tropical cluster structure} if, for each pair $(t, t^\prime)$ of vertices connected by an oriented edge, the semigroups $S_t$ and $S_{t^\prime}$ are related by a level-wise tropicalized cluster mutation in the direction corresponding to the connecting oriented edge. Refer to Definition~\ref{definition_tropicalclusterstr} for the precise definition. In the case where each polytope is given by the convex hull of the tropical integer points, the Newton--Okounkov polytopes have a tropical cluster structure if the FG conjecture holds. 

Assume in addition that the Newton--Okounkov polytope $\Delta_{t_0}$ at the initial seed is $\Q$-Gorenstein Fano and the central fiber of the associated toric degeneration is normal. The following proposition claims that the other Newton--Okounkov polytopes also have the same properties.

\begin{propx}[Proposition~\ref{proposition_Qgorennormalind}]\label{propB}
Suppose that the family $\{S_t \mid t \in \mathbb{T} \}$ of finitely generated semigroups has a tropical cluster structure parametrized by a rooted regular tree $\mathbb{T}$ with the initial seed $t_0$.
For $t \in \mathbb{T}$, we denote by $\Delta_t$ the Newton--Okounkov polytope of $X$ corresponding to a semigroup $S_t$. If the following conditions at $t_0$ hold$\colon$
\begin{enumerate}
\item the semigroup $S_{t_0}$ is saturated (and hence the central fiber of the toric degeneration associated with $\Delta_{t_0}$ is a normal toric variety), 
\item the Newton--Okounkov polytope $\Delta_{t_0}$ is  $\mathbb{Q}$-Gorenstein Fano, and
\item the center of $\Delta_{t_0}$ is fixed under the tropicalized cluster mutation of each direction,
\end{enumerate}
then each polytope $\Delta_t$ in the family is also  $\mathbb{Q}$-Gorenstein Fano and the central fiber of the toric degeneration associated with $\Delta_t$ is also a normal toric variety.
\end{propx}

Suppose that the K\"{a}hler form on $X$ inherited from the ambient space of the toric degeneration is monotone. By Propositions~\ref{propA} and~\ref{propB}, there are as many monotone Lagrangian tori in $X$ as seeds in the cluster algebra once the initial Newton--Okounkov body is $\Q$-Gorenstein Fano and the central fiber of the associated toric degeneration is normal. If the cluster algebra is of infinite type, we obtain {infinitely} many Lagrangian tori, some of which could be related by a symplectomorphism. 

The next crucial question, therefore, is how to distinguish the constructed monotone tori. The Newton polytope of invariants counting holomorphic disks or the displacement energy of neighboring Lagrangians has been employed to compare the constructed Lagrangians, see \cite{EliashbergPolterovich, Chekanov} for instance. To compute (or estimate) them, a description of the facets of the Newton--Okounkov polytope is quite crucial. In our case, \emph{however}, each Newton--Okounkov polytope is given by the convex hull of a discrete set of tropical integer points so that we do \emph{not} have a preferred description at hand. In fact, finding such an explicit ``polyhedral" description of a family of polytopes is a challenging problem, see \cite{BerensteinZelevinsky, RietschWilliams} for results on certain classes of examples in this problem. This difficulty motivates us to devise a \emph{practical} criterion that relies solely on the data of tropical integer points along with the exchange matrices (\emph{not} on the data of facets).

\begin{thmx}[Theorem~\ref{theorem_maininthebody}]\label{theoremC}
Let $X$ be a smooth Fano variety and let $\mathcal{L}$ be a very ample line bundle, given by a power of its anticanonical line bundle $K_X^{-1}$. Let us equip $X$ with the K\"{a}hler form inherited from the projective embedding induced by $\mathcal{L}$. Assume that $X$ admits the family $\{ \Delta_t \mid t \in \mathbb{T} \}$ of Newton--Okounkov polytopes of $\mcal{L}$ arising from a family of finitely generated semigroups having a tropical cluster structure parametrized by a rooted regular tree $\mathbb{T}$. Assume that all conditions in Proposition~\ref{propB} hold and the Newton--Okounkov polytope at the initial seed contains the origin. Let $L_{t}$ be the monotone Lagrangian torus constructed from $\Delta_{t}$ by Proposition~\ref{propA}. If there exists a sequence $\left( t_\ell \right)_{\ell \in \mathbb{N}}$ of vertices and a sequence $(r_\ell, s_\ell)_{\ell \in \mathbb{N}}$ of indices such that
\begin{enumerate}
\item the sequence $( \varepsilon_{r_\ell, s_\ell} )_{\ell \in \mathbb{N}}$ of the $(r_\ell, s_\ell)$-entry in the exchange matrix $\varepsilon_{t_\ell}$ associated with $t_\ell$ diverges to $- \infty$ as $\ell \to \infty$ and each $r_\ell$ indexes an unfrozen variable, 
\item both the Newton--Okounkov polytope $\Delta_{t_\ell}$ and the image of $\Delta_{t_\ell}$ under the tropicalized cluster mutation in the $r_\ell$-direction are contained in the half-space
$$
\{ \mathbf{u} \in \R^m \mid u_{s_\ell} \geq 0 \}.
$$
\end{enumerate}
then the family $\left\{ L_{t_\ell} \mid \ell \in \mathbb{N} \right\}$ contains infinitely many monotone Lagrangian tori, no two of which are related by any symplectomorphism.
\end{thmx}

To derive this criterion, we explore relations between various polytopes obtained from the Newton--Okounkov polytopes $\Delta_{t_\ell}$. Suppose that the Newton--Okounkov polytope $\Delta_{t_0}$ is $\Q$-Gorenstein Fano. Assume that the center $\mathbf{u}_0$ is preserved under the tropicalized mutation in each direction as in Proposition~\ref{propB}. It implies that every polytope $\Delta_{t_\ell}$ is a $\Q$-Gorenstein Fano polytope with the same center $\mathbf{u}_0$. Under this circumstance, we study a relation between the $(r_\ell, s_\ell)$-entry $\varepsilon_{r_\ell, s_\ell}$ in the exchange matrix $\varepsilon_{t_\ell}$ and the number of lattice points of the polar dual of $\Delta_{t_\ell} - \mathbf{u}_0$ in the lattice $\frac{1}{q} \Z^m$. Here the integer $q$ is completely determined by the center $\mathbf{u}_0$.

With the relation at hand, to extract a geometric consequence from the polar dual, we introduce a \emph{refined disk potential} of $L_{t_\ell}$. As \emph{not} every counting invariant bounded by $L_{t_\ell}$ is known, we only have partial information on counting invariants in general. The refined disk potential is designed to avoid undesired cancellations, and the Newton polytope of the refined disk potential is also invariant under symplectomorphism modulo unimodular equivalence. As a consequence, (a multiple of) the polar dual $(\Delta_{t_\ell} - \mathbf{u}_0)^\circ$ is contained in the Newton polytope of its refined disk potential of $L_{t_\ell}$, see Proposition \ref{proposition_deltaveedleta}. Therefore, to show that there are infinitely many \emph{distinct} monotone Lagrangian tori, it suffices to show that there are polar duals with an arbitrarily large number of lattice points. The problem turns into searching an exchange matrix with an arbitrarily large entry in the same mutation equivalence class because of the derived relation between the lattice points and the entry of the exchange matrix.

In the second part of the paper, we apply the above criterion to the flag manifolds of \emph{arbitrary} type to show that they have infinitely many distinct monotone Lagrangian tori.

\begin{thmx}[Theorem~\ref{theorem_2ndmain}]\label{theorem_main2}
Let $G$ be a simply connected and semisimple complex algebraic group and $B$ a Borel subgroup. Let $2 \rho$ be the anticanonical regular dominant weight. Then every flag manifold $X = G/B$ equipped with the K\"{a}hler form $\omega_{2 \rho}$ not of type $A_1, A_2, A_3, A_4$, and $B_2 = C_2$, that is, 
\begin{equation}\label{equ_nonfinitecase}
G \neq \mathrm{SL}_2(\C), \mathrm{SL}_3(\C), \mathrm{SL}_4(\C), \mathrm{SL}_5(\C), \mathrm{Spin}_5(\C) = \mathrm{Sp}_4(\C),
\end{equation}
has infinitely many monotone Lagrangian tori, no two of which are related by any symplectomorphism.
\end{thmx}

To construct \emph{infinitely} many monotone Lagrangian tori in flag manifolds, we exploit Newton--Okounkov bodies of Schubert varieties constructed by the theory of cluster algebras by Fujita--Oya \cite{FO20}, which will be briefly recalled.

Let $G$ be a simply connected and semisimple complex algebraic group. The coordinate ring of a unipotent cell $U^-_w$ is isomorphic to an upper cluster algebra generated by unipotent minors, constructed by \cite{BFZ05, Williams13, GLS11}. In \cite{FO20}, Fujita--Oya showed that each seed gives rise to a valuation on the function field $\C(U^-_w)$ and produces a Newton--Okounkov polytope of $X_w$ in $G/B$. We call this polytope a \emph{cluster polytope}. Moreover, they proved that the lattice points of each cluster polytope parametrize the elements of a certain basis on $\C[U^-_w]$. This basis arises from the dual canonical basis of Lusztig in \cite{Lus90} or the upper global basis of Kashiwara in \cite{Kas90} of the quantum group of $G$. By specializing the basis at $q=1$, it gives rise to a basis on the coordinate ring of the unipotent radical and induces a basis on $\C[U^-_w]$. The dual canonical basis has remarkable properties, satisfying the axioms of a triangular basis in Qin \cite{Qin17, Qin20}. For instance, for each choice of seed, we obtain a parametrization of the induced basis given by the extended $g$-vectors. In particular, two parametrizations are related by a finite sequence of tropicalized cluster mutations. It in turn implies that each pair of cluster polytopes is related by a finite sequence of tropicalized cluster mutations. Therefore, the family of cluster polytopes has a tropical cluster structure.

This construction of cluster polytopes is crucial for Theorem~\ref{theorem_main2} as the cluster polytopes enable us to construct \emph{infinitely} many toric degenerations if the cluster algebra is of infinite type, which happens in~\eqref{equ_nonfinitecase}. Recall that the previously known toric degenerations of $G/B$ arise from the Gelfand--Zeitlin polytope or string polytopes, see \cite{GonciuleaLakshmibai, KoganMiller, Caldero} for instance. For a fixed group $G$ and a regular integral dominant weight $\lambda$, there are only \emph{finitely} many string polytopes and hence there are only finitely many toric degenerations of $G/B$ arising from a string polytope. Cluster polytopes are a generalization of string polytopes. The polyhedral description of a ``standard" string polytope $\Delta$ was known by Littelmann \cite{Littelmann}. Using this description, we prove that the polytope $\Delta$ is $\Q$-Gorenstein Fano. Moreover, every cluster polytope is $\Q$-Gorenstein Fano by Proposition~\ref{propB} and hence we have infinitely many monotone Lagrangian tori.

To apply Theorem~\ref{theoremC}, we need to check the two conditions. 
One of the key ingredients for the first condition is a classification of mutation finite skew-symmetrizable matrices with frozen indices, established in \cite{FT21}. Assume that a skew-symmetrizable matrix with one frozen index $s$ is mutation infinite but the submatrix of the unfrozen part is mutation-finite. Then its mutation equivalence class contains a sequence of matrices $(\varepsilon_{t_\ell})_{\ell\in \mathbb N}$ such that $\varepsilon_{r_\ell,s_\ell} \to -\infty$ as $\ell \to \infty$ for a sequence of unfrozen indices $(r_\ell)_{\ell \in \mathbb N}$. Since the mutation of matrices is compatible with the restriction, the problem is reduced to finding an exchange matrix that contains a submatrix with the above property. Using the compatibility of the Dynkin diagram embedding with the seed arising from a reduced expression of the longest element in the Weyl group, we reduce the problem to the five cases $\g= A_5, B_3, C_3, D_4$ and $G_2$. Finally, we provide a case-by-case analysis for those cases. 

To check the second condition, we make use of the correspondence between the lattice points of Newton--Okounkov bodies of a flag manifold $G/B$ and the dual canonical basis elements. In the case of $G/B$, the frozen components of the extended $g$-vectors of dual canonical basis elements are always \emph{non-negative}. A cluster polytope and its mutations in any direction are simultaneously supported by the half-space associated with each frozen variable. This observation determines a lower bound for the lattice points the polar dual of the cluster polytope.

\subsection*{Acknowledgement} 
The authors would like to thank the referee for detailed and invaluable feedback which greatly improved this manuscript. They also thank Cheol-Hyun Cho and Jae-Hoon Kwon for their helpful discussions, comments, and encouragement. Additionally, the authors are grateful to Roger Casals, Eunjeong Lee, and Felix Schlenk for their helpful feedback and comments. The third named author would like to thank Hyun Kyu Kim and Grigory Mikhalkin for helpful discussions.

\section{Newton--Okounkov bodies, toric degenerations, and Lagrangian tori}\label{sec_NObodiesandtoricdeg}

In this section, we briefly review a way of constructing a completely integrable system on a given smooth projective variety via the theory of Newton--Okounkov bodies, following \cite{HaradaKaveh, Anderson}. The outline of the construction is depicted below.

\begin{center}
\vspace{0.2cm}
\begin{tikzpicture}[node distance=3cm]
\node (start) [startstop] {Choice of valuation};
\node (in1) [startstop, below of=start, xshift=1.5cm, yshift=1cm] {Newton--Okounkov body};
\node (in2) [startstop, right of=start, xshift=2.0cm] {Toric degeneration};
\node (in3) [startstop, right of=in1, xshift=3.0cm, yshift=0cm] {Completely integrable system};
\node (in4) [startstop, right of=in2, xshift=2cm, yshift=0cm] {Lagrangian torus};
\draw [arrow] (start) -- (in1);
\draw [arrow] (in1) -- (in2);
\draw [arrow] (in2) -- (in3);
\draw [arrow] (in3) -- (in4);
\end{tikzpicture}
\vspace{0.2cm}
\end{center}

\subsection{Lagrangian tori from toric degenerations}

Recall that a {\em toric degeneration} of a smooth projective variety $X$ is a flat family $\frak{X} \subset \CP^N \times \C$ of projective varieties in $\CP^N$ with the 
following commutative diagram 
\begin{equation}\label{equ_toricdeg}
		    \xymatrix{
			      \frak{X}  ~\ar[drr]_{\pi} \ar@{^{(}->}[rr]^{i}
			      & & \CP^N \times \C \ar[d]^{\text{pr}} \\
			      & & \C
			    }
\end{equation}
such that 
\begin{enumerate}
\item the family is trivial over $\C^*$, $\pi^{-1}(\C^*) \cong X \times \C^*$, and 
\item $\pi^{-1}(0)$ is a (not necessarily normal) projective toric variety. 
\end{enumerate}
We denote by $X_t$ the fiber $\pi^{-1}(t)$ over $t \in \C$. 

\begin{example}
	Let $C = \{ [x : y : z] \in \CP^2 ~|~ y^3 = x^3 + z^3\}$ be an elliptic curve in $\CP^2$. For example, a toric degeneration of $C$ is given by 
	\[
		\frak{X} = \{ ([x : y : z], t) \in \CP^2 \times \C ~|~ y^2z = x^3 + t z^3 \},
	\]
	a flat family of curves such that $X_t \cong C$ for $t \in \C^*$. The variety $X_0$ over the origin is the cuspidal cubic. It is a (non-normal) toric variety because it admits the obvious $\C^*$-action on $X_0$ induced from 
	the $\C^*$-action on $\p^2$ given by 
	\[
		t \cdot [x : y : z] = [x : ty :t^{-2}z] \quad \mbox{for $t \in \C^*$}
	\]
	with two fixed points $[0 : 1 : 0]$ and $[0 : 0 : 1]$.
\end{example}

Let us equip $X$  with a K\"{a}hler form $\omega$ induced from the Fubini--Study form on $\CP^N$. Consider a K\"{a}hler form $\Omega$ on the smooth part $\mathring{\frak{X}}$ of $\frak{X}$ induced from the standard K\"{a}hler form $\omega_{\mathrm{FS}} \oplus \omega_{\mathrm{std}}$ on $\CP^N \times \C$ where $\omega_{\mathrm{FS}}$ 
and $\omega_{\mathrm{std}}$ are the Fubini--Study form on $\CP^N$ and the standard symplectic form on $\C$, respectively. Then each fiber $X_t$ inherits a K\"{a}hler form $\omega_t$ 
from $(\mathring{\frak{X}},\Omega)$, i.e., $\omega_t = \Omega|_{X_t}$. We call $\pi$ a {\em toric degeneration of $(X,\omega)$} if the equipped K\"{a}hler form $\omega$ agrees with the restricted form $\omega_1$. 

The following theorem produces Lagrangian tori in $X$ from a toric degeneration of $(X,\omega)$.

\begin{theorem}[\cite{HaradaKaveh}]\label{theorem_HaradaKaveh} Suppose that $\pi \colon \frak{X} \rightarrow \C$ is a toric degeneration of $(X,\omega)$. 
Then there is a continuous map $\phi \colon X = X_1 \to X_0$ satisfying the following$\colon$
\begin{enumerate}
	\item Let ${U_0}$ be the smooth locus of $X_0$ and $U \coloneqq \phi^{-1}(U_0) \subset X$. 
		   Then the map $\phi$ restricted to $U$ is a symplectomorphism onto $U_0$. 
	\item Let $\Phi_0 \colon X_0 \rightarrow \Delta$ be the moment map for $(X_0,\omega_0)$. Then the composition 
			\begin{equation}\label{equ_commutativitiy}
				\Phi_1 \coloneqq \Phi_0 \circ \phi \colon X \to \Delta
			\end{equation} is a Lagrangian torus fibration on $U$, an open dense subset of $X$.
\end{enumerate}
\end{theorem}

Note that the toric variety $X_0$ carries a completely integrable system $\Phi_0$ on $U_0$ generated by the real torus action induced from the one in $\CP^N$. By pulling this system on $U_0$ back to $U$ via $\phi$, we produce the completely integrable system $\Phi_1$ in \eqref{equ_commutativitiy}. 

\begin{remark}\label{remark_interior_facet}
	If the toric variety $X_0$ in a toric degeneration $\pi$ is normal as an algebraic variety, $U_0$ contains 
	$\Phi_0^{-1}(\Delta \setminus \Delta_2)$ 
	where $\Delta_2$ is the union of codimension two faces of $\Delta$. Consequently, the Lagrangian fibration $\Phi_1 \colon X \rightarrow \Delta$ can be identified with the moment map 
	$\Phi_0 \colon X_0 \rightarrow \Delta$ on the open dense subset $\Phi^{-1}_1(\Delta \setminus \Delta_2) \subset X$ under the symplectomorphism $\phi$. 
\end{remark}

\subsection{Toric degenerations via Newton--Okounkov bodies} \label{subSec: NObody}
 
Let $X$ be a complex $m$-dimensional smooth projective variety with a very ample line bundle $\mcal{L}$ on $X$. Fix a total order $\geq$ on $\Z^m$ respecting the addition. To construct a Newton--Okounkov body, in addition to $\mcal{L}$, we need to choose two data; a valuation $v$ and a reference section $h$. 

Suppose that $v$ is a valuation on the function field, that is, a function $v \colon \C(X) \backslash \{0\} \to \Z^m$ such that for all rational functions $f, g \in \C(X) \backslash \{0\}$ and $c \in \C \backslash \{0\}$,
\begin{enumerate}
\item $v(fg) = v(f) + v(f)$,
\item $v(f+g) \geq \mathrm{min} (v(f), v(g))$, and
\item $v(cf) = v(f)$.
\end{enumerate}

We additionally assume that the valuation $v$ has \emph{one-dimensional leaves}, that is, $v$ satisfies that for all rational functions $f, g \in \C(X) \backslash \{0\} $ with $v(f) = v(g)$, there is $c \in \C \backslash \{0\}$ such that $v(g - cf) > v(g)$ or $g - cf = 0$. Equivalently, setting $\C(x)_\alpha = \{ f \in \C(x) ~|~ v(f) \geq \alpha \hspace{0.1cm} \text{or} \hspace{0.1cm}  f = 0 \}$,
\begin{equation}\label{equ_one_dimensional_leaves}
	 \dim_\C \left(\bigslant{\C(x)_\alpha}{\sum_{\beta; \beta > \alpha}  \C(x)_\beta}\right) \leq 1, \quad \mbox{ for each $\alpha \in \Z^m$.}
\end{equation}
Note that every divisorial valuation has one-dimensional leaves. For instance, a valuation obtained from a flag $X_\bullet = X = X_m \supset X_{m-1} \supset \dots \supset X_1 \supset X_0$ of subvarieties in $X$ has one-dimensional leaves.

Let $\scr{L} \coloneqq H^0(X, \mcal{L})$ be the space of global sections. Then
\[
	R = \bigoplus_{k \geq 0} R_k, \quad R_k = \scr{L}^{\otimes k}, \quad R_0 \coloneqq \C
\]
is the homogeneous coordinating ring of $X$. Choose a non-zero reference section $h \in \mathscr{L}$ and define a semigroup 
\begin{equation}\label{equ_SR}
	S(\mcal{L},v,h) = \bigcup_{k \geq 1} \{ (k, v(f/h^k)) \mid f \in \scr{L}^k \backslash \{0\} \}  \subset \Z_{\geq 0} \times \Z^m \subset \R_{\geq 0} \times \R^m. 
\end{equation}
The \emph{Newton--Okounkov body} associated with the triple $(\mcal{L}, v, h)$ is defined by 
\begin{equation}\label{equ_Deltav}
	\Delta(\mcal{L}, v, h) \coloneqq \overline{\mbox{the convex hull of} \left( \bigcup_{k \in \mathbb{N}} \{ x / k \mid (k,x) \in S(\mcal{L},v,h) \} \right) } \subset \R^m.
\end{equation}
We call $\Delta$ a \emph{Newton--Okounkov body of} $\mcal{L}$ if $\Delta = \Delta(\mcal{L}, v, h)$ for some pair $(v, h)$ of a valuation and a reference section.

One can systematically produce a toric degeneration of $X$ from a Newton--Okounkov body.  

\begin{theorem}[\cite{Anderson}]\label{theorem_Andersontoric}
If the semigroup $S(\mcal{L},v,h)$ in~\eqref{equ_SR} is finitely generated, then the Newton--Okounkov body $\Delta(\mcal{L}, v, h)$ is a rational polytope. 
Furthermore, there is a toric degeneration $\pi \colon \frak{X} \rightarrow \C$ of $X$ where the central fiber $X_0$ is a projective toric variety whose normalization is the toric variety associated with $\Delta(\mcal{L}, v, h)$. 
\end{theorem}

\section{$\mathbb{Q}$-Gorenstein Fano polytopes and Monotone Lagrangian tori}
\label{sec_monoLagtoriQGor}

The aim of this section is to construct a monotone Lagrangian torus in a smooth projective variety $X$ equipped with a monotone K\"{a}hler form when $X$ admits a \emph{normal} toric degeneration arising from a \emph{$\Q$-Gorenstein Fano} Newton--Okounkov polytope. We also define a \emph{refined} version of Newton polytopes from counting invariants of a monotone Lagrangian torus, which will be employed to distinguish the constructed Lagrangian tori later on.

\subsection{Monotone Lagrangians and gradient holomorphic disks}

Let $X$ be a symplectic manifold with a symplectic form $\omega$. The symplectic form $\omega$ is said to be \emph{monotone} if $c_1(TX) = \nu \cdot [\omega] \in H^2(X;\Z)$ for some real number $\nu > 0$. By scaling the symplectic form $\omega$ if necessary, we may assume that $\omega$ is {\em normalized}, that is, 
\begin{equation}\label{equ_integralmonotonesym}
	c_1 (TX) = [\omega] \quad \mbox{ in $H^2(X; \Z)$}.
\end{equation} 
If $X$ is simply connected, then the cohomology classes $[\omega]$ and $c_1(TX)$ can be regarded as a homomorphism assigning the symplectic area and the Chern number to each spherical homotopy class in $\pi_2(X)$, respectively, that is, $[\omega]$ and $c_1(TX) \in \mathrm{Hom}(\pi_2(X), \Z)$. Note that the symplectic form $\omega$ is monotone if $\omega(\alpha) = c_1(TX)(\alpha)$ for all $\alpha \in \pi_2(X)$.

Recall that a submanifold $L$ of $X$ is \emph{Lagrangian} if $\dim L = (\dim X) / 2$ and $\omega|_L \equiv 0$. 
Let $\mathbb{D} = \{ z \in \C \mid |z| \le 1 \}$ be the unit disk. To each continuous map $\varphi \colon \mathbb{D} \to X$ with the Lagrangian boundary condition $\varphi(\partial \mathbb{D}) \subset L$, one can assign an integer $\mu_L([\varphi])$, called the {\em Maslov index}. It can be thought as a Chern number of the disk $\varphi$ in the sense that 
if the image of a continuous map  $\varphi \colon S^2 \rightarrow X$ intersects $L$ so that $\varphi(S^2)$ decomposes into two disks $\varphi^+ \colon \mathbb{D} \rightarrow X$ and $\varphi^- \colon \mathbb{D} \rightarrow X$, then 
the following holds:
\[
	\mu_L([\varphi^+]) + \mu_L([\varphi^-]) = 2 \cdot c_1(TX)([\varphi(S^2)]).
\]
The Maslov index is well-defined up to homotopy and we may think of it as a homomorphism $\mu_L \colon \pi_2(X, L) \rightarrow \Z \subset \R$. Because of the Lagrangian boundary condition, we have a well-defined symplectic area homomorphism $\omega \colon \pi_2(X,L) \rightarrow \R$ defined by $\omega([\varphi]) \coloneqq \int_{\mathbb{D}} \varphi^*\omega$. 

\begin{definition}\label{def_monotoneLag}
A Lagrangian submanifold $L$ of $(X, \omega)$ is called \emph{monotone} if there exists $\delta > 0$ such that 
\begin{equation}\label{equ_propostional}
	\mu_L (\beta) = \delta \cdot \omega (\beta) \quad \mbox{ for all $\beta \in \pi_2 (X, L)$.}
\end{equation}
\end{definition}

Assume that $(X, \omega)$ is not symplectically aspherical, that is, there exists $\alpha \in \pi_2(X)$ such that $\omega(\alpha) \neq 0$. It is known that a monotone Lagrangian can exist only when $\omega$ is monotone. Moreover, if $c_1(TX) = \nu \cdot [\omega]$ and $X$ admits a monotone Lagrangian submanifold $L$ satisfying~\eqref{equ_propostional}, then $\delta = 2 \nu$, see \cite[Remark 2.3]{Ohmonotone}. Thus, if $\omega$ is normalized as in \eqref{equ_integralmonotonesym}, we then have 
\[
	\mu_L (\beta) = 2 \cdot \omega (\beta) \quad \mbox{ for all $\beta \in \pi_2 (X, L)$}.
\]

\begin{example}
	Let $X$ be the complex projective line $\CP^1 \simeq S^2$ with the Fubini--Study form $\omega_{\mathrm{FS}}$. Then every closed Lagrangian submanifold $L$ of $(X, \omega_{\mathrm{FS}})$ is diffeomorphic to a circle. As $L$ divides $X$ into two pieces $\mathbb{D}^+$ and $\mathbb{D}^-$ of disks, we have two holomorphic disks $\varphi^+ \colon \mathbb{D} \rightarrow \mathbb{D}^+$ and $\varphi^- \colon \mathbb{D} \rightarrow \mathbb{D}^-$, each of which has Maslov index two. Therefore $L$ is monotone if and only if the symplectic areas of $\mathbb{D}^+$ and $\mathbb{D}^-$ are equal. 
\end{example}

In \cite{ChoKimMONO}, the first and the third named authors developed a method of computing the Maslov index of a \emph{gradient holomorphic disk}, which will be briefly recalled. Let $(X,\omega)$ be a symplectic manifold with a Hamiltonian $S^1$-action and an $S^1$-invariant $\omega$-compatible almost complex structure $J$ on $X$. Let $p \in X$ and $C$ be the $S^1$-orbit containing $p$. If $C$ flows along the gradient vector field of a moment map with respect to $\omega$ and $J$, then the trajectory of $C$ defines an embedded $J$-holomorphic half-cylinder $Y$ diffeomorphic to $S^1 \times \R_{\geq 0}$. In the case where $C$ converges to some point, a fixed point of the action indeed, then it gives rise to a $J$-holomorphic disk
\[
	\varphi \colon \mathbb{D} \rightarrow X \quad \mbox{satisfying  $\varphi(\mathbb{D} \setminus \{0\}) = Y$}.
\]
Such a disk $\varphi$ is called a {\em gradient holomorphic disk}. Let us recall the following.

\begin{theorem}[Theorem A in \cite{ChoKimMONO}]\label{theorem_gradient_disc_formula}
	Let $(X,\omega)$ be a symplectic manifold with a Hamiltonian $S^1$-action. Let $H \colon X \rightarrow \R$ be a moment map of the $S^1$-action. Suppose that $\varphi \colon \mathbb{D} \rightarrow X$ is a gradient holomorphic disk such that 
		\begin{enumerate}
			\item $\varphi(0)$ is a fixed point of the $S^1$-action, 
			\item $L$ is a Lagrangian submanifold of $(X,\omega)$ lying on some level set $H^{-1}(c)$, and
			\item $\varphi \colon \mathbb{D} \to X$ satisfies the Lagrangian boundary condition, that is, $\varphi(\partial \mathbb{D}) \subset L$. 
		\end{enumerate}
	Then the Maslov index of $[\varphi]$ is equal to $-2n_0$, where $n_0$ is the sum of negative weights of the tangential $S^1$-representation at $\varphi(0)$. 
\end{theorem}

Gradient holomorphic disks are useful to test whether a given torus invariant Lagrangian submanifold (Lagrangian toric fibers, for instance)  is monotone or not. Here is an example.

\begin{example}[Fano toric varieties]\label{example_toricFano}
	Let $X$ be an $m$-dimensional smooth Fano toric variety with a moment map $\Phi \colon X \rightarrow \Delta \subset \frak{t}^* \cong \R^m$ where $\Delta$ is a smooth reflexive polytope that contains a unique interior lattice point $\mathbf{u}_0 \in \Delta$. By the reflexivity of $\Delta$, the variety $X$ carries a K\"{a}hler form $\omega$ invariant under the torus action and satisfying $c_1(TX) = [\omega]$. Let $F_1, F_2, \dots, F_\kappa$ be the facets of $\Delta$. For an interior point $\mathbf{u}$ of $\Delta$, Cho and Oh proved that a Lagrangian fiber $L_{\mathbf{u}} \coloneqq \Phi^{-1}(\mathbf{u})$ is monotone if and only if $\mathbf{u} = \mathbf{u}_0$ in \cite{ChoOh}. They proved that for each
	facet $F_\iota$ and a generic point $p \in L_\mathbf{u}$, there exists a unique holomorphic disk $\varphi_\iota^\prime \colon \mathbb{D} \rightarrow X$ of Maslov index two passing through $p$ such that $\varphi_\iota^\prime(\partial \mathbb{D}) \subset L_{\mathbf{u}}$ and the area of $\varphi^\prime_\iota(\mathbb{D})$ is precisely the affine distance from 
	$\mathbf{u}$ to $F_\iota$. 	Therefore, $L_{\mathbf{u}}$ is monotone if and only if the affine distance from $\mathbf{u}$ to all facets are equal, that is, $\mathbf{u} = \mathbf{u}_0$. 
	
	In fact, $\varphi^\prime_\iota$ can be described as a gradient holomorphic disk obtained as follows. We first take the point $p \in L_{\mathbf{u}}$ in the previous paragraph and 
	denote by $S^1$ the circle subgroup of $T$ whose Lie algebra is the one-dimensional subspace of 
	$\frak{t}^* \cong \R^n$ perpendicular to the facet $F_\iota$. We take the standard complex structure, which is $S^1$-invariant. Let $C$ be the $S^1$-orbit containing $p$. 
	Then 
	the orbit $C$ flows along the gradient vector field of a moment map for the $S^1$-action and it converges to some 
	fixed point whose moment map image is in the relative interior of $F_\iota$. This trajectory gives a gradient holomorphic disk which we denote by $\varphi_\iota \colon \mathbb{D} \rightarrow X$. As the fixed component of the $S^1$-action
	 corresponding to $F_\iota$ is of codimension two, the $S^1$-action near $\pi^{-1}(F_\iota)$ is semifree and hence $\varphi_\iota$
	 is of Maslov index two. By the classification of holomorphic disks of Maslov index two in \cite{ChoOh}, we conclude that $\varphi^\prime_\iota  = \varphi_\iota$.  
\end{example}

\subsection{$\Q$-Gorenstein Fano and normal toric degenerations}

Let $X$ be a smooth projective variety of complex dimension $m$. Let $v$ be a valuation on the function field $\C(X)$ with one-dimensional leaves (see \eqref{equ_one_dimensional_leaves}), $\mcal{L}$ a very ample line bundle on $X$, and $h$ a non-zero reference section of $\mcal{L}$. From the choice $(\mcal{L}, v, h)$, we produce a triple; a semigroup $S(\mcal{L},v,h)$, a Newton--Okounkov body $\Delta(\mcal{L}, v, h)$, and a toric degeneration $\pi \colon \frak{X} \to \C$ as in Section \ref{subSec: NObody}. To produce a \emph{monotone} Lagrangian torus from the toric degeneration $\pi$ of $X$, we require the toric degeneration $\pi$ to be \emph{normal} and the Newton--Okounkov body $\Delta(\mcal{L}, v, h)$ to be \emph{$\Q$-Gorenstein Fano}, which will be defined in this section.

To begin with, we review some notions of polyhedral geometry. Let $N$ be a lattice of rank $m$ and $M \coloneqq \mathrm{Hom}_\Z (N, \Z)$ the dual lattice of $N$. Let us fix a basis $\seed$ for the free abelian group $N$. Using the fixed basis $\seed$, $N$ can be identified with $\mathbb{Z}^m$. The dual lattice $M$ can be identified with $\mathbb{Z}^m$ via the dual basis of $\seed$. We also have the identifications $N_\R \coloneqq N \otimes \R \simeq \R^m$ and $M_\R \coloneqq M \otimes \R \simeq \R^m$. We take the Cartesian coordinate system $\mathbf{v} \coloneqq (v_1, v_2, \cdots, v_m)$ (resp. $\mathbf{u} \coloneqq (u_1, u_2, \cdots, u_m)$) for $N_\R \simeq \R^m$ (resp. $M_\R \simeq \R^m$). Let $\langle \cdot, \cdot \rangle \colon M_\R \times N_\R \to \R$ be the canonical pairing. For a non-zero vector $\mathbf{v} \in N_\R$ and a real number $\alpha \in \R$,
\begin{itemize} 
\item the \emph{hyperplane} $H_{\mathbf{v}, \alpha}$ is defined by $H_{\mathbf{v}, \alpha} \coloneqq \{ \mathbf{u} \in M_\R \mid \langle \mathbf{u}, \mathbf{v} \rangle + \alpha = 0 \}$ and
\item the \emph{(closed) half-space} $H^+_{\mathbf{v}, \alpha}$ is defined by $H^+_{\mathbf{v}, \alpha} \coloneqq \{ \mathbf{u} \in M_\R \mid \langle \mathbf{u}, \mathbf{v} \rangle + \alpha \geq 0 \}.$  
\end{itemize}

Suppose that $\Delta \subset M_\R$ is a full-dimensional polytope that contains the origin in its interior, that is, $\mathbf{0} \in \mathrm{Int}(\Delta)$. Then the polytope $\Delta$ in $M_\R$ can be uniquely expressed as an intersection of half-spaces
\begin{equation}\label{equ_polytopeshalfsp}
\Delta = \bigcap_{\iota=1}^{\kappa}  H^+_{\mathbf{v}_\iota, 1}
\end{equation}
satisfying the following conditions;
\begin{enumerate}
\item each vector $\mathbf{v}_\iota \in N_\R$, 
\item each hyperplane $H_{\mathbf{v}_\iota, \alpha}$ contains a facet of $\Delta$, and
\item $\{ \mathbf{v}_1, \mathbf{v}_2, \cdots, \mathbf{v}_\kappa \}$ is pairwise distinct.  
\end{enumerate}
In the presentation of $\Delta$ in~\eqref{equ_polytopeshalfsp}, note that each vector $\mathbf{v}_\iota$ is an \emph{inward} normal vector to a facet of $\Delta$ and the number of facets is equal to $\kappa$. 

Let $C$ be a subset of $M_\R$. For a non-zero vector $\mathbf{v} \in N_\R$ and a real number $\alpha \in \R$, the half-space $H^+_{\mathbf{v}, \alpha}$ is called a \emph{supporting half-space} of $C$ if $C \subset H^+_{\mathbf{v}, \alpha}$ and $C \cap H_{\mathbf{v}, \alpha} \neq \emptyset$. In this case, we call $H_{\mathbf{v}, \alpha}$ a \emph{supporting hyperplane} of $C$. Note that if $C$ is a polytope $\Delta$, then a hyperplane containing a lower dimensional stratum can be a supporting hyperplane. We emphasize that a supporting hyperplane need \emph{not} contain a facet of $\Delta$.  

We recall the ``polar dual" of a polytope $\Delta$.

\begin{definition} Let $\Delta$ be a full-dimensional polytope in $M_\R$ containing the origin in its interior. The \emph{polar dual of $\Delta$} is defined by 
\[
\Delta^\circ \coloneqq \left\{ \mathbf{v} \in N_\R \mid \langle \mathbf{u}, \mathbf{v} \rangle  + 1 \geq 0 \mbox{ for all $\mathbf{u} \in \Delta$} \right\}.
\]
\end{definition}

Indeed, $\Delta^\circ$ is also a polytope containing the origin in its interior. Moreover, it satisfies $(\Delta^\circ)^\circ = \Delta$, see \cite[Theorem 6.2]{Brn} for the proof. The following proposition provides an alternative way of presenting the polar dual of $\Delta$.

\begin{proposition}\label{proposition_polardualandnewton}
Suppose that $\Delta \subset M_\R \simeq \R^m$ is a full-dimensional polytope containing the origin $\mathbf{0}$ in the interior of $\Delta$. Then the polar dual $\Delta^\circ$ is equal to the convex hull of $\{ \mathbf{v}_\iota \mid \iota = 1,2, \cdots, \kappa \}$ in $N_\R$ where the vectors $\mathbf{v}_\iota$ are from the unique presentation~\eqref{equ_polytopeshalfsp} of $\Delta$.
\end{proposition}

\begin{proof}
Let $\Delta^\prime$ be the convex hull of $\{ \mathbf{v}_\iota \mid \iota = 1,2, \cdots, \kappa \}$. Since $\langle \mathbf{u}, \mathbf{v}_\iota \rangle \geq -1$ for all $\mathbf{u} \in \Delta$, the polar dual $\Delta^\circ$ contains all $\mathbf{v}_\iota$'s. Since $\Delta^\circ$ is convex, we have $\Delta^\prime \subset \Delta^\circ$. On the other hand, if $\mathbf{u} \in (\Delta^\prime)^\circ$, then $\langle \mathbf{u}, \mathbf{v} \rangle + 1 \geq 0$ for all $\mathbf{v} \in \Delta^\prime$. In particular, $\langle \mathbf{u}, \mathbf{v}_\iota \rangle + 1 \geq 0$ for all $\iota = 1, 2, \cdots, \kappa$. Hence  $ (\Delta^\prime)^\circ \subset \bigcap_{\iota=1}^{\kappa}  H^+_{\mathbf{v}_\iota, 1} =\Delta $
 so that we have
$$
\Delta^\circ \subset ((\Delta^\prime)^\circ)^\circ.
$$
Since $\Delta^\prime$ is a closed and convex set containing $\mathbf{0}$ in its interior, $\Delta^\circ \subset \Delta^\prime$ follows from $((\Delta^\prime)^\circ)^\circ = \Delta^\prime$.
\end{proof}

As we are interested in Newton--Okounkov bodies constructed from Theorem~\ref{theorem_Andersontoric}, we are only concerned with rational polytopes. From now on, every polytope $\Delta$ is assumed to be \emph{rational}, which means that all vertices of $\Delta$ have rational coordinates.

\begin{definition}\label{def_qgorensteinpoly}
A full-dimensional rational polytope $\Delta \subset M_\R \simeq \R^m$ is called \emph{$\mathbb{Q}$-Gorenstein Fano} if there exists a vector $\mathbf{u}_0 \in M_\R$ and a number $\nu \in \mathbb{N}$ such that $\nu \cdot (\Delta - {\bf u}_0)^\circ$ is a \emph{Fano} polytope, a lattice polytope each of which vertex is a \emph{primitive lattice} vector, see \cite{KasprzykNill}. Equivalently, the translated polytope $\Delta - \mathbf{u}_0$ has a presentation of the form~\eqref{equ_polytopeshalfsp} satisfying that the $\nu$-multiple of each inward facet normal vector $\mathbf{v}_\iota$ is primitive. To keep track of this number $\nu$, we sometimes call the polytope $\Delta$ a \emph{$\mathbb{Q}$-Gorenstein Fano polytope of size $\nu$}.

In this circumstance, the point $\mathbf{u}_0$ which maps to the origin $\mathbf{0}$ under the translation $\mathbf{u} \mapsto \mathbf{u} - \mathbf{u}_0$ is called the \emph{center} of $\Delta$. A $\mathbb{Q}$-Gorenstein Fano polytope of size $1$ is said to be a \emph{normalized} $\mathbb{Q}$-Gorenstein Fano polytope.
\end{definition}

For example, the convex hull $\Delta_1 $ of $\{(\pm 2, \pm 2) \}$ in $\R^2$ is a $\Q$-Gorenstein Fano polytope of size $2$, whereas the convex hull $\Delta_2$ of $\{(\pm 1, \pm 2) \}$ in $\R^2$ is not $\Q$-Gorenstein Fano, see Figure~\ref{exqgorpoly}. 
 \begin{figure}[H]
		\scalebox{0.8}{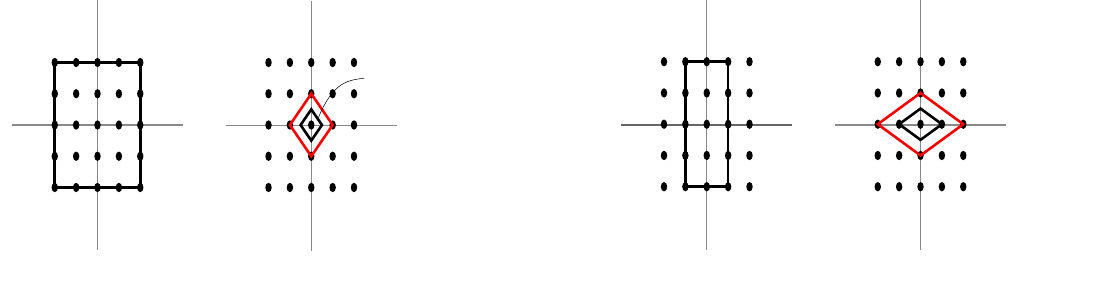}
		\caption{$\Q$-Gorenstein Fano polytope, non $\Q$-Gorenstein Fano polytope}\label{exqgorpoly}
\end{figure}

\begin{definition}\label{def_combinatorialnewtonpoly}
For a $\Q$-Gorenstein Fano polytope $\Delta$, the Fano polytope $\nu \cdot (\Delta - {\bf u}_0)^\circ \subset N_\R$ in Definition~\ref{def_qgorensteinpoly} is called the \emph{combinatorial dual} of $\Delta$ and denoted by $\Delta^\vee$. In other words, $\Delta^\vee$ is the convex hull of all the primitive generators of $1$-cones of the normal fan of $\Delta$.
\end{definition}

\begin{remark}\label{remark_QGoren} The terminology ``$\mathbb{Q}$-Gorenstein Fano" originated from the following geometric fact. The toric variety associated with the normal fan of a $\mathbb{Q}$-Gorenstein Fano polytope is $\mathbb{Q}$-Gorenstein Fano as an algebraic variety, see \cite[Corollary 2.2.19, Theorem 4.2.8, and Proposition 6.1.10]{CLS11} for instance. Note that the $\mathbb{Q}$-Gorenstein Fano condition is more general than the ``reflexive" condition. Namely, every reflexive polytope is normalized $\mathbb{Q}$-Gorenstein Fano. In particular, a $\mathbb{Q}$-Gorenstein Fano polytope need \emph{not} be a lattice polytope.
\end{remark}

\begin{remark}\label{lemma_QGorensteinpoly}
When $\Delta$ is a normalized $\mathbb{Q}$-Gorenstein Fano polytope with center $\mathbf{u}_0$, we can compute the polar dual $(\Delta - \mathbf{u}_0)^\circ$ using the facet normal vectors of $\Delta$ by Proposition \ref{proposition_polardualandnewton} as follows. Let $\{ F_1, F_2, \cdots, F_\kappa \}$ be the set of facets of $\Delta$ and $\mathbf{v}_{F_\iota}$ the primitive inward normal vector to the facet $F_\iota$. Then $\Delta^\vee = (\Delta - \mathbf{u}_0)^\circ$ is expressed as 
\[
    (\Delta - \mathbf{u}_0)^\circ = \mbox{the convex hull of $\left\{ \mathbf{v}_{F_\iota} \mid \iota = 1, 2, \cdots, \kappa  \right\}$}.
\]

If $\Delta$ is a $\mathbb{Q}$-Gorenstein Fano polytope of size $\nu$ with center $\mathbf{u}_0$, by Proposition~\ref{proposition_polardualandnewton}, the combinatorial dual $\Delta^\vee$ of $\Delta$ is equal to
$$
\mbox{the convex hull of $\left\{ \nu \cdot \mathbf{v}_{\iota} \mid \iota = 1, 2, \cdots, \kappa  \right\}$} = 
\mbox{the convex hull of $\left\{ \mathbf{v}_{F_\iota} \mid \iota = 1, 2, \cdots, \kappa  \right\}$}.
$$
\end{remark}

Let $X$ be a smooth projective variety polarized by a very ample line bundle $\mcal{L}$ with $N \coloneqq \dim_\C H^0(X,\mcal{L}) - 1$. Via the Kodaira embedding $\varphi_\mcal{L} \colon X \to \CP^N$, the variety $X$ can be regarded as a subvariety of the projective space $\CP^N$. Recall that we have chosen a valuation $v \colon \C(X) \backslash \{0\} \rightarrow \C$ with one-dimensional leaves and a nonzero section $h \in H^0(X,\mcal{L})$. We then have a semigroup $S(\mcal{L},v,h)$ associated with the triple $(\mcal{L},v,h)$.

\begin{assumption}\label{ass_semigpfg}
Throughout this section, the associated semigroup $S(\mcal{L},v,h)$ is always assumed to be finitely generated. 
\end{assumption}

By Theorem~\ref{theorem_Andersontoric}, there exists a toric degeneration $\pi \colon \frak{X} \rightarrow \C$ of $X$ associated with the Newton--Okounkov polytope $\Delta(\mathcal{L}, v, h)$ such that the diagram~\eqref{equ_toricdeg} commutes. Let us choose the K\"{a}hler form $\omega$ on $X$ induced from the Fubini--Study form on the ambient projective space, that is, $\omega = \varphi_{\mcal{L}}^*  \omega_\mathrm{FS}$. Such a form $\omega$ is said to be \emph{inherited from} the very ample line bundle $\mcal{L}$. When considering a Newton--Okounkov polytope of $\mcal{L}$, we equip $X$ with the K\"{a}hler form inherited from ${\mcal{L}}$.

\begin{definition}\label{def_qgorensteinnormaltoricdeg}
    Let $X$ be a smooth projective variety with a toric degeneration $\pi \colon \frak{X} \rightarrow \C$ arising from a Newton--Okounkov polytope $\Delta(\mcal{L},v,h)$.
    \begin{itemize}
        \item We say that $\pi$ is \emph{$\Q$-Gorenstein Fano} if the Newton--Okounkov polytope $\Delta(\mcal{L},v,h)$ is $\Q$-Gorenstein Fano in the sense of Definition~\ref{def_qgorensteinpoly}. (Then the central fiber is $\mathbb{Q}$-Gorenstein Fano as an algebraic variety, see Remark~\ref{remark_QGoren}) 
        \item We say that $\pi$ is {\em normal} if the central fiber of $\pi$ is {normal} as an algebraic variety. Note that the degeneration $\pi$ is normal if $S(\mcal{L},v,h)$ is saturated, see \cite[Theorem 1.3.5, Exercise 2.1.5]{CLS11}.
    \end{itemize}
\end{definition}

We now construct a monotone Lagrangian torus of a smooth Fano variety equipped with a K\"{a}hler form induced from a very ample line bundle $\mcal{L}$, which is defined as a positive power of the anticanonical bundle of $X$.

\begin{proposition}[Proposition~\ref{propA}]\label{proposition_toricdegenerationsmonotone}
 Let $X$ be a smooth Fano variety whose anticanonical line bundle $K^{-1}_X$ is ample. Take a very ample line bundle $\mathcal{L} = (K^{-1}_X)^{\otimes \nu}$ for some $\nu \in \mathbb{N}$ and equip $X$ with the K\"{a}hler form inherited from $\mathcal{L}$.
Suppose that the Newton--Okounkov polytope $\Delta = \Delta(\mathcal{L}, v, h)$ of $\mathcal{L}$ with a choice of $(v, h)$ in Section~\ref{subSec: NObody} is a $\Q$-Gorenstein Fano polytope and the central fiber of the associated toric degeneration is a normal toric variety. Let $\Phi \colon X \to \Delta$ be a completely integrable system constructed from Theorem~\ref{theorem_HaradaKaveh}. Then this $\Q$-Gorenstein Fano polytope $\Delta$ is of size $\nu$ and the fiber $L \coloneqq \Phi^{-1}(\mathbf{u}_0)$ of $\Phi$ located at the center $\mathbf{u}_0$ of $\Delta$ is a monotone Lagrangian torus.
\end{proposition}

\begin{proof}
Observe that the K\"{a}hler form $\omega$ induced from a projective embedding into $\p(H^0(X,\mcal{L})^*)$ is monotone. More precisely, if the very ample line bundle $\mathcal{L}$ is a positive power of the anticanonical bundle $K^{-1}_X$ of $X$, i.e., $\mathcal{L} = (K^{-1}_X)^{\otimes \nu}$ for some $\nu \in \mathbb{N}$, then we have 
\begin{equation}\label{c1omegaeq}
    [\omega] = c_1 (\mathcal{L}) = \nu \cdot c_1(TX).
\end{equation}
Since every smooth Fano variety is simply connected (see \cite{Kob61} for instance), by the long exact sequence of the pair $(X, L)$, we have
\begin{equation}\label{equ_homotopygroup1}
\pi_2(X, L) \simeq \pi_2(X) \oplus \pi_1(L).
\end{equation}
Regarding $\pi_2(X)$ as $H_2(X)$ via the Hurewicz map, \eqref{c1omegaeq} yields
\begin{equation}\label{equ_1spherical}
\nu \cdot \mu_L (\alpha) = 2 \cdot \omega (\alpha) \quad \mbox{ for every spherical class $\alpha \in \pi_2(X)$.}
\end{equation}

Suppose that $\Delta$ is a $\Q$-Gorenstein Fano polytope of size $\nu^\prime$ for some $\nu^\prime \in \mathbb{N}$. We claim that $\nu = \nu^\prime$. Let $\mathbf{u}_0$ be the center of $\Delta$ and consider the Lagrangian torus fiber $L_0 \coloneqq \Phi_0^{-1}(\mathbf{u}_0)$ in the toric variety $X_0$. Since ${X}_{0}$ is normal, the singular locus of $X_0$ has complex codimension $\geq 2$. In particular, the smooth locus $U_0$ of $X_0$ contains the inverse image $\Phi_0^{-1}(\mathrm{int}(F))$ for each facet $F$ of $\Delta$. Let ${\bf v}_F \in N$ be the primitive integral vector inward normal to $F$ and $S^1_F$ the circle subgroup generated by ${\bf v}_F$. As in Example~\ref{example_toricFano}, there is a gradient holomorphic disk of Maslov index two
\[
	u_F \colon \mathbb{D} \rightarrow X
\]
such that $u_F(\partial \mathbb{D}) \subset L_0$ and $u_F(\mathbb{D}) \subset U_0$. 

Let $U_1 \coloneqq \phi^{-1} (U_0) \subset X_1 = X$ be the inverse image of the smooth locus $U_0$ of $X_0$ via the map $\phi$ in Theorem \ref{theorem_HaradaKaveh}. The map $\phi$ restricted to $U_1$ is a symplectomorphism from $U_1$ to $U_0$. Moreover, $L$ maps to $L_0$ via $\phi$. For each facet $F_\iota$ with $\iota =1,2,\cdots, \kappa$, 
we denote by $u_{F_\iota} \colon \mathbb{D} \rightarrow X_0$ a gradient holomorphic disk of Maslov index two bounded by $L_0$ and let $\varphi_\iota \coloneqq u_{F_\iota} \circ \phi$ be the pull-back of $u_{F_\iota}$ along $\phi \colon U_1 \rightarrow U_0$. Since $\phi$ is a symplectomorphism, $\varphi_\iota$ is of Maslov index two and its symplectic area is the same as that of $u_\iota$. Recall that the area of $u_\iota$ is equal to $\nu^\prime$, the affine distance between the center and the facet $F_\iota$, see \cite[Theorem 8.1]{ChoOh} for the detail. Therefore, we have 
\begin{equation}\label{equ_1disk}
	\nu^\prime \cdot \mu_{L}([\varphi_\iota]) = 2 \cdot \omega([\varphi_\iota]).
\end{equation}

Since $\Delta \subset M_\R$ is full dimensional, the set of ${\bf v}_F$'s positively span $N_\R$ and so there exist $\iota_1, \iota_2, \cdots, \iota_s \in \mathbb{N}$ with $1 \leq \iota_1 < \iota_2 < \cdots < \iota_s \leq \kappa$ and $a_{\iota_1}, a_{\iota_2}, \cdots, a_{\iota_s} \in \mathbb{N}$ such that 
$$
\partial \left( a_{\iota_1} [\varphi_{\iota_1}] + a_{\iota_2} [\varphi_{\iota_2}] + \cdots +
a_{\iota_s} [\varphi_{\iota_s}] \right) = 0 \quad \mbox{ in $\pi_1(L)$.}
$$
Then $\alpha \coloneqq a_{\iota_1} [\varphi_{\iota_1}] + a_{\iota_2} [\varphi_{\iota_2}] + \cdots +
a_{\iota_s} [\varphi_{\iota_s}]$ is a nontrivial class in $\pi_2(X)$ by the the long exact sequence of pair $(X,L)$. From \eqref{equ_1spherical} and~\eqref{equ_1disk}, it follows that $\nu = \nu^\prime$. 

We may assume that the first $m$ normal vectors $\{\mathbf{v}_{1}, \mathbf{v}_{2}, \cdots, \mathbf{v}_{m}\} \subset N$ in the expression~\eqref{equ_polytopeshalfsp} are linearly independent and they form a basis of $N_\Q \simeq \Q^m$ after renumbering the indices if necessary. 
Note that $\{[\varphi_\iota] \mid \iota=1, 2, \cdots, m\}$ is a $\Q$-basis of $\Q\otimes_\Z \pi_1(L)$ and $\nu \cdot \mu_{L}([\varphi_\iota]) = 2 \cdot \omega([\varphi_\iota])
$ for each $\iota = 1, 2, \cdots, m$ by~\eqref{equ_1disk}. 
Combining~\eqref{equ_homotopygroup1} and ~\eqref{equ_1spherical}, together with the fact that $\mu_{L}$ and $\omega$ are homomorphisms on $\pi_2(X,L)$, we may conclude that
\[
	\nu \cdot \mu_L (\beta) = 2 \cdot \omega (\beta)
\]
for all $\beta \in \pi_2(X,L)$. Hence, $L$ is monotone. 
\end{proof}

\begin{remark}\label{remark_GM22}
In \cite{GalkinMikhalkin}, Galkin--Mikhalkin proved an analogous statement. They showed the monotonicity of a Lagrangian torus obtained from symplectic parallel transport of the toric fiber at the center $\mathbf{u}_0$ to a generic fiber.
\end{remark}

\subsection{Newton polytopes from disk counting}

The goal of this subsection is to introduce two dual polytopes obtained by counting disks bounded by a monotone Lagrangian torus. One has been widely used to distinguish monotone Lagrangian tori in the previous literature and the other is its slight modification thereof.

To each monotone Lagrangian torus $L$ in a $2m$-dimensional monotone symplectic manifold $(X,\omega)$, one can assign a \emph{disk potential}
\[
	W_L \in \C[z_1^{\pm}, \dots, z_m^{\pm}],
\] 
a Laurent polynomial which encodes counting invariants (a.k.a open Gromov-Witten invariants) of $L$. We briefly review the construction of the disk potential $W_L$ of $L$.

Choose an $\omega$-compatible almost complex structure $J$. For each homotopy class $\beta \in \pi_2(X, L)$ of Maslov index two, 
we denote by $\mcal{M}_1(L, J; \beta)$ the moduli space of $J$-holomorphic disks with one boundary marked point $(\varphi,z_0)$ in the class $\beta$ (modulo the automorphisms), that is, 
\[
\mcal{M}_1(L, J; \beta) = \left\{ (\varphi, z_0) \mid \varphi \colon (\mathbb{D}, \partial \mathbb{D}) \to (X, L), z_0 \in \partial \mathbb{D}, J \circ d \varphi = d \varphi \circ j  \mbox{ and } [\varphi] = \beta \right\} / \sim
\] 
where $j$ is the standard complex structure on $\mathbb{D}$.
The expected dimension of $\mcal{M}_1(L, J; \beta)$ is $\dim_\R L + \# \{\mbox{marked points}\} + \mu_L(\beta) - \dim_\R \mathrm{PSL}(2, \R)$.
In this monotone case, by choosing a generic almost complex structure $J$, the moduli space becomes transversal when  $\mu_L(\beta) = 2$. Moreover, the monotonicity condition ensures that the moduli space $\mcal{M}_1(L, J; \beta)$ with $\mu_L(\beta) = 2$ is a compact smooth manifold \emph{without} boundary. Taking an orientation and relative spin structure on the Lagrangian $L$, the degree of the evaluation map 
\[
	\mathrm{ev} \colon \mcal{M}_1(L, J; \beta) \to L, \quad (\varphi, z_0) \mapsto \varphi (z_0)
\]
is defined and denoted by $n(L, J; \beta)$. In other words, the number $n(L, J; \beta)$ is the (signed) count of $J$-holomorphic disks passing through a generic point of $L$. 

In general, for a general Lagrangian submanifold $L$ (with a fixed choice of orientation and relative spin structure), the number $n(L, J; \beta)$ \emph{does} depend on the choice of $J$. When $L$ is monotone, one can construct an oriented cobordism between two moduli spaces (with different choices of a generic almost complex structure) and hence the number $n(L, J; \beta)$ becomes a deformation invariant as stated below.

\begin{proposition}[\cite{EliashbergPolterovich}] \label{lemma_invarianceofcounting}
Suppose that $L$ is a monotone Lagrangian submanifold of $(X,\omega)$. For any class $\beta \in \pi_2(X,L)$ of Maslov index two and generic compatible almost complex structures $J$ and $J^\prime$, the counting invariants $n(L,J; \beta)$ and $n(L,J^\prime; \beta)$ are equal. 
\end{proposition}

From now on, for simplicity, $n(L,J; \beta)$ is denoted by $n_\beta$. To define a disk potential of $L$, we make a choice on a tuple $\mcal{B} = (\vartheta_1, \vartheta_2, \cdots,\vartheta_m)$ of oriented loops that form a $\Z$-basis of $\pi_1(L)$. The choice leads to the identification $\pi_1(L)$ with $N \coloneqq  \Z^m$. Namely, the basis elements in $\mcal{B}$ then correspond to the standard lattice vectors in $N$. Let $\frak{L}$ be a trivial complex line bundle over $L$ and consider a flat $\C^*$-connection $\nabla$ on $\frak{L}$. We take the monomial $z_i \coloneqq \mathrm{hol}_\nabla (\vartheta_i ) \in (\C^*)^m$ for $i = 1, 2, \cdots, m$. The set of those monomials forms a coordinate system for the moduli space of flat $\C^*$-connections on $\frak{L}$ (modulo gauge equivalences).

For a lattice vector $\mathbf{v} \in N$, we set $\mathbf{z}^{\mathbf{v}} = \mathbf{z}^{(v_1, \cdots, v_m)}$ to be the monomial $z_1^{v_1} \cdots z_m^{v_m}$. The \emph{disk potential} of $L$ in \cite{ChoOh, FOOOToric1} is defined by
\begin{equation}\label{equ_diskpotentialfun}
W_{L,\mcal{B}} \colon (\C^*)^m \to \C, \quad  W_{L,\mcal{B}} (\mathbf{z}) = \sum_{\beta \in \pi_2(X,L)} n_\beta \cdot \mathbf{z}^{\partial \beta}
\end{equation}
where the sum is taken over all homotopy classes of Maslov index two. 
We define the first combinatorial object associated to the disk potential $W_{L, \mcal{B}}$. We often suppress $\mcal{B}$ in the notation $W_{L, \mcal{B}}$, that is, $W_{L} \coloneqq W_{L, \mcal{B}}$.

\begin{definition}
We denote by $\Delta^{\mathrm{disk}}_{L, \mcal{B}}$ the Newton polytope of the disk potential  $W_{L, \mcal{B}}$ of a monotone Lagrangian torus $L$, that is,
$$
\Delta^{\mathrm{disk}}_{L, \mcal{B}} \coloneqq \mbox{the convex hull of $\left\{ \mathbf{v} \mid \sum_{\beta; \partial \beta = \mathbf{v}} n_\beta \neq 0 \right\}$ in $N_\R$.} 
$$
\end{definition}

\begin{remark}
By the Gromov's compactness theorem, the number of $J$-holomorphic disks of Maslov index two is finite. Hence, the convex hull $\Delta^{\mathrm{disk}}_{L, \mcal{B}}$ is compact (and hence a polytope). 
\end{remark}

\begin{remark}
In general, it is hard to find an explicit expression of $W_{L}$. It is known that the disk potential can be immediately written from the facets of $\Delta$ if the singular loci of $X_0$ are of particular type, see \cite{NishinouNoharaUeda, BGM2022}.
\end{remark}

Suppose that $\mathcal{B}$ and $\mathcal{B}^\prime$ are two $\Z$-bases for $\pi_1(L) \simeq N$. Then there exists an invertible matrix $A \in \mathrm{GL}(m, \Z)$ such that $\mathcal{B}$ maps to $\mathcal{B}^\prime$ under the linear transformation $\mathbf{v} \mapsto A \mathbf{v}$. Two disk potentials $W_{L, \mathcal{B}}$ and $W_{L, \mathcal{B}^\prime}$ are simply related by the monomial coordinate change defined by $\mathbf{z}^{\mathbf{v}} \mapsto \mathbf{z}^{A \mathbf{v}}$. In particular, the Newton polytopes $\Delta^{\mathrm{disk}}_{L, \mcal{B}}$ and $\Delta^{\mathrm{disk}}_{L, \mcal{B}^\prime}$ are unimodularly equivalent.

\begin{definition}
Two polytopes $\Delta$ and $\Delta^\prime$ are called \emph{unimodularly equivalent} if  $\Delta$ maps to $\Delta^\prime$ under a unimodular equivalence, that is, an affine transformation $\mathbf{v} \mapsto A \mathbf{v} + \mathbf{v}_0$ for some $A \in \mathrm{GL}(m, \Z)$ and $\mathbf{v}_0 \in \R^m$.
\end{definition}

From now on, two Newton polytopes are regarded as the same if they are unimodularly equivalent. We simply suppress $\mathcal{B}$ in the notation $\Delta^{\mathrm{disk}}_{L, \mcal{B}}$, that is, $\Delta^{\mathrm{disk}}_{L} \coloneqq \Delta^{\mathrm{disk}}_{L, \mcal{B}}$ for simplicity. 

Suppose that $\phi \colon X \to X$ is a symplectomorphism such that $\phi(L_1) = L_2$. We take a $\Z$-basis $\phi_* \mathcal{B} = (\phi \circ \vartheta_1, \phi \circ \vartheta_2, \cdots, \phi \circ \vartheta_m)$ consisting of oriented loops for $\pi_1(L_2)$ and set $z_i^\prime \coloneqq \mathrm{hol}_\nabla (\phi \circ \vartheta_i )$. By Proposition~\ref{lemma_invarianceofcounting}, the two counting invariants $n(L_1, J_1; \beta_1)$ and $n(L_2, \phi_* J_1; \phi_* \beta_1)$ are equal and hence $W_{L_1, \mathcal{B}}$ agrees with $W_{L_2, \phi_* \mathcal{B}}$ under the transformation $z_i \mapsto z_i^\prime$ for $i = 1, 2, \cdots, m$. We then have the following corollary.

\begin{corollary}\label{corollary_newtonpolytopeuni}
Consider two monotone Lagrangian tori $L_1$ and $L_2$ in a symplectic manifold $(X, \omega)$. If there is a symplectomorphism $\phi \colon X \to X$ such that $\phi (L_1) = L_2$, then the Newton polytopes $\Delta^{\mathrm{disk}}_{L_1}$ and $\Delta^{\mathrm{disk}}_{L_2}$ of disk potentials are unimodularly equivalent. 
\end{corollary} 

\begin{remark}
This corollary has been employed to distinguish the monotone Lagrangian tori, see \cite{Aurouxinfinitely, Via16, Via17} for instance.
\end{remark}

We introduce the Newton polytope $\Delta^{\mathrm{ref}}_L$ of the ``refined" version of the disk potential $W_L$ in the sense of effective disk classes. For motivation, we now discuss some drawbacks of this disk potential $W_L$ in~\eqref{equ_diskpotentialfun} to distinguish monotone Lagrangian tori. Originally, the disk potential function $W_L$ in \cite{FOOObook} was defined on the Maurer--Cartan space of $L$. To make the expression malleable, it is usually written in the form of a Laurent polynomial or series by restricting the Maurer--Cartan space to $H^1(L)$ (provided $H^1(L)$ is a subset of the Maurer--Cartan space) and taking a basis $\mcal{B} = \{ \vartheta_1, \vartheta_2, \cdots,\vartheta_m \}$ for $\pi_1(L) = \mathrm{Hom}(H^1(L);\mathbb{Z})$. However, throughout this process, we lose some information. Suppose that there are two distinct classes $\beta$ and $\beta^\prime$ that can be realized as a holomorphic disk satisfying 
\begin{enumerate}
\item $\partial(\beta) = \partial(\beta^\prime)$ and
\item $\mu_L (\beta) = \mu_L (\beta^\prime) = 2$.
\end{enumerate}
If their counting invariants are canceled, that is, $n_\beta + n_{\beta^\prime} = 0$, then the disk potential $W_{L}$ cannot capture the existence of such disks. To capture this piece of information on the existence, we shall look at a modified disk potential which is better suited to distinguish the constructed monotone Lagrangian tori in our situation.

\begin{definition}
Identify $N \simeq \Z^m$ with the fundamental group $\pi_1 (L)$.
For a lattice vector $\mathbf{v} \in N$, we define
$$
c_{\mathbf{v}} \coloneqq \begin{cases}
1 &\mbox{ if there exists $\beta \in \pi_2(X, L)$ such that $\partial \beta = \mathbf{v}$ and $n_\beta \neq 0$, } \\
0 &\mbox{ otherwise.}
\end{cases}
$$
The \emph{refined disk potential} of $L$ is defined by
\begin{equation}\label{equ_refdiskpotentialfun}
W^{\mathrm{ref}}_L (\mathbf{z}) \coloneqq \sum_{\mathbf{v} \in N}  c_{\mathbf{v}} \cdot z^{\mathbf{v}}.
\end{equation}
\end{definition}

\begin{definition}\label{def_newtonrefineddisk}
We denote by $\Delta^{\mathrm{ref}}_L$ the Newton polytope of the refined disk potential  $W_L^{\mathrm{ref}}$ of a monotone Lagrangian torus $L$, that is, 
$$
\Delta^{\mathrm{ref}}_L \coloneqq \mbox{the convex hull of $\{ \mathbf{v} \mid \exists \beta \mbox{ such that } \mathbf{v} = \partial \beta \mbox{ and } n_\beta \neq 0 \}$ in $N_\R$.} 
$$
\end{definition}

We also have the invariance of the Newton polytope of a refined disk potential by Proposition~\ref{lemma_invarianceofcounting}.

\begin{corollary}\label{cor_polytoperef}
Suppose that there are two monotone Lagrangian tori $L_1$ and $L_2$ in a symplectic manifold $(X, \omega)$.
If there is a symplectomorphism $\phi \colon X \to X$ such that $\phi (L_1) = L_2$, then the Newton polytopes $\Delta^{\mathrm{ref}}_{L_1}$ and $\Delta^{\mathrm{ref}}_{L_2}$ of refined disk potentials are unimodularly equivalent. \end{corollary}

In the remaining part of this section, we explore the relationship between $\Delta^\vee$ in Definition~\ref{def_combinatorialnewtonpoly} and $\Delta^{\mathrm{ref}}_{L, \mathcal{B}}$ in Definition~\ref{def_newtonrefineddisk}. Recall that the Lagrangian torus $L$ is constructed from the Newton--Okounkov polytope $\Delta \subset M_\R \simeq \R^m$. The set $\mathcal{B}_0$ is taken as the oriented loops  $\{ \vartheta_1, \vartheta_2, \cdots,\vartheta_m \}$ in $L$ respectively corresponding to the dual basis elements of the standard basis for $M$ via $N \simeq \pi_1 (L)$. From now on, we set $\Delta^{\mathrm{ref}}_{L} \coloneqq \Delta^{\mathrm{ref}}_{L, \mathcal{B}_0}$ for notational simplicity.

\begin{proposition}\label{proposition_deltaveedleta}
The combinatorial dual $\Delta^\vee$ of a $\mathbb{Q}$-Gorenstein Fano polytope $\Delta$ is contained in the Newton polytope of the refined disk potential $W_L^{\mathrm{ref}}$, that is, $\Delta^\vee \subset \Delta^{\mathrm{ref}}_{L}$.
\end{proposition}

\begin{remark}
Recall that $\Delta^{\mathrm{disk}}_{L, \mcal{B}_0}$ denotes the Newton polytope of the disk potential $W_{L, \mcal{B}_0}$ of $L$. The authors do not know whether $\Delta^\vee \subset \Delta^{\mathrm{disk}}_{L, \mcal{B}_0}$ holds or not in a general setting. Galkin--Mikhalkin in \cite{GalkinMikhalkin} discovered an algebro-geometric condition on $\Delta^\vee = \Delta^{\mathrm{disk}}_{L, \mcal{B}_0}$. To apply their criterion, one needs to verify that both the toric variety $({X}_0,\mathrm{Sing}({X}_0)))$ and the total space $(\mathfrak{X}, \mathrm{Sing}(\mathfrak{X}) = \mathrm{Sing}({X}_0))$ are $\mathbb{Q}\Gamma$-spaces, see the definition therein. In our criterion, we do \emph{not} need to check that $(\mathfrak{X}, \mathrm{Sing}(\mathfrak{X}))$ is a $\mathbb{Q}\Gamma$-space because we introduce the refined disk potential and make use of the map $\phi$ in Theorem~\ref{theorem_HaradaKaveh}.
\end{remark}

We denoted by $F_\iota$ the facet of $\Delta$ contained in the hyperplane $H_{\mathbf{v}_\iota, 1}$ in~\eqref{equ_polytopeshalfsp}. Let $\beta_\iota \in \pi_2(X,L)$ be a homotopy class represented by a gradient holomorphic disk intersecting the inverse image of $F_\iota$. Such a class is called \emph{basic}. To orient the moduli space $\mcal{M}(L; \beta_\iota)$, we take the spin structure and the orientation of $L$ given by the torus action.

\begin{lemma}\label{lemma_disksuperpotentialfacet}
For each $\iota = 1, 2, \cdots, \kappa$, the counting invariant of $\beta_\iota$ is one, that is, $n_{\beta_\iota} = 1$. In particular, the refined disk potential of $L$ is of the form
$$
W_L^{\mathrm{ref}} (\mathbf{z}) = \sum_{\iota=1}^\kappa \mathbf{z}^{\mathbf{v}_\iota} + \tilde{W}_L^{\mathrm{ref}} (\mathbf{z})
$$
where $\tilde{W}_L^{\mathrm{ref}} (\mathbf{z})$ is a subtraction-free Laurent polynomial.
\end{lemma}

\begin{proof}
Let $\beta$ be a basic class of $(X, L)$ corresponding to a facet $F$. To compute the counting invariant $n_\beta$, we exploit the toric degeneration $\mathfrak{X}$ of $X$. Set $L_\tau \coloneqq \phi_{1,\tau}(L)$. By using the toric structure on ${X}_0$, we generate a gradient holomorphic disk $\varphi \colon (\mathbb{D}, \partial \mathbb{D}) \to ({X}_0, L_0)$ that intersects with the inverse image of the relative interior of $F$. The image of $\varphi$ is contained in the smooth locus ${U}_0 \coloneqq {X}_0 \backslash \mathrm{Sing}({X}_0)$. 

Choose a compatible almost complex structure $J_1$ on ${X}_1$ as an extension of the pull-back of the toric complex structure on an open subset of $\phi^{-1} ({U}_0)$ via $\phi$. By taking the interpolation part sufficiently small enough for the extension, we may assume the composition $\phi^{-1} \circ \varphi$ becomes a $J_1$-holomorphic disk. In particular, $\phi_{*} (\beta) = [\varphi]$. 

Let $\widetilde{{X}}_0$ be a simplicialization of the central toric variety ${X}_0$. It comes with a toric morphism $\Pi \colon \widetilde{{X}}_0 \to {{X}_0}$. The inverse image $\widetilde{L}_0 \coloneqq \Pi^{-1} (L_0)$ is a Lagrangian toric fiber of $\widetilde{X}_0$ because the algebraic torus of $\widetilde{{X}}_0$ maps into the algebraic torus of ${X}_0$. We denote by $\widetilde{\beta}$ the class represented by a strict transformation $\widetilde{\varphi}$ of $\varphi$. On the other hand, since the toric morphism $\Pi$ is holomorphic, the map $\Pi$ induces 
$$
\Pi_* \colon \mcal{M}_1(\widetilde{L}_0, \widetilde{J}_0; \widetilde{\beta}) \to \mcal{M}_1( {L}_0, J_0; \phi_{*} (\beta)) \quad \varphi^\prime \mapsto \Pi \circ \varphi^\prime
$$
where $J_0$ and $\widetilde{J}_0$ are the standard complex structures.

We claim that the map $\Pi_*$ is an orientation preserving diffeomorphism. To see this, we first note that $\varphi$ is equal to $\Pi \circ \widetilde{\varphi}$ in $\mcal{M}_1({L}_0, J_0; \phi_{*} (\beta))$ by the construction of strict transformations.  Also, $\varphi^\prime$ is equal to the strict transformation of $\Pi \circ \varphi^\prime$ in $\mcal{M}_1( \widetilde{L}_0, \widetilde{J}_0; \widetilde{\beta})$. This is because the class $\widetilde{\beta} = [\varphi^\prime]$ is a regular and basic class of Maslov index two and hence spheres cannot bubble off. The classification of holomorphic disks in \cite[Theorem 6.2]{ChoPoddar} yields that two holomorphic disks in this class $\widetilde{\beta}$ having the same boundary are equal. (More precisely, this is the case where $c_{ij} = 0$ and only one $d_j = 1$ therein.) Thus, the map $\Pi_*$ is a diffeomorphism. By choosing a spin structure and orientation of $L_0$, $L_1$, and $\widetilde{L}_0$ induced from the torus action, we obtain the following commutative diagram 
$$
\xymatrix{
\mcal{M}_1 (L_1, J_1; \beta) \ar[r]^{\phi_{*} \,\,\,\,\,\,\,\,\, } \ar[d]_{\mathrm{ev}} & \mcal{M}_1 (L_0, J_0;  \phi_{*} (\beta)) \ar[d]^{\mathrm{ev}}   & \mcal{M}_1 (\widetilde{L}_0, \widetilde{J}_0;  \widetilde{\beta}) \ar[l]_{\,\,\,\,\,\,\,\,\,\,\, \Pi_* } \ar[d]_{\mathrm{ev}}\\
L_1 \ar[r]_{ \phi}  & L_0  & \widetilde{L}_0 \ar[l]^\Pi }
$$
such that every arrow is an orientation preserving diffeomorphism. By \cite[Proposition 9.3]{ChoPoddar}, the counting invariant is one, that is, $n_{\widetilde{\beta}} = 1$. We conclude that $n_\beta = 1$.
\end{proof}

\begin{proof}[Proof of Proposition~\ref{proposition_deltaveedleta}]
By Lemma~\ref{lemma_disksuperpotentialfacet}, the conclusion follows.
\end{proof}

\section{Infinitely many distinct monotone Lagrangian tori via cluster mutations}\label{sec_clustermutandinf}

Let $X$ be a smooth Fano variety of complex dimension $m$ polarized by a very ample line bundle $\mcal{L}$, a positive power of the anticanonical bundle of $X$. Thanks to Proposition~\ref{proposition_toricdegenerationsmonotone}, $X$ possesses a monotone Lagrangian torus if 
$X$ admits the normal toric degeneration associated with a $\mathbb{Q}$-Gorenstein Fano Newton--Okounkov polytope. Suppose that $X$ has an \emph{infinite} family of such $\mathbb{Q}$-Gorenstein Fano and normal toric degenerations. Then we produce an infinite family of monotone Lagrangian tori of $X$. The goal of this section is to devise a criterion for the family to contain infinitely many \emph{distinct} monotone Lagrangian tori when the family of Newton--Okounkov polytopes has a \emph{tropical cluster structure}, see Definition~\ref{definition_tropicalclusterstr} for the precise meaning.

\subsection{Cluster varieties and cluster algebras}\label{sec_clustercaralg}

We begin by reviewing some notions on skew-symmetrizable cluster algebras and varieties, following Fock--Goncharov \cite{FockGoncharov} and Gross--Hacking--Keel \cite{GrossHackingKeelBirclu}. 

We fix the following data $\Gamma$ (called a \emph{fixed data}) consisting of 
\begin{itemize}
\item a lattice $N$ of finite rank with a skew-symmetric bilinear form
\[
	\{ \cdot \,\, , \, \cdot \} \colon N \times N \to \Q,
\]
\item an \emph{unfrozen} sublattice $N_\mathrm{uf} \subset N$, a saturated sublattice of $N$,
\item an index set $J$ such that $|J| = \mathrm{rank}\, N$, 
\item a subset $J_\mathrm{uf}$ of $J$ such that $| J_\mathrm{uf} | = \mathrm{rank}\, N_\mathrm{uf}$,
\item a sublattice $N^\circ \subset N$ of finite index such that 
$$
\{ N_\mathrm{uf}, N^\circ \} \subset \Z \mbox{ and } \{ N, N_\mathrm{uf} \cap N^\circ \} \subset \Z,
$$ 
\item positive integers $d_j$ for $j \in J$ of which greatest common divisor is one.
\end{itemize}
For the fixed data $\Gamma$, a \emph{seed} $\seed$ is a labeled collection of elements of $N$
\begin{equation}\label{equ_seedei}
\seed \coloneqq ( \mathbf{e}_j \mid j \in J )
\end{equation}
such that $ \{ \mathbf{e}_j \mid j \in J \}$ is $\Z$-basis for $N$, $ \{ \mathbf{e}_j \mid j \in J_\mathrm{uf} \}$ is a $\Z$-basis of $N_\mathrm{uf}$, and $\{ d_j \mathbf{e}_j \mid j \in J \}$ is a $\Z$-basis for $N^\circ$. 

We set $\varepsilon_{r,s} \coloneqq \{ \mathbf{e}_r, \mathbf{e}_s \} \, d_s$. Note that $\varepsilon_{r,s} \in \Z$ if at least one of $r$ and $s$ is in $J_\mathrm{uf}$. We then obtain a (not necessarily skew-symmetric) matrix $\widehat{\varepsilon} = ( \varepsilon_{r,s} )_{r,s \in J} \in M_{J \times J}(\mathbb{Q})$. This matrix $\widehat{\varepsilon}$ is called an \emph{exchange matrix} of $\Gamma$. Note that 
\begin{equation}\label{equ_exchangematri}
\varepsilon_{r,s} d_r + \varepsilon_{s,r} d_s = 0.
\end{equation}

Let $M \coloneqq \mathrm{Hom}_\Z (N, \Z)$ be the dual lattice of $N$. We denote by $\{\mathbf{e}_j^* \mid j \in J \}$ the dual basis of $\{ \mathbf{e}_j \mid j \in J \}$ for $M$. Let $M^\circ \coloneqq \mathrm{Hom}_\Z (N^\circ, \Z)$ be the dual lattice of $N^\circ$. We also have a $\Z$-basis $\{ \mathbf{f}_j \coloneqq d_j^{-1} \mathbf{e}_j^* \mid {j \in J} \}$ for $M^\circ$. To each seed $\seed$, we associate the dual pair of an $\mcal{A}$-torus $\mcal{A}_\seed$ and an $\mcal{X}$-torus $\mcal{A}^\vee_\seed$ defined by
\begin{equation}\label{equ_axtorus}
\mcal{A}_\seed \coloneqq T_{N^\circ} \coloneqq \mathrm{Spec}(\C [M^\circ]) \mbox{ and } \mcal{A}^\vee_\seed \coloneqq T_{M} \coloneqq \mathrm{Spec}(\C [N]).
\end{equation} 

For a given fixed data $\Gamma$ and a chosen seed $\seed_0$, we can produce seeds and exchange matrices by mutating the seed $\seed_0$ into all possible directions in $J_\mathrm{uf}$ inductively. The chosen seed $\seed_0$ at the beginning is called an \emph{initial seed}. For each index $k \in J_\mathrm{uf}$, the \emph{mutation in the $k$-direction} of various data is defined below. 

\begin{definition}
Let $[a]_+ \coloneqq \mathrm{max}\{a, 0\}$ and $\mathrm{sgn} (a) \coloneqq$ the sign of $a$.
	For an index $k \in J_{\mathrm{uf}}$, the \emph{mutation of a seed} $\seed$ \emph{in the} $k^{\mathrm{th}}$-\emph{direction} is a seed 
	$\mu_k (\seed)  \coloneqq \seed^\prime$ defined by 
\begin{equation}\label{equ_mutationmuk}
		\seed^\prime \coloneqq ( \mathbf{e}_j^\prime \mid j \in J ), \quad 
		\mathbf{e}^\prime_j \coloneqq
			\begin{cases}
				- \mathbf{e}_j &\mbox{if $j = k$,}\\
					\mathbf{e}_j + [ \varepsilon_{j, k} ]_+ \, \mathbf{e}_k &\mbox{if $j \neq k$.}
			\end{cases}
\end{equation}
Here, $(\varepsilon_{j, k})$ is the exchange matrix associated with the seed $\seed$. The basis element $\mathbf{e}^\prime_j$ is also denoted by $\mu_k (\mathbf{e}_j)$.
\end{definition}

Then the set $\{ \mathbf{e}_j^\prime \mid j \in J \}$ is still a $\Z$-basis for $N$. The subset $\{ \mathbf{e}_j^\prime \mid j \in J_\mathrm{uf} \}$ is a $\Z$-basis for $N_\mathrm{uf}$ and $\{ d_j \mathbf{e}_j^\prime \mid j \in J\}$ also forms a $\Z$-basis for $N^\circ$. Note that the basis $\{ \mathbf{f}^\prime_j \mid {j \in J} \}$ is obtained from $\{ \mathbf{f}_j \mid {j \in J}  \}$ via the formula
\begin{equation}\label{equ_mutationmukf}
		\mathbf{f}^\prime_j \coloneqq
			\begin{cases}
				- \mathbf{f}_k + \sum_i [ \varepsilon_{k, i} ]_+ \, \mathbf{f}_i &\mbox{if $j = k$,}\\
					\mathbf{f}_j  &\mbox{if $j \neq k$}
			\end{cases}
\end{equation}
because $\varepsilon_{r,s} d_r + \varepsilon_{s,r} d_s = 0$. Also, under the mutation $\mu_k$ in~\eqref{equ_mutationmuk}, the exchange matrix is also mutated via the formula $\mu_k (\widehat{\varepsilon}) \coloneqq \widehat{\varepsilon}^\prime$ given by
	\begin{equation}\label{equ_mutationexchangematri}
		\widehat{\varepsilon}^\prime \coloneqq (\varepsilon_{r,s}^\prime)_{r, s \in J}, \quad 
		\varepsilon_{r,s}^\prime \coloneqq \{ \mathbf{e}_r^\prime, \mathbf{e}_s^\prime \} d_s = 
			\begin{cases}
				- \varepsilon_{r,s} &\mbox{if $r = k$ or $s = k$,}\\
				\varepsilon_{r,s} + \mathrm{sgn} (\varepsilon_{k,s}) [ \varepsilon_{r,k} \varepsilon_{k,s}]_+ &\mbox{otherwise.}
			\end{cases}
	\end{equation}	

As in~\eqref{equ_axtorus}, each seed assigns an $\mcal{A}$-torus and an $\mcal{X}$-torus. We define a birational transformation between $\mcal{A}$-tori (resp. between $\mcal{X}$-tori). For each seed $\seed = (  \mathbf{e}_j \mid {j \in J} )$, we associate the coordinate functions $( X_{j} \coloneqq z^{\mathbf{e}_j} \mid j \in J )$ on the $\mcal{X}$-torus $\mcal{A}^\vee_\seed$. Dually, we associate the coordinate functions $( A_{j} \coloneqq z^{\mathbf{f}_j} \mid j \in J )$ on the $\mcal{A}$-torus $\mcal{A}_\seed$. For $\seed^\prime = \mu_k (\seed)$, we denote by $A^\prime_{j}$ the coordinate function $z^{\mathbf{f}_j^{\prime}}$ on $\mcal{A}_{\seed^\prime}$ corresponding to $\mathbf{f}^\prime_j$ and have
\begin{equation}\label{equ_clusteramut}
	\mu_k \colon \mcal{A}_\seed \dashrightarrow \mcal{A}_{\seed^\prime}, \quad 
	\mu^{*}_k (A^\prime_j) = 
		\begin{cases}
				A_k^{-1} \left( \prod_{i \in J} A_i^{[\varepsilon_{k,i}]_+}  + \prod_{i \in J} A_i^{[-\varepsilon_{k,i}]_+} \right) &\mbox{if $j = k$,} \\
				A_j &\mbox{if $j \neq k$.}
		\end{cases}	
\end{equation}
For $\seed^\prime = \mu_k (\seed)$, we denote by $X^\prime_{j}$ the coordinate function $z^{\mathbf{e}_j^{\prime}}$ on $\mcal{A}^\vee_{\seed^\prime}$ and have
\begin{equation}\label{equ_clusterxmut}
	\mu_k \colon \mcal{A}^\vee_\seed \dashrightarrow \mcal{A}^\vee_{\seed^\prime}, \quad 
	\mu^{*}_k (X^\prime_j) = 
		\begin{cases}
				X_j^{-1} &\mbox{if $j = k$,} \\
				X_j \left(  1  + X_k^{-  \mathrm{sgn}(\varepsilon_{j,k})} \right)^{- \varepsilon_{j,k}} &\mbox{if $j \neq k$.}
		\end{cases}
\end{equation}

To parametrize the seeds, let $\mathbb{T}$ be the oriented rooted tree with $|J_\mathrm{uf}|$ outgoing edges from each vertex labeled by the elements of $J_\mathrm{uf}$. The root $t_0$ corresponds to the initial seed $\seed_0$. If two vertices $t$ and $t^\prime$ are connected by a directed edge from $t$ to $t^\prime$, then we denote by 
\begin{equation}\label{equ_mukttprime}
\mu_k(t) = t^\prime. 
\end{equation}
Also, the seed and the exchange matrix corresponding to $t$ are denoted by ${\seed}_t$ and $\widehat{\varepsilon}_t$, respectively. By gluing the $\mcal{A}$-tori via the birational transformations~\eqref{equ_clusteramut}, we obtain a scheme
\begin{equation}\label{eq_acluster}
\mcal{A}_{\Gamma, \seed_0} 
= \bigcup_{t \in \mathbb{T}} \mcal{A}_{\seed_t},
\end{equation}
which is called an $\mcal{A}$-\emph{cluster variety} of the chosen data $(\Gamma, \mathbf{s}_0)$. Dually, an $\mcal{X}$-\emph{cluster variety} of the data $(\Gamma, \mathbf{s}_0)$ is defined by gluing the $\mcal{X}$-tori via the birational transformations~\eqref{equ_clusterxmut}
\begin{equation}\label{eq_xcluster}
\mcal{A}^\vee_{\Gamma, \seed_0} 
= \bigcup_{t \in \mathbb{T}} \mcal{A}^\vee_{\seed_t}.
\end{equation}
The $\mcal{X}$-cluster variety $\mcal{A}_{\Gamma, \seed_0}^\vee$ is also called the \emph{(Fock--Goncharov) dual} of $\mcal{A}_{\Gamma, \seed_0}$.

We also deal with cluster algebras and upper cluster algebras later on and we briefly recall them. From the fixed data $\Gamma$ and an initial seed $\seed_0$, consider the function field $\C(\mcal{A}_{\seed_0})$ of the $\mcal{A}$-torus $\mcal{A}_{\seed_0}$. Suppose that $\seed_{t}$ is obtained by applying a finite sequence of mutations in~\eqref{equ_mutationmuk} to $\seed$, that is, $\seed_{t} = \mu_{{t}_0, t} (\seed_{t_0})$ for some
$\mu_{{t}_0, t} \coloneqq \mu_{k_\ell} \circ \cdots \circ \mu_{k_2} \circ \mu_{k_1}$. For $j \in J$, we denote by ${A}_{j,\seed_t}$ the coordinate function corresponding to $\mathbf{f}_{j, \seed_t}$ on the torus $\mcal{A}_{\seed_t}$ associated with $t$. By the Laurent phenomenon in \cite{FZ02}, every coordinate function $A_{j,\seed_t}$ for $\mcal{A}_{\seed_t}$ can be expressed as a Laurent polynomial of the variables $\{ A_{j,t_0} \mid j \in J\}$ by pulling $A_{j,t}$ via the sequence of mutations corresponding to $\mu_{{t}_0, t}$ for $\mcal{A}$-cluster charts in~\eqref{equ_clusteramut}. By abuse of notation, we denote by $\mu_{{t}_0, t}$ the corresponding sequence of mutations in~\eqref{equ_clusteramut} and denote the Laurent polynomial  
\begin{equation}\label{equ_defatjvar}
A_{j,t} \coloneqq \mu_{{t}, t_0}^* (A_{j, \seed_t}) \in \C(\mcal{A}_{\seed_0}).
\end{equation}
We define upper cluster algebras and ordinary cluster algebras associated with $(\Gamma, \seed_0)$.

\begin{definition}\label{def_ucaca}
An \emph{upper cluster algebra} $\mathrm{up}(\mcal{A})$ is defined by 
$$
\mathrm{up}(\mcal{A}) = \bigcap_{t \in \mathbb{T}} \C [A^{\pm 1}_{j,t} \mid j \in J] \subset \C(\mcal{A}_{\seed_0}).
$$ 
An \emph{(ordinary) cluster algebra} $\mathrm{ord}(\mcal{A})$ is defined by the $\C$-subalgebra of $\C(\mcal{A}_{\seed_0})$ generated by
$$
\left\{A_{j,t} \mid t \in \mathbb{T}, j \in J_\mathrm{uf} \right\} \cup \left\{A_{j,t}^{\pm 1} \mid t \in \mathbb{T}, j \in J \backslash J_\mathrm{uf} \right\}.
$$
Each element of the tuple $(A_{j,t})_{j \in J}$ is called a \emph{cluster variable}.
\end{definition}

Indeed, the upper/ordinary cluster algebras in Definition~\ref{def_ucaca} can be defined from smaller pieces of information than a fixed data $\Gamma$ and a seed $\seed_0$ in~\eqref{equ_seedei}. For the later usage of cluster algebras in Chapter~\ref{chap_clusterflag} and~\ref{Chap_exchlarge}, we recall Fomin--Zelevinsky's way of constructing cluster algebras in \cite{FZ02}. Let $\C (z_j \mid j \in J)$ be the field of rational functions. A \emph{Fomin--Zelevinsky's seed} $\seed$ consists of 
\begin{itemize}
\item a $J$-tuple $(A_j)_{j \in J}$ of elements that is a free generating set of $\C (z_j \mid j \in J)$
\item a matrix $\varepsilon = (\varepsilon_{r,s})_{r \in J_{\mathrm{uf}}, s \in J} \in {M}_{J_{\mathrm{uf}} \times J} (\Z)$ such that~\eqref{equ_exchangematri} holds for all $r,s \in J_{\mathrm{uf}}$. 
\end{itemize}
The matrix $\varepsilon$ is called an \emph{(extended) exchange matrix}. When constructing the cluster algebras from $\Gamma$ and $\seed_0$, we take the extended exchange matrix $\varepsilon$ consisting of entries of $\widehat{\varepsilon}$. As a result, this extended exchange matrix $\varepsilon$ is a submatrix of the previously defined exchange matrix $\widehat{\varepsilon} \in M_{J \times J}(\mathbb{Q})$. For this reason, $\varepsilon$ is also called an exchange matrix, and the $(r,s)$-entry of both $\varepsilon$ and $\widehat{\varepsilon}$ will be denoted by $\varepsilon_{r,s}$ in this case.

Out of a seed $( (A_j)_{j \in J}, \varepsilon )$, we can produce Fomin--Zelevinsky's seeds by mutating the seed $\seed_0$ into all possible directions in $J_\mathrm{uf}$ inductively. The mutation formulas for Fomin--Zelevinsky's seeds are given in~\eqref{equ_clusteramut} and~\eqref{equ_mutationexchangematri}. We then can produce $\mathrm{up}(\mcal{A})$ and $\mathrm{ord}(\mcal{A})$ in Definition~\ref{def_ucaca}. In this case, the mutations are involutive and hence it suffices to employ an unoriented $|J_{\mathrm{uf}}|$-regular tree whose edges are labeled by $J_{\mathrm{uf}}$ to parametrize the Fomin--Zelevinsky's seeds. By abuse of notation, we also denote by $\mathbb{T}$ this tree and call it the \emph{exchange graph}.

\subsection{Tropicalized cluster mutations and $\Q$-Gorenstein Fano polytopes}\label{sec_tropGnormal}

Let $X$ be a smooth projective variety polarized by a very ample line bundle $\mcal{L}$. Let us take a reference section $h \in \scr{L} \coloneqq H^0(X, \mcal{L})$ as in~\eqref{equ_SR}. Suppose that there exist a fixed data $\Gamma$ and an initial seed $\seed_{0} = \seed_{t_0}$. We inductively produce the set $\{ \seed_t \mid t \in \mathbb{T} \}$ of seeds generated by $\seed_{t_0}$ as in Section~\ref{sec_clustercaralg} and parametrized by the vertices of an oriented rooted tree $\mathbb{T}$ with $|J_\mathrm{uf}|$ outgoing edges labeled by the elements of $J_\mathrm{uf}$ from each vertex. 

Suppose that we have a family of valuations with one-dimensional leaves and parametrized by the vertices of $\mathbb{T}$, say $\{ v_t \mid t \in \mathbb{T} \}$. Assume that the semigroup $S_t \coloneqq S(\mcal{L}, v_t, h)$ constructed from $v_t$ is finitely generated. Let  $\Delta_t \coloneqq \Delta(\mcal{L}, v_t, h)$ be the Newton--Okounkov polytope of $X$ corresponding to $S_t$ for each $t \in \mathbb{T}$ in~\eqref{equ_Deltav}. Using the seed $\seed_t$, each Newton--Okounkov polytope $\Delta_t$ (resp. semigroup $S_t$) in the family can be regarded as a subset of $M_\R = \R \langle \mathbf{f}_j \mid j \in J \rangle = \left\{ \sum_j u_j \mathbf{f}_j \mid u_j \in \R \mbox{ for all $j \in J$} \right\} \simeq \R^{J}$ (resp. $\R_{\geq 0} \times M_\R  \simeq \R_{\geq 0}  \times \R^{J})$.

\begin{remark}
If we consider the family of valuations constructed from a cluster algebra in the GHHK setting, with a choice of the opposite dominance order as in \cite[Theorem 4.3]{FO20}, then the semigroup obtained from each valuation is always finitely generated by \cite[Lemma 8.29 and Theorem 8.30]{GrossHackingKeelKontsevich}. This case can serve as an example in Section~\ref{sec_tropGnormal}. 
\end{remark}

We now describe relation between semigroups $S_t$ and $S_{t^\prime}$ and between Newton--Okounkov polytopes $\Delta_t$ and $\Delta_{t^\prime}$.

\begin{definition}\label{definition_tropicalclusterstr}
Consider a family $\{ S_t \subset \R_{\geq 0} \times M_\R \mid t \in \mathbb{T} \}$ of semigroups  parametrized by $\mathbb{T}$ and the family $\{ \Delta_t \subset M_\R \mid t \in \mathbb{T} \}$ of associated Newton--Okounkov bodies, which are also parametrized by $\mathbb{T}$. 

\begin{enumerate}
\item For an unfrozen index $k \in J_\mathrm{uf}$ and a pair $(t, t^\prime)$ of vertices with $\mu_k (t) = t^\prime$ in  $\mathbb{T}$, a \emph{tropicalized cluster mutation in the $k^{\mathrm{th}}$-direction}
$$
\mu^{{T}}_k \colon M_\R \simeq_{\seed_t} \R^{J} \to M_\R \simeq_{\seed_t^\prime} \R^{J}
$$
is defined to be a piecewise-linear transformation defined by 
\begin{equation}\label{equ_tropical1}
\mathbf{u} \coloneqq (u_j) \mapsto \mathbf{u}^\prime \coloneqq (u^\prime_j), \quad 
u_j^\prime = 
\begin{cases}
- u_j &\mbox{ if $j = k$} \\
u_j + \left[ \varepsilon_{k,j} \right]_+ u_k  &\mbox{ if $j \neq k$ and $u_k \geq 0$} \\
u_j + \left[ -\varepsilon_{k,j} \right]_+ u_k &\mbox{ if $j \neq k$ and $u_k \leq 0$}
\end{cases}
\end{equation}
where ${\varepsilon}_t = (\varepsilon_{i,j})_{(i,j)\in J_{uf} \times J}$ is the exchange matrix associated with $t$. 
\item We also define 
\begin{equation}\label{equ_hattropical1}
\widehat{\mu}^{{T}}_k \colon \R_{\geq 0} \times M_\R \simeq_{\seed_t} \R_{\geq 0} \times \R^{J} \to \R_{\geq 0} \times M_\R \simeq_{\seed_{t^\prime}} \R_{\geq 0} \times \R^{J}
\end{equation}
by $\widehat{\mu}^{{T}}_k (r, \mathbf{u}) \coloneqq (r, \mu^{{T}}_k (\mathbf{u}))$. This map $\widehat{\mu}^{{T}}_k$ is also called a \emph{tropicalized cluster mutation in the $k^{\mathrm{th}}$-direction}.

\item The family of semigroups parametrized by $\mathbb{T}$ is said to have a \emph{tropical cluster structure} if for every pair $(t, t^\prime)$ of vertices with $t^\prime = \mu_k (t)$, the associated semigroups $S_t$ and $S_{t^\prime}$ are related by the tropicalized cluster mutation $\widehat{\mu}^T_k$ in the $k^{\mathrm{th}}$-direction, that is, $S_{t^\prime} = \widehat{\mu}^T_k (S_t)$.

\item The family of Newton--Okounkov bodies of $X$ parametrized by $\mathbb{T}$ is said to have a \emph{tropical cluster structure} if for every pair $(t, t^\prime)$ of vertices with $t^\prime = \mu_k (t)$, the associated Newton--Okounkov bodies $\Delta_t$ and $\Delta_{t^\prime}$ are related by the tropicalized cluster mutation $\mu^T_k$ in the $k^{\mathrm{th}}$-direction, that is, $\Delta_{t^\prime} = \mu^T_k (\Delta_t)$.
\end{enumerate}
\end{definition}

\begin{remark}
We outline a scenario where a smooth projective variety $X$ has a tropical cluster structure. The reader is referred to \cite[Section 4]{FO20} for a detailed explanation. Consider a cluster ensemble $(\mcal{A}_{\Gamma, \seed_0}, \mcal{A}_{\Gamma, \seed_0}^\vee)$ defined in~\eqref{eq_acluster} and~\eqref{eq_xcluster}. Assume that $X$ is a compactification of the $\mathcal{A}$-cluster variety $\mcal{A}_{\Gamma, \seed_0}$ and let $X^\vee$ be a (partial compactification) of its dual $\mathcal{X}$-cluster variety $\mcal{A}_{\Gamma, \seed_0}^\vee$. Suppose that the compactification $X$ is a smooth Fano projective variety and a Landau--Ginzburg mirror of $X$ is given by a regular function $W$ on the $\mathcal{X}$-cluster variety $X^\vee$. In \cite{GrossHackingKeelKontsevich}, they constructed a theta basis for a certain class of cluster algebras from a scattering diagram. When a seed $\seed$ is fixed, the $g$-vectors of this theta basis are parametrized by the tropical integer points of the Newton--Okounkov polytope, which is constructed from a valuation from the exchange matrix associated with $\seed$. In this case, every set of $g$-vectors from a seed and the set of $g$-vectors from the initial seed are related by a finite sequence of tropicalized ($\mcal{X}$-)cluster mutations and so are the corresponding Newton--Okounkov polytopes. Also, it is expected that the Newton--Okounkov polytopes are given by the tropicalization of the LG mirror $(X^\vee, W)$ restricted to a cluster chart. 
\end{remark}

Suppose that the family $\{S_t \mid t \in \mathbb{T} \}$ of {finitely generated} semigroups has a tropical cluster structure and hence the family $\{\Delta_t \mid t \in \mathbb{T} \}$ of corresponding Newton--Okounkov bodies of $X$ also has a tropical cluster structure. By Theorem~\ref{theorem_Andersontoric}, each $\Delta_t$ is a {rational polytope}. The following lemmas discuss that if the toric degeneration at the initial seed is $\mathbb{Q}$-Gorenstein Fano and normal, then so are all the other toric degenerations.

\begin{lemma}\label{lemma_saturatedind}
Suppose that the semigroup $S_{t_0}$ at the initial seed is saturated. Then, for each $t \in \mathbb{T}$, the semigroup $S_t$ is also saturated.
\end{lemma}
\begin{proof}
By induction, it suffices to prove that if $\mu_k (t) = t^\prime$ for $k \in J_\mathrm{uf}$ and $S_t$ is saturated, then so is $S_{t^\prime}$. We need to check  that if $n \cdot (r, \mathbf{u}^\prime) \in S_{t^\prime}$ for $(r, \mathbf{u}^\prime) \in \mathbb{N} \times \Z^{J}$ and $n \in \mathbb{N}$, then $(r, \mathbf{u}^\prime) \in S_{t^\prime}$. Since $\mu_{k}^T$ is a bijection on the lattice $\Z^{J}$ and $\widehat{\mu}_k^T$ is a bijection from $S_t$ to $S_{t^\prime}$, there exists $(r, \mathbf{u}) \in \mathbb{N} \times \Z^{J}$ such that $\widehat{\mu}_k^T (r, \mathbf{u}) = (r, \mathbf{u}^\prime)$ and $n \cdot (r, \mathbf{u}) \in S_{t}$. Since $S_t$ is saturated, $(r, \mathbf{u}) \in S_t$ and hence $(r, \mathbf{u}^\prime) \in S_{t^\prime}$.
\end{proof}

\begin{lemma}\label{lemma_QGorenstein}
Suppose that the Newton--Okounkov polytope $\Delta_{t_0}$ at the initial seed is $\mathbb{Q}$-Gorenstein Fano with the center $\mathbf{0}$. Then, for every $t \in \mathbb{T}$, the corresponding Newton--Okounkov polytope $\Delta_t$ is also $\mathbb{Q}$-Gorenstein Fano with the center $\mathbf{0}$. More precisely, if $\nu \cdot \Delta_{t_0}^\circ$ is a Fano polytope for some $\nu \in \mathbb{N}$, then so is $\nu \cdot \Delta_{t}^\circ$ for all $t \in \mathbb{T}$.
\end{lemma}

\begin{proof}
By induction, it suffices to prove that if $\mu_k (t) = t^\prime$ for some $k \in J_\mathrm{uf}$ and $\Delta_t$ is $\mathbb{Q}$-Gorenstein Fano with the center at the origin $\mathbf{0}$, then so is $\Delta_{t^\prime}$. Since $\Delta_t$ is $\Q$-Gorenstein Fano, there exists a natural number $\nu \in \mathbb{N}$ such that $\nu \cdot \Delta_{t}^\circ$ is Fano. Set $c = \nu^{-1}$. The polytope $\Delta_t$ can expressed as
$$
\Delta_t = \bigcap_{\iota=1}^{\kappa}  H^+_{c \mathbf{v}_\iota, 1}
$$
satisfying the conditions for~\eqref{equ_polytopeshalfsp} where $\mathbf{v}_1, \mathbf{v}_2, \cdots, \mathbf{v}_\kappa$ are \emph{primitive} lattice vectors.
Since $\Delta_{t^\prime}$ also contains the origin in the interior, $\Delta_{t^\prime}$ can be expressed as
\begin{equation}\label{equ_deltatprime}
\Delta_{t^\prime} = \bigcap_{\iota=1}^{\kappa^\prime}  H^+_{c\mathbf{v}^\prime_\iota, 1} \quad \mbox{where $\mathbf{v}_\iota^\prime \in N_\R$ for $\iota=1, 2, \cdots,  \kappa^\prime$}
\end{equation}
satisfying the conditions for~\eqref{equ_polytopeshalfsp}.

We claim that every $\mathbf{v}_\iota^\prime$ in~\eqref{equ_deltatprime} is a primitive lattice vector. From the explicit expression of the tropicalized cluster mutation $\mu_{k}^T$ in~\eqref{equ_tropical1}, depending on the sign of $u_k$, the image $\mu_{k}^T(\mathbf{u})$ is determined by two linear transformations. Let us denote $\mu_{k}^T(\mathbf{u})$ by $A \mathbf{u}$ or $B \mathbf{u}$ for some $A, B \in \mathrm{GL}(J, \Z)$. Then the half space $H^+_{c\mathbf{v}_\iota, 1}$ corresponds to $H^+_{c(A^T)^{-1}\mathbf{v}_\iota, 1}$, $H^+_{c(B^T)^{-1}\mathbf{v}_\iota, 1}$, or both. Consequently,
$$
\{ \mathbf{v}^\prime_1, \mathbf{v}^\prime_2, \cdots, \mathbf{v}^\prime_{\kappa^\prime} \} \subset \left\{(A^T)^{-1}  \mathbf{v}_1, \cdots, (A^T)^{-1} \mathbf{v}_{\kappa} \right\} \cup \left\{ (B^T)^{-1} \mathbf{v}_1, \cdots, (B^T)^{-1} \mathbf{v}_{\kappa} \right\}. 
$$
If $A \in \mathrm{GL}(J, \Z)$ and $\mathbf{v}$ is a primitive lattice vector, then $A \mathbf{v}$ is also primitive. Thus, every vector $\mathbf{v}^\prime_\iota$ is a primitive lattice vector and $\Delta_{t^\prime}$ is also $\mathbb{Q}$-Gorenstein Fano with the center $\mathbf{0}$. 
\end{proof}

In general, the origin $\mathbf{0} \in M_\R$ may \emph{not} be the center of a $\Q$-Gorenstein Fano Newton--Okounkov polytope $\Delta$ according to our definition of $\Q$-Gorenstein Fano polytope in Definition~\ref{def_qgorensteinpoly}. To apply Lemma~\ref{lemma_QGorenstein}, we need to translate $\Delta$ to position the center at the origin $\mathbf{0}$. In general, however, the translations and the tropicalized cluster mutations may \emph{not} commute. We explore a condition that ensures commutativity.

Recall that for each $t \in \mathbb{T}$, a point $\mathbf{u}$ of $M_\R$ can be regarded as a point in $\R^{J} \simeq_{\seed_{t}} M_\R$. Let $(t, t^\prime)$ be a pair of vertices with $\mu_k (t) = t^\prime$ in $\mathbb{T}$ for an index $k \in J_\mathrm{uf}$. Fix a point $\mathbf{u}_{0} \in M_\R$. For each $t \in \mathbb{T}$, let $\mathbf{u}_{t,0}$ be the image of $\mathbf{u}_0$ under the identification given by $\seed_t \colon$
\begin{equation}\label{equ_defut0u}
M_\R \simeq_{\seed_{t}} \R^{J} \quad (\mathbf{u}_{0} \mapsto \mathbf{u}_{t,0}).
\end{equation}
We denoted by $\tau_t \colon \R^{J} \to \R^{J}$ the translation defined by $\mathbf{u} \mapsto \mathbf{u} - \mathbf{u}_{t,0}$.

\begin{lemma}\label{lemma_transtropicalcomm}
Let $(t, t^\prime)$ be a pair of vertices with $\mu_k (t) = t^\prime$ in $\mathbb{T}$ for an index $k \in J_\mathrm{uf}$. Fix a point $\mathbf{u}_0 \in M_\R$. Let $\mathbf{u}_{t,0}$ be the image of the point $\mathbf{u}_0$ under~\eqref{equ_defut0u}. Then the following are equivalent.
\begin{enumerate}
\item The point $\mathbf{u}_{t,0}$ is fixed under the tropicalized cluster mutation $\mu^T_k$, that is,
\begin{equation}\label{equ_fixeduntrop}
\mu^T_k (\mathbf{u}_{t,0}) = \mathbf{u}_{t^\prime,0}.
\end{equation}
\item The translations and the tropicalized cluster mutation are commutative, that is,
$$
\tau_{t^\prime} \circ \mu_k^T = \mu_k^T \circ \tau_t \mbox{ on $\R^{J}$.}
$$
\item The $k$-th component of $\mathbf{u}_{t,0}$ is equal to zero, that is, $u_{t,0,k} = 0$ where $\mathbf{u}_{t,0} = (u_{t,0,j})_{j \in J}$. 
\end{enumerate}
\end{lemma}

\begin{proof}
The statement $(1) \Rightarrow (2)$ follows from
$$
(\mu_k^T \circ \tau_t)(\mathbf{u}) = \mu_k^T (\mathbf{u} - \mathbf{u}_{t,0}) =  \mu_k^T (\mathbf{u}) - \mathbf{u}_{t^\prime,0} = (\tau_{t^\prime} \circ \mu_k^T) (\mathbf{u}).
$$
For $(2) \Rightarrow (3)$, we observe
\begin{equation}\label{equ_commuan}
- \mathbf{u}_{t^\prime,0} = \tau_{t^\prime}(\mathbf{0}) = (\tau_{t^\prime} \circ \mu_k^T)(\mathbf{0}) = (\mu_k^T \circ \tau_t)(\mathbf{0}) = \mu_k^T (- \mathbf{u}_{t,0}) = - \mu_k^T (\mathbf{u}_{t,0}).
\end{equation}
Restricting to the $k$-th component of~\eqref{equ_commuan}, we obtain 
$$
u_{t^\prime, 0, k} = \mu_k^T (\mathbf{u}_{t,0})_k = -u_{t^\prime, 0, k},
$$
It in turn implies that $u_{t, 0, k} = 0$ by the (piecewise)-linearity of $\mu_k^T$. Finally, suppose that $(3)$ holds. Then~\eqref{equ_fixeduntrop} is obtained by the expression~\eqref{equ_tropical1} for $\mu_k^T$.
\end{proof}

\begin{corollary}\label{cor_equivalenceoftranscomm}
Let $\mathbf{u}_{0}\in M_\R$ and $\mathbf{u}_{t,0}$ be the image of $\mathbf{u}_{0}$ under~\eqref{equ_defut0u}. Let $M_{\mathrm{fr}, \R}$ be the $\R$-vector space generated by $\{\mathbf{e}^{*}_j \mid j \in J \backslash J_{\mathrm{uf}} \}$. The following are equivalent. 
\begin{enumerate}
\item For each index $k \in J_\mathrm{uf}$, the point $\mathbf{u}_{t,0}$ is fixed under the tropicalized cluster mutation $\mu^T_k$, that is, $\mu^T_k (\mathbf{u}_{t,0}) = \mathbf{u}_{t^\prime,0}$ for $\mu_k (t) = t^\prime$.
\item For each index $k \in J_\mathrm{uf}$, the translations and the tropicalized cluster mutations are commutative, that is, $\tau_{t^\prime} \circ \mu_k^T = \mu_k^T \circ \tau_t$ on $\R^{J}$ for $\mu_k (t) = t^\prime$. 
\item The point $\mathbf{u}_{t,0}$ is contained in the space $M_{\mathrm{fr}, \R}$.
\end{enumerate}
\end{corollary}

Let $\mathbf{u}_{t_0,0}$ be the center of the $\Q$-Gorenstein Fano Newton--Okounkov polytope $\Delta_{t_0}$ at the initial seed. We define a point $\mathbf{u}_{t,0}$ of $\Delta_{t}$ by the image of $\mathbf{u}_{t_0,0}$ under the composition of the identifications \begin{equation}\label{equ_howthecenter}
    \R^{J} \simeq_{\seed_{t_0}} M_\R \simeq_{\seed_{t}} \R^{J} \quad  (\mathbf{u}_{t_0,0} \mapsto \mathbf{u}_{0} \mapsto \mathbf{u}_{t,0}).
\end{equation}

\begin{lemma}\label{lemma_Qgorensteinind}
For $t \in \mathbb{T}$, let $\Delta_t$ be a $\Q$-Gorenstein Fano polytope with the center $\mathbf{u}_{t,0}$. Suppose that for a pair $(t, t^\prime)$ of vertices with $\mu_k (t) = t^\prime$ in $\mathbb{T}$ and an index $k \in J_\mathrm{uf}$, \eqref{equ_fixeduntrop} holds. Then $\mu_k^T(\Delta_t) = \Delta_{t^\prime}$ is also a $\Q$-Gorenstein Fano polytope with the center $\mathbf{u}_{t^\prime,0}$.
\end{lemma}

\begin{proof}
Consider the translated polytope $\Delta_t - \mathbf{u}_{t,0} = \tau_{t}(\Delta_t)$, which is a $\Q$-Gorenstein Fano polytope with the center $\mathbf{0}$. By Lemmas~\ref{lemma_QGorenstein} and~\ref{lemma_transtropicalcomm}, $(\tau_{t^\prime} \circ \mu_k ) (\Delta_t) = (\mu_k \circ \tau_{t}) (\Delta_t)$ is also $\Q$-Gorenstein Fano with the center $\mathbf{0}$. Therefore, $\Delta_{t^\prime}$ is a $\Q$-Gorenstein Fano polytope with the center $\mathbf{u}_{t^\prime,0}$.
\end{proof}

In summary, by combining Lemmas~\ref{lemma_saturatedind} and~\ref{lemma_Qgorensteinind}, we deduce the following proposition.

\begin{proposition}[Proposition~\ref{propB}]\label{proposition_Qgorennormalind}
Suppose that the family $\{S_t \mid t \in \mathbb{T} \}$ of \emph{finitely generated} semigroups has a tropical cluster structure and hence the family $\{\Delta_t \mid t \in \mathbb{T} \}$ of corresponding Newton--Okounkov bodies of $X$ also has a tropical cluster structure. If 
\begin{enumerate}
\item the semigroup $S_{t_0}$ is saturated, 
\item the Newton--Okounkov polytope $\Delta_{t_0}$ is a $\mathbb{Q}$-Gorenstein Fano polytope of size $\nu$ for some $\nu \in \mathbb{N}$ with the center $\mathbf{u}_0$, and
\item the center $\mathbf{u}_{0}$ is fixed under the tropicalized cluster mutation of each direction in the sense of~\eqref{equ_fixeduntrop},
\end{enumerate}
then each polytope $\Delta_t$ in the family is also a $\mathbb{Q}$-Gorenstein Fano polytope of size $\nu$ with the center $\mathbf{u}_{t,0}$ and the central fiber of the toric degeneration associated with $\Delta_t$ is a normal toric variety.
\end{proposition}

\subsection{Infinitely many tori from tropically related Newton--Okounkov polytopes}

Let $X$ be a smooth Fano variety of complex dimension $m$ equipped with the K\"{a}hler form of a very ample line bundle $\mcal{L}$, a positive power of the anticanonical bundle of $X$. Suppose that we have a family $\{S_t \mid t \in \mathbb{T}\}$ of finitely generated semigroups, each of which gives rise to a $\mathbb{Q}$-Gorenstein Fano Newton--Okounkov polytope $\Delta_t$ and a normal toric degeneration of a smooth projective variety $X$. If the family has a tropical cluster structure, then we have the family of monotone Lagrangian tori of $X$ by Proposition~\ref{proposition_toricdegenerationsmonotone}. The following theorem provides a criterion for the existence of infinitely many \emph{distinct} monotone Lagrangian tori.
  
\begin{theorem}[Theorem~\ref{theoremC}]\label{theorem_maininthebody}
 Let $X$ be a smooth Fano variety whose anticanonical line bundle $K^{-1}_X$ is ample. Take a very ample line bundle $\mathcal{L} = (K^{-1}_X)^{\otimes \nu}$ for some $\nu \in \mathbb{N}$ and equip $X$ with the K\"{a}hler form inherited from $\mcal{L}$. Assume that $X$ admits the family of Newton--Okounkov polytopes of $\mcal{L}$ arising from a family of finitely generated semigroups having a tropical cluster structure. We denote by $\{ \Delta_t \mid t \in \mathbb{T} \}$ the family of Newton--Okounkov polytopes for some oriented rooted regular tree $\mathbb{T}$ with the initial seed $t_0$.
 Suppose that the following initial conditions hold. 
 \begin{itemize}
     \item The semigroup that generates $\Delta_{t_0}$ at the initial seed is saturated.
     \item The Newton--Okounkov polytope $\Delta_{t_0}$ at the initial seed is $\mathbb{Q}$-Gorenstein Fano. 
     \item The polytope $\Delta_{t_0}$ contains the origin $\mathbf{0}$ and its center $\mathbf{u}_0$ is fixed under the tropicalized cluster mutation of each direction.
\end{itemize}
By Proposition~\ref{proposition_toricdegenerationsmonotone}, for each $t \in \mathbb{T}$, we have a monotone Lagrangian torus, say $L_{t}$. 
If there exists a sequence $\left( t_\ell \right)_{\ell \in \mathbb{N}}$ of seeds and a sequence $(r_\ell, s_\ell)_{\ell \in \mathbb{N}}$ of indices in $J_\mathrm{uf} \times J$ such that
\begin{enumerate}
\item the sequence $( \varepsilon_{r_\ell, s_\ell} )_{\ell \in \mathbb{N}}$ of the $(r_\ell, s_\ell)$-entry in the extended exchange matrix $\varepsilon_{t_\ell}$ associated with $t_\ell$ diverges to $- \infty$ as $\ell \to \infty$,
\item the Newton--Okounkov polytope $\Delta_{t_\ell}$ is contained in the half-space $H^+_{\mathbf{e}_{s_\ell}, 0}$ where $\seed_{t_\ell} = ( \mathbf{e}_j \mid j \in J )$, and 
\item the image of the polytope $\Delta_{t_\ell}$ under the tropicalized cluster mutation $\mu^T_{r_\ell}$ in the $r_\ell$-direction is contained in the half-space $H^+_{\mathbf{e}^\prime_{s_\ell}, 0}$ where $\seed_{\mu_{r_\ell} (t_\ell)} = ( \mathbf{e}^\prime_j \mid j \in J )$,
\end{enumerate}
then the family $\left\{ L_{t_\ell} \mid \ell \in \mathbb{N} \right\}$ contains infinitely many monotone Lagrangian tori, no two of which are related by any symplectomorphism.
 \end{theorem}
  
Here are some remarks on the conditions of Theorem~\ref{theorem_maininthebody}.

\begin{remark}
We emphasize that the corresponding hyperplane $H_{\mathbf{e}_s, 0}$ (resp. $H_{\mathbf{e}^\prime_s, 0}$) need \emph{not} contain a facet of $\Delta_{t_\ell}$ (resp. $\mu^T_{r_\ell} (\Delta_{t_\ell})$) in the condition $(2)$ (resp. $(3)$). It would be harder to apply Theorem~\ref{theorem_maininthebody} if one has to verify that the hyperplane contains a facet. Indeed, the conditions $(2)$ and $(3)$ can be checked by information from the lattice points contained in the semigroup \emph{without} figuring out facets of $\Delta_{t_\ell}$ and $\mu^T_{r_\ell} (\Delta_{t_\ell})$. 

Note that if a polytope $\Delta^\prime$ is obtained from another polytope $\Delta$ by applying the tropicalized cluster mutation $\mu^T_r$ in the $r^{\mathrm{th}}$-direction and $\Delta$ contains the origin $\mathbf{0} \in M$, then $\Delta^\prime$ also contains the origin $\mathbf{0} \in M$. Hence, if the polytope $\Delta_{t_0}$ at the initial seed contains the origin $\mathbf{0}$, then so do the others. 
\end{remark} 

\begin{remark}
To discuss the case where such half-spaces in $(2)$ and $(3)$ exist, we consider the scattering diagram for the cluster variety $\mcal{A}_{\Gamma, \seed_0}$ in \cite{ GrossHackingKeelKontsevich}. The cluster variety $\mcal{A}_{\Gamma, \seed_0}$ is contained in the log Calabi--Yau manifold $X \backslash D$ where $D$ is a normal crossing anticanonical divisor. To each irreducible component of $D$, there is the corresponding initial ray getting into a scattering diagram. While theta function $\theta_{Q, m_0} = \sum_\gamma \mathrm{mono}(\gamma)$
 of \cite[Definition 3.3]{GrossHackingKeelKontsevich} counts a prior several broken lines $\gamma$, one has in fact a single broken line $\gamma$ and hence $\theta_{Q, m_0}$ is a monomial whenever $m_0$ corresponds to an irreducible component of $D$. The tropicalization of the sum of theta functions produces the Newton--Okounkov body (often called a \emph{superpotential polytope}). The above theta function should give rise to the half-spaces in $(2)$ and $(3)$. 
 \end{remark}

A key part of the proof of Theorem~\ref{theorem_maininthebody} is to extract pieces of data from the iterative process~\eqref{equ_tropical1} that enable us to distinguish Lagrangian tori \emph{without} probing \emph{individual} toric degeneration and Newton--Okounkov polytope. We begin by collecting some lemmas.

\begin{lemma}[Theorem 6.4 in \cite{Brn}]\label{proposition_polardual}
Suppose that a polytope $\Delta$ contains the origin $\mathbf{0}$ in its interior. Let $\Delta^\circ$ be the polar dual of $\Delta$. Then the following are equivalent.
\begin{enumerate}
\item The hyperplane $H_{\mathbf{v}, 1}$ is a supporting hyperplane of the polytope $\Delta$.
\item The vector $\mathbf{v}$ is contained in the boundary of ${\Delta}^\circ$, that is, $\mathbf{v} \in \Delta^\circ \backslash \mathrm{Int}({\Delta}^\circ)$.
\end{enumerate}
\end{lemma}

\begin{lemma}\label{lemma_keyinfitnie}
Suppose that a polytope $\Delta^\prime$ is the image of a polytope $\Delta$ under the tropicalized cluster mutation $\mu_r^T$ in~\eqref{equ_tropical1}, that is, $\Delta^\prime = \mu_r^T (\Delta)$. Assume that the polytope $\Delta$ contains the origin $\mathbf{0}$. If $\Delta$ is contained in the half-space $H^+_{\mathbf{e}_s, 0}$, $\Delta^\prime$ is contained in the half-space $H^+_{\mathbf{e}^\prime_s, 0}$, and $\varepsilon_{r,s} \leq 0$, then the half-space $H^+_{\mathbf{e}_s - \varepsilon_{r,s} \mathbf{e}_r, 0}$ is a supporting half-space of $\Delta$. 
\end{lemma}

\begin{proof}
To show that $H^+_{\mathbf{e}_s- \varepsilon_{r,s}\mathbf{e}_r, 0}$ is a supporting half-space of $\Delta$, consider the mutated polytope $\Delta^\prime$. Since $\Delta^\prime$ is contained in $H^+_{\mathbf{e}^\prime_s, 0}$, we have
$$
(\mu_r^T(\mathbf{u}))_s \geq 0 \quad \mbox{for all $\mathbf{u} \in \Delta$}.
$$
By the relation~\eqref{equ_tropical1}, each point $\mathbf{u}$ in the Newton--Okounkov body $\Delta$ satisfies
$$
(\mu_r^T(\mathbf{u}))_s = 
\begin{cases}
u_s  & \mbox{if $u_r\ge 0$} \\
u_s - \varepsilon_{r,s}u_r & \mbox{if $u_r\le 0$}.
\end{cases}
$$
In both cases, we claim that $u_s- \varepsilon_{r,s}u_r \geq 0$.
\begin{enumerate}
    \item If $u_r \leq 0$, then $u_s- \varepsilon_{r,s}u_r = (\mu_r^T(\mathbf{u}))_s \geq 0$.
    \item If $u_r \geq 0$, then $u_s- \varepsilon_{r,s}u_r \geq u_s =(\mu_r^T(\mathbf{u}))_s \geq 0$ because $\varepsilon_{r,s} \leq 0$.
\end{enumerate}
Moreover, the origin $\mathbf{0}$ is contained in the intersection $H_{\mathbf{e}_s-\varepsilon_{r,s}\mathbf{e}_r, 0} \cap \Delta$ and hence $H^+_{\mathbf{e}_s - \varepsilon_{r,s} \mathbf{e}_r, 0}$ is a supporting half-space of $\Delta$.
\end{proof}

We are ready to prove Theorem~\ref{theorem_maininthebody}. 
Before presenting its proof, we address an issue that one should be cautious about. 
Suppose that $\Delta$ is a normalized $\Q$-Gorenstein Fano polytope with the center $\mathbf{u}_0$. Let $H^+_{\mathbf{v}, a}$ be a supporting half-space of the polytope $\Delta$ where $\mathbf{v}$ is a \emph{primitive} lattice vector. Let $\tau \colon \R^J \to \R^J$ be the translation $\mathbf{u} \mapsto \mathbf{u} - \mathbf{u}_0$. If the hyperplane $H_{\mathbf{v}, a}$ contains a \emph{facet}, then $a = 1$ and $\tau(\Delta ) = \Delta - \mathbf{u}_0$ also has a supporting half-space $H^+_{\mathbf{v}, 1}$ by Remark~\ref{lemma_QGorensteinpoly}. However, if $H_{\mathbf{v}, a}$ does \emph{not} contain a facet, depending on $\mathbf{u}_0$ and $a$, $H^+_{\mathbf{v}, 1}$ may \emph{not} be a supporting half-space because
$$
\tau(H^+_{\mathbf{v}, a}) = H^+_{\frac{\mathbf{v}}{\langle \mathbf{u}_0, \mathbf{v} \rangle + a}, 1}.
$$

More generally, suppose that $\Delta$ is a $\Q$-Gorenstein Fano polytope of size $\nu$ with the center $\mathbf{u}_0$. Then $\nu^{-1} \cdot (\Delta - \mathbf{u}_0)$ is normalized. If $\Delta$ has a supporting half-space $H^+_{\mathbf{v}, a}$, then the corresponding half-space supporting $\nu^{-1} \cdot (\Delta - \mathbf{u}_0)$ is 
\begin{equation}\label{equ_translatedformula}
    H^+_{\frac{\nu \mathbf{v}}{\langle \mathbf{u}_0, \mathbf{v} \rangle + a} , 1}.
\end{equation}
To obtain an easy-to-use criterion, we do \emph{not} require the half-spaces to contain a facet of $\Delta$. Extra care is necessary for this reason.
With this issue in mind, we prove Theorem~\ref{theorem_maininthebody}.

\begin{proof}[Proof of Theorem~\ref{theorem_maininthebody}]
Since $|J|$ is finite, there exist indices $r, s \in J$ such that $s = s_\ell$ and $r = r_\ell$ for infinitely many indices $\ell$ in the sequence $\{(r_\ell, s_\ell)\}_{\ell \in \mathbb{N}}$. We denote by $\varepsilon^\ell_{r,s}$ the $(r,s)$-entry of the exchange matrix $\varepsilon_{t_\ell}$. By taking a subsequence if necessary, we may assume that $s_\ell = s$, $r_\ell = r$ and $\varepsilon^\ell_{r,s} < 0$ for all $\ell \in \mathbb{N}$. For notational simplicity, set $\Delta_\ell \coloneqq \Delta_{t_\ell}$ and $\mathbf{u}_{\ell, 0} \coloneqq \mathbf{u}_{t_\ell, 0}$ in~\eqref{equ_howthecenter}. Let $\Delta_{\ell}^\vee = \nu \cdot (\Delta_\ell - \mathbf{u}_{\ell,0})^\circ$ be the combinatorial dual polytope of $\Delta_{\ell}$ where $\nu \in \mathbb{N}$ is determined by $\mathcal{L} = (K^{-1}_X)^{\otimes \nu}$. 

Since the center $\mathbf{u}_{0}$ is fixed by the tropicalized cluster mutation of each direction, $\mathbf{u}_{\ell,0}$ is the same element in $M_\R$ by the choice of $\mathbf{u}_{\ell,0}$ in~\eqref{equ_howthecenter} and hence we have $\langle \mathbf{u}_{0,0}, \mathbf{v} \rangle = \langle \mathbf{u}_{\ell,0} , \mathbf{v} \rangle$
for all $\ell \in \mathbb{N}$. According to Corollary~\ref{cor_equivalenceoftranscomm}, $\mathbf{u}_{\ell,0}$ can be expressed as
$$
\mathbf{u}_{\ell,0} = \sum_{j \in J \backslash J_\mathrm{uf}} a_j \mathbf{e}^*_j.
$$ 
For both $\mathbf{v} = \mathbf{e}_s$ and $\mathbf{e}_s - \varepsilon^\ell_{r,s} \mathbf{e}_{r}$, we have
$$
\langle \mathbf{u}_{\ell,0},\mathbf{v} \rangle = \langle \mathbf{u}_{0,0},\mathbf{v} \rangle = a_s
\in \Q \quad \mbox{for all $\ell$.}
$$

By Lemma~\ref{lemma_keyinfitnie}, we have the following supporting half-spaces of $\Delta_\ell$
$$    
H^+_{\mathbf{v}, 0} \mbox{ for $\mathbf{v} = \mathbf{e}_s$ and $\mathbf{e}_s - \varepsilon^\ell_{r,s} \, \mathbf{e}_{r}$}.
$$
By~\eqref{equ_translatedformula}, we have the following supporting half-spaces of $\nu^{-1} \cdot (\Delta_\ell - \mathbf{u}_{\ell,0})$
\begin{equation}\label{equ_esesrlsl}
    H^+_{\frac{\nu\mathbf{v}}{a_s + 0}, 1} \mbox{ for $\mathbf{v} = \mathbf{e}_s$ and $\mathbf{e}_s - \varepsilon^\ell_{r,s} \, \mathbf{e}_{r}$}.
\end{equation}
Let $\frac{\nu}{a_s} = \frac{p}{q}$ for $p \in \Z, q \in \N$ with $\mathrm{gcd}(p,q) = 1$. We emphasize that $\nu$ and $a_s$ are independent of $\ell$ and hence so is $q$. Consider the super-lattice $\frac{1}{q} N \supset N$. Because of Lemma~\ref{proposition_polardual}, the polytope $\Delta_{\ell}^\vee$ contains the line segment joining the points $\frac{\nu}{a_s}\mathbf{e}_{s}$ and $\frac{\nu}{{a_s}}(\mathbf{e}_{s} - \varepsilon^\ell_{r,s} \mathbf{e}_{r})$. Note that this line segment contains  $(1 - \varepsilon^\ell_{r,s})$ lattice points of $\frac{1}{q} N$. From this observation, $\Delta_{\ell}^\vee \subset N_\R$ has at least $(1 - \varepsilon^\ell_{r,s})$ lattice points of $\frac{1}{q} N$. Since the condition $- \varepsilon^\ell_{r,s} \to \infty$ as $\ell \to \infty$, the number of lattice points of $\Delta^\vee_{\ell}$ in $\frac{1}{q} N$ diverges to infinity. 

Let $L_{\ell} \coloneqq L_{t_\ell}$ be a monotone Lagrangian torus constructed from $\Delta_{\ell}$ in Proposition~\ref{proposition_toricdegenerationsmonotone}. We claim that the family $\{ L_{\ell} \mid \ell \in \mathbb{N} \}$ contains infinitely many \emph{distinct}  monotone Lagrangian tori. Because of Proposition~\ref{proposition_deltaveedleta}, we have
$$
\Delta_{\ell}^\vee \subset \Delta^{\mathrm{ref}}_{L_\ell}
$$
and hence the number of lattice points of $\Delta^{\mathrm{ref}}_{L_\ell}$ in $\frac{1}{q} N$ also diverges to infinity as $\ell \to \infty$. Recall that the number of lattice points of a polytope in $\frac{1}{q}N$ is invariant under unimodular equivalence. It implies that the family $\{  \Delta^{\mathrm{ref}}_{L_\ell} \mid \ell \in \mathbb{N}\}$ contains infinitely many polytopes, no two of which are related by any unimodular equivalence. Thus, Corollary~\ref{cor_polytoperef} concludes that there are infinitely many distinct monotone Lagrangian tori in $\{ L_\ell \mid \ell \in \mathbb{N} \}$ as desired.
\end{proof}

\section{Cluster polytopes for flag manifolds}\label{chap_clusterflag}

In the remaining sections, we shall apply Theorem~\ref{theorem_maininthebody} to a flag manifold to show that it carries infinitely many Lagrangian tori, no two of which are related by any symplectomorphism.
 
We set up some notations which will be used in the remaining sections. Let $G$ be a simply connected and semisimple algebraic group over $\C$ and $\g$ the corresponding Lie algebra. Let $\cmA = (a_{i,j})_{i,j\in I}$ be the Cartan matrix of $\g$. Let $B$ be a Borel subgroup of $G$, $H$ a maximal torus in $B$, $B^-$ the opposite Borel subgroup, and $U^-$ the unipotent radical of $B^-$. We then have the complex flag manifold $G/B$. 
 We denote by $W \coloneqq N(H)/H$ the \emph{Weyl group} of $\g$, where $N(H)$ is the normalizer of $H$ in $G$, and by $s_i$ the $i^{\mathrm{th}}$ simple reflection of $W$. Let $\mathfrak{h}$ be the Lie algebra of $H$ and $\mathfrak{h}^* \coloneqq \mathrm{Hom}_\C(\mathfrak{h}, \C)$ the dual space of $\mathfrak{h}$. 
We denote by $\langle- , -  \rangle$  the natural pairing between $\mathfrak{h}$ and $\mathfrak{h}^*$. We often write $ \lambda(x) \coloneqq \langle x, \lambda \rangle$ for $x \in \mathfrak{h}$ and $ \lambda \in \mathfrak{h}^* $.

Let $\{ \alpha_i \mid i\in I\}  \subset \mathfrak{h}^* $ be the set of \emph{simple roots} and $\{ h_i \mid i\in I \}\subset \mathfrak{h}$  the set of \emph{simple coroots}. Note that  $\langle h_i, \alpha_j \rangle = a_{i,j} $ for any $i,j \in I$. For $i\in I$, we define the $i^{\mathrm{th}}$ fundamental weight $\varpi_i \in  \mathfrak{h}^*$ by $\langle h_j, \varpi_i  \rangle = \delta_{j, i}$ for any $j\in I$. The \emph{weight lattice} $P$ is defined as $P \coloneqq \bigoplus_{i\in I} \Z \varpi_i$ and set $ P^+ \coloneqq \{ \lambda \in P \mid \lambda (h_i) \ge 0 \text{ for any } i\in I \}$. An element of $P^+$ is called a \emph{dominant integral weight}.

For a dominant integral weight $\lambda \in P^+$, let $V(\lambda)$ be the irreducible highest weight $G$-module over $\C$ and let $v_\lambda$ a highest weight vector of $V(\lambda)$.
For each $w\in W$, the \emph{extremal weight vector} $v_{w\lambda}$ is defined as $v_{w\lambda} \coloneqq \overline{w} v_\lambda$, where $\overline{w}$ is a lift of $w$ in $N(H)$.
For $w, u \in W$ and $g\in G$, define $\Delta_{w\lambda, u\lambda}(g) \coloneqq ( v_{w\lambda}, g v_{u \lambda}  )_\lambda$, where $(\ ,\ )_\lambda$ is the non-degenerate symmetric $\C$-bilinear form on $V(\lambda)$ with certain invariant properties. Note that $\Delta_{w\lambda, u\lambda} \in \C[G]$, which is called a \emph{generalized minor} (see \cite[Section 2.2 and Proposition 5.1]{FO20} for details). 

Let $U(\g)$ be the \emph{universal enveloping algebra} of the Lie algebra $\g$ and let $U_q(\g)$ be the $q$-deformation of $U(\g)$ over $\C(q)$, which is called the \emph{quantum group} associated with $\cmA$ (see \cite{LusztigBook}). Here $q$ is an indeterminate.
The algebra $U(\g)$ can be viewed as the specialization of $U_q(\g)$ at $q=1$. 
For $i\in I$, we denote by $f_i$ and $e_i$ the \emph{Chevalley generators} of $U_q(\g)$ with weight $ -\alpha_i$ and $\alpha_i$ respectively. Let $U_q^-(\g)$ be the \emph{negative half} of $U_q(\g)$, which is a subalgebra of $U_q(\g)$ generated by $f_i$ for all $i\in I$. 
We denote by $B(\infty)$ the \emph{infinite crystal}  of the negative half $U_q^-(\g)$ and by $\Gup(\infty)$ the \emph{dual canonical basis} (or \emph{upper global basis}) of the dual of $U_q^-(\g)$ (see \cite{KashiwaraBook, LusztigBook} and references therein). Since $B(\infty)$ is regarded as the specialization of $\Gup(\infty)$  at $q=0$, one can write 
$$
\Gup(\infty) = \{ \Gup(b) \mid b\in B(\infty) \}.
$$

The discussion of this section will be centered around the following objects. For each element $w \in W$,
\begin{itemize}
\item the \emph{Schubert variety} $X_w$ associated to $w$ by the Zariski closure of $B w B / B$ in $G/B$,
\item the \emph{unipotent cell} $U_w^-$ associated to $w$ by $U^- \cap B w B$ in $G$.
\end{itemize}

The goal of this section is to present the main result of this paper, which claims that there are infinitely many distinct monotone Lagrangian tori in every flag manifold $G/B$ of arbitrary type except a few low-dimensional cases. To construct such a family of Lagrangian tori, we employ a family of Newton--Okounkov polytopes of $G/B$ constructed by cluster algebra. We recall the cluster structure on the coordinate ring of a unipotent cell and the relationship between the dual canonical basis of the quantum group and the Newton--Okounkov polytopes. The connection leads to the proof of some conditions to apply the distinguishing criterion (Theorem~\ref{theorem_maininthebody}) to the family.

\subsection{Infinitely many Lagrangian tori in flag manifolds}

In the early stages, the construction problem for a toric degeneration of an algebraic variety was tackled by the theory of Gr\"{o}bner basis or SAGBI basis. By constructing such a basis, a toric degeneration of a flag variety $SL_n/B$ of type $A$ (and Schubert varieties therein) was constructed by Gonciulea--Lakshmibai \cite{GonciuleaLakshmibai} and Kogan--Miller \cite{KoganMiller} for instance. Indeed, the toric variety at the central fiber in this toric degeneration corresponds to the toric variety associated with the Gelfand--Zeitlin polytope. 

In \cite{Caldero}, Caldero brought a new approach to the degeneration problem on Schubert varieties based on the string parametrizations of the dual canonical basis or the upper global basis. The string parametrizations of the dual canonical basis in the irreducible representation of $G$ with the highest weight $\lambda$ were introduced by Littelmann \cite{Littelmann} and Berenstein--Zelevinsky \cite{BerensteinZelevinsky}. A \emph{string polytope} is defined by the convex hull of the string parametrizations in the irreducible representation of $G$. Caldero constructed a family of toric degenerations of a flag variety (and its Schubert varieties) of arbitrary type corresponding to a string polytope. Indeed, string polytopes are important examples of Newton--Okounkov bodies developed by \cite{Okounkov, LazarsfeldMustata, KiumarsKhovanskii}. Later on, Kaveh \cite{Kaveh} proved how string polytopes can be viewed as a Newton--Okounkov polytope by finding a suitable valuation on the functions field of the flag manifold, see Fujita~\cite{FujitaNObody} for a generalization to Schubert varieties.

To construct a string polytope, we need to choose two data. One is a dominant integral weight $\lambda$, a non-negative linear combination of fundamental weights, and the other is a reduced expression $\underline{w}$ of a Weyl group element $w \in W$. From the geometric perspective, the choice of $\lambda$ determines an ambient partial flag variety and its adorned K\"{a}hler form $\omega_\lambda$. Next, the choice of a Weyl group element $w$ determines a Schubert variety $X_w$. In particular, $X_{w_0}$ is the ambient flag manifold if $w_0$ is the longest element of $W$. Finally, the choice of a reduced expression $\underline{w}$ of the element $w$ gives rise to a valuation on $\C(X)$ arising from the sequence of (resolutions of) Schubert subvarieties of $X_w$ given by the truncations of $\underline{w}$. In sum, a string polytope associated with the choice $(\lambda, \underline{w})$ can be used to understand a toric degeneration of $X_w$. One important point is that we obtain possibly different Newton--Okounkov bodies if we make a different choice of a reduced expression $\underline{w}$ leaving the other choices $\lambda$ and $w$ fixed. In representation terminology, the reduced expression selects the order of types of the Kashiwara operators on the crystal graph so that they give rise to different string parametrizations (and they produce different convex hulls.)

Yet, there are \emph{only} finitely many reduced expressions of the longest element of $W$. Thus, we can have only finitely many distinct monotone Lagrangian tori at best. We need to find other sources for toric degenerations to produce infinitely distinct objects.

Gross--Hacking--Keel--Kontsevich \cite{GrossHackingKeelKontsevich} laid down a general framework for constructing toric degenerations via cluster algebra. This framework can be very useful to construct meaningful Lagrangian tori of a smooth projective variety thanks to the work of Harada--Kaveh \cite{HaradaKaveh}. In \cite{FO20}, Fujita--Oya showed that toric degenerations of a Schubert variety $X_w$ in a flag manifold arising from the cluster structure on the unipotent cell $U^-_w$ in Berenstein--Fomin--Zelevinsky \cite{BFZ05} are typical examples for the GHKK construction. In particular, they demonstrated how the toric degenerations can be realized as a Newton--Okounkov body on $X_w$, relying on Anderson's construction \cite{Anderson}.

An upper cluster algebra structure on the unipotent coordinate ring $\C[U_w^-]$ was discovered by Berenstein--Fomin--Zelevinsky \cite{BFZ05}. It yields that $U_w^-$ is birational to the $\mcal{A}$-cluster variety. By using the $\mcal{A}$-cluster structure, for each seed $\seed$, Fujita--Oya \cite{FO20} constructed a valuation $v_\seed$ on the function field of $\mcal{A}$. Since it is birational to the Schubert variety $X_w$, the valuation on $\C(\mcal{A})$ defines the valuation on $\C(X_w)$. Therefore each seed gives rise to a Newton--Okounkov body of the Schubert variety $X_w$. 

Moreover, the constructed Newton--Okounkov bodies can be described via the Fock--Goncharov dual $\mcal{A}^\vee$.
By the work of Qin \cite{Qin20}  it was shown that the dual canonical basis of $\C[U^-_w]$ is \emph{pointed}  and the \emph{extended $g$-vector} agrees with the valuation $v_\seed$ of a seed $\seed$. Through the relation, we can extract data of lattice points of Newton--Okounkov bodies via properties of extended $g$-vectors. In fact, for a special choice of seeds, the corresponding Newton--Okounkov polytope is unimodularly equivalent to a string polytope, and hence the family of polytopes can be thought of as a generalization of string polytopes.

With this background in mind, we choose an anticanonical regular dominant weight $\lambda = 2 \rho$ where $\rho$ is the sum of fundamental weights and the longest element $w = w_0$ of the Weyl group $W$. Hence the Schubert variety $X_w$ becomes a flag manifold equipped with a monotone K\"{a}hler form $\omega_{2 \rho}$. The toric degeneration arising from each Newton--Okounkov body can be shown to be $\mathbb{Q}$-Gorenstein Fano and normal so that for each seed $\seed$, we obtain a monotone Lagrangian torus $L_\seed$ by Proposition~\ref{proposition_toricdegenerationsmonotone} and~\ref{proposition_Qgorennormalind}. By applying Theorem~\ref{theorem_maininthebody}, we shall verify that the family contains infinitely many distinct monotone Lagrangian tori, as stated below.

\begin{theorem}[Theorem~\ref{theorem_main2}]\label{theorem_2ndmain}
Suppose that $G$ is a simply connected and semisimple complex algebraic group not of type $A_1, A_2, A_3, A_4$, and $B_2 = C_2$, that is, 
$$
G \neq \mathrm{SL}_2(\C), \mathrm{SL}_3(\C), \mathrm{SL}_4(\C), \mathrm{SL}_5(\C), \mathrm{Spin}_5(\C) = \mathrm{Sp}_4(\C).
$$
Let $\rho$ be the sum of fundamental weights and consider the flag manifold $X \coloneqq G/B$ equipped with a monotone K\"{a}hler form $\omega_{2 \rho}$. Let $\{L_\seed\}$ be the family of monotone Lagrangian tori of the flag manifold $X$ constructed by the toric degenerations from the upper cluster structure on the coordinate ring of the unipotent cell $U^-_{w_0}$. Then the family $\{L_\seed\}$ contains infinitely many monotone Lagrangian tori, no two of which are related by any symplectomorphism.
\end{theorem}

\begin{remark}
If $G = A_1$, then the flag manifold $G/B$ is the projective space $\CP^1 \simeq S^2$. In this case, every simply closed curve dividing the sphere into two pieces having the same area is Hamiltonian isotopic to a great circle. In other words, the sphere has a unique monotone Lagrangian circle. It would be interesting to see whether $G/B$ carries infinitely many distinct monotone Lagrangian tori or not when $G$ is of type $A_2, A_3, A_4$, and $B_2 = C_2$.
\end{remark}

The proof of Theorem~\ref{theorem_2ndmain} will be given in Section~\ref{Chap_exchlarge} and 
here is the outline of the proof of Theorem~\ref{theorem_2ndmain}. 

\begin{enumerate}
\item We briefly recall a construction of Newton--Okounkov polytopes of a flag manifold $G/B$ via the upper cluster structure on the coordinate ring of the unipotent cell $U^-_{w_0}$. A constructed Newton--Okounkov polytope is called a \emph{cluster polytope}. We discuss the tropical cluster structure on the family of cluster polytopes of flag manifolds, see Section~\ref{subsec_clusterdual}.
\item As an initial step, we first prove that the cluster polytope $\Delta_{t_0}$ at the initial seed is $\Q$-Gorenstein Fano, see Proposition~\ref{proposition_clusterqgornor}. Inductively, Lemma~\ref{lemma_QGorenstein} shows that every cluster polytope is also $\Q$-Gorenstein Fano. As a consequence of  Proposition~\ref{proposition_Qgorennormalind}, we have a family of infinitely many monotone Lagrangian tori in $G/B$, see Corollary~\ref{cor_familyofmonotone}.
\item  By exploiting the correspondence between the tropical integer points of a cluster polytope and the dual canonical (or upper global) basis for the negative half $U_q^-(\g)$ of the quantum group $U_q(\g)$, we shall check the conditions $(2)$ and $(3)$ for the criterion in Theorem~\ref{theorem_maininthebody}, see Section~\ref{ssec_nonnegativity}.
\item 
Section~\ref{Chap_exchlarge} confirms the remaining condition $(1)$ for Theorem~\ref{theorem_maininthebody}. We shall find a sequence of exchange matrices in the same mutation class with an arbitrarily large entry between an unfrozen variable and frozen variables. Consequently, Theorem~\ref{theorem_maininthebody} shows that the constructed family has infinitely many distinct monotone Lagrangian tori in $G/B$. 
\end{enumerate}

\subsection{Cluster polytopes and the dual canonical basis}\label{subsec_clusterdual}

In this subsection, we review the cluster algebra structure on the coordinate ring $\C[U_w^-]$ of the unipotent cell $U_w^-$ in \cite{BFZ05, Williams13, GLS11} to construct a family of Newton--Okounkov bodies of the Schubert variety $X_w$. We also recall some results on these Newton--Okounkov bodies in Fujita--Oya \cite{FO20}, which will be key ingredients for the proof of Theorem~\ref{theorem_2ndmain}.

Let $G$ be a simply connected and semisimple algebraic group over $\C$. For a reduced expression $w = s_{i_1} s_{i_2} \cdots s_{i_m} $ of $ w \in W$, we set $\supp(w) \coloneqq \{   i_1, i_2, \cdots, i_m \} \subset I$ where $I$ is the set of indices for the simple roots. Note that $\supp(w)$ does \emph{not} depend on the choice of a reduced expression of $w$. Assume that $\supp(w) = I$ for simplicity. For a dominant integral weight $\lambda$ and $u,v \in W$, we denote by $\Delta_{u\lambda, v\lambda}$ the {generalized minor} associated with $ u,v$ and $\lambda$. 
We set 
$$
D_{u \lambda, v \lambda} \coloneqq \Delta_{u \lambda, v \lambda}|_{U^-_w},
$$  
which is called a \emph{unipotent minor}. 

For a reduced expression $\rxw = s_{i_1} s_{i_2}\cdots s_{i_m}$ of $w$ and $1 \le k \le m$, we set 
\begin{align*}
w_{\le k} &\coloneqq s_{i_1} s_{i_2}\cdots s_{i_k}, 	\\
k^+ &\coloneqq \min (  \{ m+1 \} \cup \{  k+1 \le j \le m \mid i_j = i_k \}  ), \\
k^- &\coloneqq \max (  \{ 0 \} \cup \{  1 \le j \le k-1 \mid i_j = i_k \}  ).
\end{align*}
Let $J \coloneqq \{ 1,2, \ldots, m \}$,  $J_\fr \coloneqq \{j\in J \mid j^+ = m+1 \}$ and $J_\uf \coloneqq J \setminus J_\fr$.
We set 
$$
D_j \coloneqq D_{w_{\le j} \varpi_{i_j}, \varpi_{i_j}} \qquad \text{ for $1 \le j \le m$}
$$
and define the extended exchange matrix $\varepsilon_0 = (\varepsilon_{r,s})_{r \in J_{\rm uf}, s \in J}$ by 
\begin{equation} \label{eq:GLSseed}
\varepsilon_{r,s} = 
\begin{cases}
-1 & \text{ if } r = s^+, \\
-a_{i_s, i_r} & \text{ if } s < r < s^+ < r^+, \\
1 & \text{ if } r^+ = s, \\
a_{i_s, i_r} & \text{ if } r < s < r^+ < s^+, \\
0 & \text{ otherwise }
\end{cases}
\end{equation}
where $\cmA = (a_{i,j})_{i,j\in I}$ is the Cartan matrix of $G$. Note that the submatrix $(\varepsilon_{r,s})_{r \in J_{\rm uf}, s \in J_{\rm uf}}$ of $\varepsilon_0$ is skew-symmetrizable. It is skew-symmetric if and only if $\cmA$ is symmetric. 

It turns out that the coordinate ring $\C[U_w^-]$ has a cluster algebra structure. The set $D_{\rxw} \coloneqq \{ D_j \mid j =1, 2, \cdots, m \}$ together with the extended exchange matrix $\vep_0$ forms a (Fomin--Zelevinsky's) seed. Then $\C[U_w^-]$ is isomorphic to the upper cluster algebra generated by the initial seed $\seed_{0} \coloneqq ( D_{\rxw}, \vep_0)$, see Definition~\ref{def_ucaca}. Let $\mathbb{T}$ be the exchange graph associated with the cluster algebra $\C[U_w^-]$. Let $(A_{j,t})_{j \in J}$ be the cluster variables associated with $t \in \mathbb{T}$, defined in~\eqref{equ_defatjvar}. Note that $A_{j,t_0} = D_j$. 

By utilizing the above cluster algebra structure, for a fixed $t \in \mathbb{T}$, we define a valuation $v_t$ on the function field of $U^-_w$ that is isomorphic to the function field of $X_w$. Let $ \seed_t = \left((A_{j,t})_{j \in J}, \vep \right)$ be the seed of $\C[ U^-_w]$ associated with $t$. For $\bfa, \bfb \in \Z^{J}$, we write 
$$
\bfa \preceq_\vep \bfb \quad \Longleftrightarrow \quad \bfa = \bfb + \textbf{v} \vep  \text{ for some $\textbf{v} \in \Z_{\ge0}^{J_\uf}$.}
$$
The order $ \preceq_\vep$ on $\Z^J$ is called the \emph{dominance order} with respect to $\vep$ in \cite{Qin17}. We consider the Laurent polynomial ring $\mathcal{F} \coloneqq \C[ A_{j,t}^{\pm 1} \mid j \in J ]$. By identifying a Laurent monomial $ \prod_{j \in J} A_{j,t}^{a_j}$ with $\bfa = (a_1, \ldots, a_m) \in \Z^{J}$, 
we obtain the induced order $ \preceq_\vep$ on the set of Laurent monomials in $\mathcal{F}$. We denote by $ v_{\seed_t}$ the highest term valuation on $\mathcal{F}$ with respect to a total order $<_t$ refined from $\preceq_\vep$. We sometimes write $v_t$ for $v_{\seed_t}$ if no confusion arises.

For a certain class of elements in $\mathcal{F}$, the valuation can be calculated by the extended $g$-vector, which we are about to recall. Following \cite{FominZelevinsky4, FockGoncharov}, we set
$$
X_{i,t} \coloneqq \prod_{j \in J} A_{j,t}^{\varepsilon_{i,j}}.
$$
An element $f \in \mathcal{F}$ is said to be \emph{weakly pointed} at $(g_j)_{j \in J} \in \Z^J$ if $f$ can be expressed as
\begin{equation}\label{equ_extgvector}
f = \left( \prod_{j \in J} A^{g_j}_{j,t} \right) \left( \sum_{\mathbf{a} = (a_{j}) \in \mathbb{Z}^{J_\mathrm{uf}}_{\geq 0} } c_\mathbf{a} \prod_{j \in J_\mathrm{uf}} X_{j,t}^{a_j} \right)
\end{equation}
for some nonzero $c_\mathbf{a} \in \C$ with $c_\mathbf{0} \neq 0$. In this case, $g_{\seed_t}(f) \coloneqq (g_j)_{j \in J}$ is called the \emph{extended g-vector} of $f$. If $c_\mathbf{0} = 1$ in addition, then the element $f$ is called \emph{pointed}. We sometimes write $g_t$ for $g_{\seed_t}$ if no confusion arises. By \cite[Corollary 3.10]{FO20}, for every weakly pointed element $f \in \mathcal{F}$, we have
\begin{equation}\label{equ_val=g}
v_t(f) = g_t(f).    
\end{equation}

\begin{remark}
Note that the extended $g$-vector $g_t$ corresponds to $\mathbf{g}_t^{\rm L}$ defined in \cite{KK19} under the categorification using quiver Hecke algebras (see \cite[Remark C.4]{FO20}). 
\end{remark}

For a dominant integral weight $\lambda \in P^+$, we define a line bundle $\Lb_\lambda \coloneqq (G \times \C)/B$ over the flag manifold $G/B$, where $B$ acts on $G\times \C$ from the right as 
$ (g,c) \cdot b \coloneqq (gb, \lambda(b) c)$ for $g\in G, c\in \C$, and $b\in B$. Restricting to $X_w$, we obtain a line bundle on $X_w$ which is also denoted by $\Lb_\lambda$. If $\lambda$ is regular in addition, that is, $\langle h_i, \lambda \rangle >0$ for every $i\in I$, then the line bundle $\mcal{L}_\lambda$ is very ample. We fix a lowest weight vector $\tau_\lambda$ in $H^0(G/B, \Lb_\lambda)$ and restrict it to $X_w$. Following Section~\ref{subSec: NObody} and using the valuation $v_t$ on the function field $\mathcal{F} \simeq \C(X_w)$, one can produce the semigroup $S(X_w, \mcal{L}_\lambda,v_t,\tau_\lambda)$ defined in \eqref{equ_SR} and the Newton--Okounkov body $\Delta(X_w, \mcal{L}_\lambda,v_t,\tau_\lambda)$ defined by \eqref{equ_Deltav}. Let $C_{\seed_t}(w)$ be the smallest real closed cone containing $v_t ( \C[ U^- \cap X_w] \setminus \{ 0 \} )$ in $\R^{J}$, which is called the \emph{cluster cone} of $X_w$ associated with the seed $\seed_t$ of $t \in \mathbb{T}$. 

\begin{theorem}[Theorem 6.8, Corollary 6.9, Corollary 6.10 in \cite{FO20}]\label{Thm: FO1}
For each $t$ in $\mathbb{T}$, the following hold.
\begin{enumerate}
\item If a dominant integral weight $\lambda$ is regular, then $ \mcal{L}_\lambda$ is very ample.
\item $S(X_w,\mcal{L}_\lambda,v_t,\tau_\lambda)$ is finitely generated and saturated (and hence $\Delta(X_w,\mcal{L}_\lambda,v_t,\tau_\lambda)$ is a rational polytope and there exists a normal toric degeneration of $X_w$ corresponding to $\Delta(X_w,\mcal{L}_\lambda,v_t,\tau_\lambda)$ by Theorem~\ref{theorem_Andersontoric}.)
\item $C_{\seed_t}(w) \cap \Z^J = v_t ( \C[ U^- \cap X_w] \setminus \{ 0 \} ) $.
\item $ C_{\seed_t}(w) = \bigcup_{\lambda \in P^+} \Delta(X_w,\mcal{L}_\lambda,v_t,\tau_\lambda)$.
\end{enumerate}	
\end{theorem}	

\begin{definition} \label{Def: cluster polytope}
We simply call $\Delta(X_w,\mcal{L}_\lambda,v_t,\tau_\lambda)$ a \emph{cluster polytope}. If no confusion arises,  we simply write $\Delta_{\seed_t}(w, \lambda)$ for $\Delta(X_w,\mcal{L}_\lambda,v_t,\tau_\lambda)$. 
\end{definition}

To describe the relation between the Newton--Okounkov polytopes from two different choices of seeds, we consider a ``nice" basis on $\C[U^-_w]$. More precisely, the unipotent coordinate ring $\C[U^-_w]$ admits a $\C$-basis $\mcal{B}_w$ satisfying the following properties.
\begin{enumerate}
\item Every basis element in $\mcal{B}_w$ is weakly pointed for all $t \in \mathbb{T}$.
\item For each $t \in \mathbb{T}$, the map $\mcal{B}_w \to \Z^{J}$ given by $b \mapsto g_t(b)$ is injective.
\item If $t^\prime = \mu_k (t)$ for some $k \in J_{\mathrm{uf}}$ and $b \in \mcal{B}_w$, then the extended $g$-vector $g_t(b) = (g_j)_{j \in J}$ at $t$ and the extended $g$-vector $g_{t^\prime}(b) = (g^\prime_j)_{j \in J}$ at $t^\prime$ are related by the tropicalized cluster mutation $\mu_k^T$ in~\eqref{equ_tropical1}.
\item For each dominant integral weight $\lambda$, there exists a subset $\mcal{B}_w (\lambda)$ of $\mcal{B}_w$ such that $\mcal{B}_w (\lambda)$ is a $\C$-basis for the space $\{ \sigma / \tau_\lambda \mid  \sigma \in H^0(X_w, \mcal{L}_\lambda) \}$.
\end{enumerate}
Indeed, the dual canonical basis/upper global basis on the negative half ${U}^-_q(\frak{g})$ of the quantized enveloping algebra induces such a basis of $\C[U_w^-]$ via the process of specialization and localization. The induced basis on $\C[U_w^-]$ carries nice properties inherited from the dual canonical basis, the above properties $(1)-(4)$, see  Lusztig \cite{Lus90, Lus91}, Kashiwara \cite{Kas90, Kas91, Kas93}, and see also \cite{KK19, Qin20}, \cite[Appendix C]{FO20}. 

By $(1)$, every element $f$ of the dual canonical basis is weakly pointed. Because of~\eqref{equ_val=g} and $(4)$, the valuation $v_t$ of the basis element is equal to its extended $g$-vector $g_t$. From the relation $(3)$ on the extended $g$-vectors from two different choices $t$ and $t^\prime$, it follows the relation between two sets of integral points realized by $v_t$ and $v_{t^\prime}$. It in turn yields that the semigroups and cluster polytopes are related by a sequence of tropicalized cluster mutations in~\eqref{equ_hattropical1} and ~\eqref{equ_tropical1}, respectively. Here are more precise statements.

\begin{theorem}[Corollary 5.7 and 5.8 in \cite{FO20}]\label{theorem_tropicalizedclus}
If $t' = \mu_k(t)$ for some $k \in J_{\mathrm{uf}}$, then 
\begin{enumerate}
    \item the associated semigroups $S(X_w,\mcal{L}_\lambda,v_t,\tau_\lambda)$ and $S(X_w,\mcal{L}_\lambda,v_{t^\prime},\tau_\lambda)$ are related by the tropicalized cluster mutation $\widehat{\mu}_k^T$ in the $k^{\mathrm{th}}$-direction, that is, 
	\[ S(X_w,\mcal{L}_\lambda,v_{t^\prime},\tau_\lambda) = \widehat{\mu}^{{T}}_k (S(X_w,\mcal{L}_\lambda,v_t,\tau_\lambda)).
	\]
 Therefore, the family of semigroups has a tropical cluster structure.
    \item the associated cluster polytopes $\Delta_{\seed_t}(w, \lambda)$ and $\Delta_{\seed_{t^\prime}}(w, \lambda)$ are related by the tropicalized cluster mutation $\mu_k^T$ in the $k^{\mathrm{th}}$-direction, that is, 
	\[
		\Delta_{\seed_{t'}}(w, \lambda) = \mu^{{T}}_k (\Delta_{\seed_{t}}(w, \lambda)).
	\]
 Therefore, the family of cluster polytopes has a tropical cluster structure.
\end{enumerate}
\end{theorem}

\subsection{Cluster polytopes are $\mathbb{Q}$-Gorenstein Fano.}

The cluster polytopes can be thought of as a generalization of string polytopes in the following sense. Fix a dominant integral weight $\lambda$. Let $\rxw = s_{i_1} s_{i_2} \cdots s_{i_m} $ and $\seed_{t_0} \coloneqq ( D_{\rxw}, \vep_0)$ be the corresponding seed. On one hand, the seed gives rise to a cluster polytope $\Delta_{\seed_{t_0}}(w, \lambda)$ in Definition~\ref{Def: cluster polytope}. On the other hand, setting $ \bfi \coloneqq (i_1, i_2, \ldots, i_m) \in I^m$, and there is a \emph{string polytope} $\Delta_\bfi (\lambda)$ associated with $\bfi$, see \cite{BZ97, Littelmann} for the precise definition. Then they are unimodularly equivalent. 

\begin{theorem}[Corollary 6.7 in \cite{FO20}]\label{theorem_stringploytopes}
Let $\lambda$ be a dominant integral weight. Then the Newton--Okounkov body $\Delta_{\seed_{t_0}}(w, \lambda)$ is unimodularly equivalent to  the string polytope $\Delta_\bfi (\lambda)$.
\end{theorem}	

\begin{remark}
The string polytope of type $A$ can be described by the combinatorics of wiring diagrams in \cite{GleizerPostnikov}. One can associate a quiver to each wiring diagram and the braid $3$-move is compatible with a quiver mutation. Indeed, the cluster polytopes arise from the quiver mutations of an initial quiver.
\end{remark}

Throughout this subsection, we assume that $w$ is the longest element $w_0$ and hence $X_{w_0} = G/B$. The main goal is to prove that every cluster polytope $\Delta_{\seed_{t}}(w_0, \lambda)$ is a $\mathbb{Q}$-Gorenstein Fano polytope of size $1$ if we take the line bundle $\mathcal{L}_\lambda$ with the anticanonical weight $\lambda = 2 \rho$ where $\rho$ is the sum of fundamental weights.  We set $\Delta_{\seed_t}(\lambda) \coloneqq \Delta_{\seed_t}(w_0,\lambda)$ for a seed $\seed_t$ ($t \in \mathbb{T}$). As the base step for the induction, we prove that $\Delta_{\bfi} (2 \rho)$ is $\Q$-Gorenstein Fano for a particular reduced expression  $\bfi$ of $w_0$, and hence so is the corresponding $\Delta_{\seed_{t_0}}(\lambda)$ because the $\Q$-Gorenstein Fano condition is preserved under any unimodular equivalence. 

Recall that in  \cite[Corollary 2, Section 5,6,7,8, and 9]{Littelmann} a special reduced expression of $w_0$, which will be denoted by $\bfi_{\rm std}$, is chosen for each $G$, and the string polytope $\Delta_{\bfi_{\rm std}} (\lambda)$ is explicitly calculated.

\begin{lemma}\label{lemma_stringpolytoperef}
The string polytope $\Delta_{\bfi_{\rm std}} (2\rho)$ is a normalized $\mathbb{Q}$-Gorenstein Fano polytope.
\end{lemma}

\begin{proof}
In \cite{Littelmann}, Littelmann provided two sets of lattice vectors $\{{\bf{z}}_p\}_{1\le p\le m}, \{{\bf{w}}_q\}_{1\le q\le \ell}$ in $N$ such that
$$\Delta_{\bfi_{\rm std}} (2\rho) = \left(\bigcap_{p=1}^{m} H^+_{\mathbf{z}_p, 0} \right) \cap
\left(
\bigcap_{q=1}^{\ell} H^+_{\mathbf{w}_q, 2} \right).$$
Here the vectors $\mathbf{z}_p$ are the coefficients of the inequalities defining the string cone associated with $\bfi_{\rm std}$, which can be read off from the inequalities in \cite[Corollary 2, Section 5,6,7,8, and 9]{Littelmann}, and the vectors $\mathbf{w}_q$ are the coefficients of the additional inequalities in \cite[Definition on page 149]{Littelmann} to define the string polytope associated with the highest weight $2\rho$.  One may assume that each vector $\mathbf{z}_p$ is primitive, since $H^+_{\mathbf{z}_p, 0}= H^+_{a\mathbf{z}_p, 0}$ for any $a\in \R_{>0}$. Each vector $\mathbf{w}_q$ contains a component $-1$ (see \cite[Definition on page 149]{Littelmann}) so that $\mathbf{w}_q$ is primitive.

By applying the main theorem of Steinert in~\cite[Theorem on page 860]{Steinert} to $\Delta_{\bfi_{\rm std}} (2\rho)$, we conclude that it has a unique interior lattice point  $\mathbf a$ and
the polar
$(\Delta_{\bfi_{\rm std}} (2\rho) - \mathbf{a})^\circ$
is a lattice polytope.
Since $$\Delta_{\bfi_{\rm std}} (2\rho) - \mathbf{a} = \left(\bigcap_{p=1}^{m} H^+_{\frac{\mathbf{z}_p}{\langle  \mathbf{z}_p ,\mathbf{a}\rangle}, 1} \right) \cap
\left(
\bigcap_{q=1}^{\ell} H^+_{\frac{\mathbf{w}_q}{\langle  \mathbf{w}_q, \mathbf{a} \rangle +2}, 1} \right),$$
if $H_{\mathbf{z}_p, 0}$ (resp. $H_{\mathbf{w}_q, 2}$) contains a facet of $\Delta_{\bfi_{\rm std}} (2\rho) $, then 
$\frac{\mathbf{z}_p}{\langle  \mathbf{z}_p ,\mathbf{a}\rangle}$ (resp.  $\frac{\mathbf{w}_q}{\langle  \mathbf{w}_q, \mathbf{a} \rangle +2}$) 
is a lattice vector.
Since $\mathbf{z}_p$ (respectively $\mathbf{w}_q$) is primitive, we have ${\langle  \mathbf{z}_p ,\mathbf{a}\rangle}=\pm 1$
(respectively, ${\langle  \mathbf{w}_q ,\mathbf{a}\rangle}+2=\pm 1$)  and $\frac{\mathbf{z}_p}{\langle  \mathbf{z}_p ,\mathbf{a}\rangle}$ (respectively,  $\frac{\mathbf{w}_q}{\langle  \mathbf{w}_q, \mathbf{a} \rangle +2}$) 
is a primitive vector, as desired.
\end{proof}

The following lemma describes the center point of the cluster polytopes. As mentioned in \cite[Remark 4.22]{FH21}, \cite[Corollary 4.19]{FH21} holds not only for simply laced cases but also for non-simply laced cases, which in turn implies \cite[Proposition 4.20]{FH21} for non-simply laced cases.

\begin{lemma}[Theorem on page 860 in \cite{Steinert}, Proposition 4.20 in \cite{FH21}] Let
\begin{equation}\label{equ_centerfrozen}
\mathbf{u}_{0} \coloneqq \sum_{j\in J \backslash J_\mathrm{uf}} \bfe^*_j.
\end{equation}
For each $t \in \mathbb{T}$, let $\mathbf{u}_{t,0}$ be the point in~\eqref{equ_defut0u}.
Then the cluster polytope $\Delta_{\seed_t}(2\rho)$ satisfies 
\begin{enumerate}
\item the point $\mathbf{u}_{t,0}$ in~\eqref{equ_centerfrozen} is the unique interior point of $\Delta_{\seed_t}(2\rho)$,
\item the polar dual of $\Delta_{\seed_t}(2\rho) - \mathbf{u}_{t,0}$ is a lattice polytope, and
\item $\mathbf{u}_{t,0}$ is fixed under the tropicalized cluster mutation of each direction.
\end{enumerate}
\end{lemma} 

The following proposition further claims that the polar dual of $\Delta_{\seed_t}(2\rho) - \mathbf{u}_{t,0}$ is the convex hull of primitive lattice vectors. It in turn yields that $\Delta_{\seed_t}(2\rho)$ is a normalized $\mathbb{Q}$-Gorenstein Fano polytope. 

\begin{proposition}\label{proposition_clusterqgornor}
For each $t \in \mathbb{T}$, the cluster polytope $\Delta_{\seed_t}(2\rho)$ is a normalized $\mathbb{Q}$-Gorenstein Fano polytope with the center $\mathbf{u}_{t,0}$ and the associated toric degeneration is normal.
\end{proposition}

\begin{proof}
For each seed $\seed_t$ ($t \in \mathbb{T}$), let $S_{t} \coloneqq S(X_{w_0},\mcal{L}_{2 \rho},v_t,\tau_{2\rho})$ be the semigroup associated to $v_t$. By Theorem~\ref{Thm: FO1}, each semigroup $S_t$ is finitely generated and saturated. Hence, the toric degeneration associated with the corresponding Newton--Okounkov polytope $\Delta_{\seed_t}(2 \rho)$ is normal. By Theorem \ref{theorem_stringploytopes}, $\Delta_{\seed_{t_0}}(2\rho)$ is a normalized $\mathbb{Q}$-Gorenstein Fano polytope with the center $\mathbf{u}_{t_0,0}$. By Lemma~\ref{lemma_QGorenstein} and Corollary~\ref{cor_equivalenceoftranscomm}, so are all the other cluster polytopes.
\end{proof}

\begin{corollary}\label{cor_familyofmonotone}
For each $t \in \mathbb{T}$, consider a completely integrable system $\Phi_t \colon G/B \to \Delta_{\seed_t}(2\rho)$ constructed from Theorem~\ref{theorem_HaradaKaveh}. Then the fiber $\Phi_t^{-1}(\mathbf{u}_{t,0})$ at the center is a monotone Lagrangian torus.
\end{corollary}

\begin{proof}
The line bundle $\mcal{L}_{2\rho}$ is very ample and anticanonical and hence the K\"{a}hler form inherited from $\mcal{L}_{2\rho}$ is monotone. The statement follows from Proposition~\ref{proposition_clusterqgornor} and Theorem~\ref{proposition_toricdegenerationsmonotone}.
\end{proof}

\subsection{Nonnegativity of  the exponents of the frozen variables}\label{ssec_nonnegativity}

The goal of this subsection is to verify the conditions $(2)$ and $(3)$ in Theorem~\ref{theorem_maininthebody} in the case of cluster polytopes. Throughout this subsection, we assume that $w$ is the longest element $w_0$ and hence $X_{w_0} = G/B$. Set $\Delta_{\seed_t}(\lambda) \coloneqq \Delta_{\seed_t}(w_0,\lambda) $ 
and $C_{\seed_t} \coloneqq C_{\seed_t}(w_0) $ for a seed $\seed_t$ ($t \in \mathbb{T}$). 

Let $B(\infty)$ be the infinite crystal of the negative half $U_q^-(\g)$ of the quantum group $U_q(\g)$. We consider the dual canonical basis (or upper global basis) 
$$
\Gup(\infty) = \{ \Gup(b) \mid b\in B(\infty) \},
$$
where $\Gup(b)$ is the element corresponding to $b$ in $B(\infty)$. Then the specialization $\Gupq(\infty) \coloneqq \{ \Gupq(b) \mid b\in B(\infty) \}$ of $\Gup(\infty)$ at $q=1$ forms a basis of the coordinate ring $\C[U^-]$. 
Since the dual canonical basis $ \Gup(\infty)$ is pointed (\cite{Qin20}), \eqref{equ_val=g} yields
$$
 v_{\seed_t} (b) = g_{\seed_t} (b) \qquad \text{ for any $b \in \Gupq(\infty)$.}
$$
Thanks to Theorem \ref{Thm: FO1}, we have 
\begin{align} \label{Eq: CtZ}
C_{\seed_t} \cap \Z^J = \{ g_{\seed_t} ( b ) \mid b \in \Gupq(\infty) \}.
\end{align}

The unipotent coordinate ring $\C[U^-]$ is a polynomial ring,  and it is equipped with a cluster algebra structure on which the frozen variables are \emph{not} invertible, see  \cite[Theorem 3.3]{GLS11} and \cite{GY17, GY20}. Indeed, for each reduced expression $\rxw_0 = s_{i_1} s_{i_2}\cdots s_{i_\ell}$, there exists a seed  of $\C[U^-]$ consisting of the cluster variables 
\begin{equation} \label{eq:seedU-}
    \left\{\Delta_{w_{\le k} \varpi_{i_k}, \varpi_{i_k}}\vert_{U^{-}} \mid 1\le k\le \ell \right\}
\end{equation}
with the same set of indices of frozen variables and the same exchange matrix \eqref{eq:GLSseed} given in the previous section. 

On the other hand, we have
$$U^-_{w_0} = U^- \cap \mathbf O_{w_0},$$
where $$\mathbf O_{w_0} = \{ g\in U^- \mid \Delta_{w_0 \varpi_{i}, \varpi_{i}}(g)\neq 0  \quad \text{for} \ i\in I \},
$$
see \cite{BZ97}, \cite[Propositions 8.4 and 8.5]{GLS11} and see also \cite[Proposition 2.20]{KO21}. Hence, the coordinate ring $\C[U^{-}_{w_0}]$ is the localization of $\C[U^{-}]$ at $\{\Delta_{w_0\varpi_i, \varpi_i} \vert_{U^-} \mid i \in I\}$, which is the set of frozen variables of the cluster algebra $\C[U^{-}]$. It follows that the cluster $\{A_j \mid j \in J\}$ of a seed $\seed_t$ of $\C[U^-_{w_0}]$ is also a cluster of $\C[U^-]$. Hence any element in $\C[U^-] \subset \C[U^-_{w_0}]$ can be written as a Laurent polynomial in the cluster variables in $\{A_j \mid j \in J\}$, in which the exponents of the frozen variables in any monomials are nonnegative by the Laurent phenomenon (\cite{FZ02}). Since $\Gupq(\infty) \subset \C[U^-] $, we deduce the following proposition. 

\begin{proposition}\label{proposition_frozen ge 0}
For every $b \in \Gupq(\infty)$ and for each frozen index $k \in J_\fr$, the $k^{\mathrm{th}}$-component of the extended $g$-vector of $b$ is nonnegative, that is, 
\begin{align} \label{Eq: frozen ge 0}
(g_{\seed_t}(b))_k \ge 0. 
\end{align}
\end{proposition}

For a dominant integral weight $\lambda$, let $B(\lambda)$ be the crystal of the irreducible highest weight $U_q(\g)$-module $V(\lambda)$ and let $\Gup(\lambda)$ be the \emph{dual canonical basis} (or \emph{upper global basis}) of $V(\lambda)$. We regard $\Gup(\lambda)$ as a subset of $\Gup(\infty)$ via the dual map of the natural projection $\pi \colon U_q^-(\g) \twoheadrightarrow V(\lambda) $ sending $1$ to the highest weight vector $v_\lambda$ of $V(\lambda)$. By \cite[Thoerem 6.8 (2), Corollary 6.10]{FO20}, we have  
\begin{align} \label{eq:integral_points}
\Delta_{\seed_t}(\lambda) \cap \Z^J = \{ g_{\seed_t} ( b ) \mid b \in \Gupq(\lambda) \}.
\end{align}
 
\begin{proposition}\label{lem:frozen supp}
 Let $\seed_t $ be a seed for $t\in \mathbb{T}$, and let $ k \in J_{\rm fr} $. We consider 
 $$
 H_{\bfe_k, 0} = \{ \bfu \in \R^J \mid \langle \bfu, \bfe_k \rangle  = 0 \} \mbox{ and } H^+_{\bfe_k, 0} = \{ \bfu \in \R^J \mid \langle \bfu, \bfe_k \rangle \geq 0 \}, 
 $$ where $\bfe_j \coloneqq (\delta_{i, j})_{i\in J}$ is the standard unit vector for $j\in J$. Then the following hold.
\begin{enumerate}
\item The half-space $H^+_{\bfe_k, 0}$ contains the cluster cone $C_{\seed_t}$ and the hyperplane $ H_{\bfe_k, 0}$ is a supporting hyperplane of $C_{\seed_t}$. Moreover, $ H_{\bfe_k, 0}$ contains a facet of $C_{\seed_t}$. 
\item For any dominant integral weight $\lambda$, the half-space $H^+_{\bfe_k, 0}$ contains the cluster polytope $\Delta_{\seed_t}(\lambda)$ and $ H_{\bfe_k, 0}$ is a supporting hyperplane of $\Delta_{\seed_t}(\lambda)$.
\end{enumerate}
\end{proposition}

\begin{proof}
(1) Let $\seed_t = (  \{ A_j \}_{j\in J}, \vep )$. Since the cluster variable $A_j$ ($j\in J$) is contained in $\Gupq(\infty)$ (see \cite{McN21, Qin20}), by the definition of the extended $g$-vector $g_{\seed_t}$ and \eqref{Eq: CtZ}, we have 
$$
g_{\seed_t}(A_j) = \bfe_j \in C_{\seed_t} \qquad \text{ for any $j\in J$.}
$$
Since $ C_{\seed_t}$ is a convex polyhedral cone,  we have 
\begin{align} \label{Eq: faset}
F_{ k} \coloneqq \Span_{\R_{\ge0}} ( B \setminus \{ \bfe_k \} ) \subset C_{\seed_t} \cap H_{\bfe_k, 0} \qquad \text{for any frozen index $k \in J_\fr$,}
\end{align}
where we set $B \coloneqq \{ \bfe_j \mid j\in J \}$. Note that $F_{k}$ has codimension one.
It follows from \eqref{Eq: frozen ge 0} and \eqref{Eq: faset} that 
$$
C_{\seed_t} \subset H_{\bfe_k, 0}^+,
$$
and $H_{\bfe_k, 0}$ contains a facet $F_k$ of $C_{\seed_t}$.

(2) It follows from $ 0 \in \Delta_{\seed_t}(\lambda) \cap H_{\bfe_k, 0} $ and $ \Delta_{\seed_t}(\lambda) \subset C_{\seed_t} \subset H_{\bfe_k, 0}^+$.
\end{proof}	

Note that Proposition~\ref{lem:frozen supp} confirms that the conditions $(2)$ and $(3)$ for the criterion in Theorem~\ref{theorem_maininthebody} in the case of cluster polytopes.

\section{Exchange matrix with a large entry between  a frozen and unfrozen}\label{Chap_exchlarge}

The aim of this section is to prove the existence of exchange matrices with an arbitrarily large entry between an unfrozen variable and frozen variables and complete the proof of Theorem~\ref{theorem_2ndmain}.

We say that two extended exchange matrices are \emph{mutation equivalent} if one is obtained by a sequence of mutations of the other. It is an equivalence relation.

\begin{proposition} \label{prop:ABCDEFG}
Let $\g$ be a complex simple Lie algebra of type other than $A_1, A_2, A_3, A_4$ or $B_2=C_2$.
Then there is $s\in J \backslash J_\mathrm{uf}$ such that the extended exchange matrix $\varepsilon_0 = (\varepsilon_{i,j})_{i\in J_\uf, j\in J}$ in~\eqref{eq:GLSseed} of the cluster structure of $\C[U^{-}_{w_0}]$ satisfies the following property
\begin{equation} \label{eq:manyarrows2}
    \text{\parbox{35em}{
 for any $\ell \ge 0$, there exists a matrix $\varepsilon_{\ell}=(\varepsilon^{\ell}_{i,j})_{ i\in J_\uf, j\in J}$ which is  mutation equivalent to $\varepsilon_0 = (\varepsilon_{i,j})_{i\in J_\uf, j\in J }$ such that  
$- \varepsilon^{\ell}_{r_\ell,s} \ge \ell$ for some $r_\ell \in J_\uf$.}}
\end{equation}   
\end{proposition}

Now, assuming Proposition~\ref{prop:ABCDEFG}, we wrap up the proof of Theorem~\ref{theorem_2ndmain}.

\begin{proof}[Proof of Theorem~\ref{theorem_2ndmain}]
By Corollary~\ref{cor_familyofmonotone}, we have the family of infinitely many monotone Lagrangian tori on $G/B$. We can apply Theorem~\ref{theorem_maininthebody} to this family because of Theorem~\ref{theorem_tropicalizedclus}, Proposition~\ref{lem:frozen supp}, and Proposition~\ref{prop:ABCDEFG}. Therefore, the family contains infinitely many monotone Lagrangian tori, no two of which are related by any symplectomorphism. 
\end{proof}

Let $J=J_\uf \sqcup J_\fr$ be a finite set and  $\varepsilon \coloneqq (\varepsilon_{i,j})_{i\in J_\uf, j\in J}$ be an extended skew-symmetrizable matrix. 
We say that $\varepsilon$ is \emph{mutation finite} if the mutation equivalence class $[\varepsilon]$ of $\varepsilon$ is finite. If  $[\varepsilon]$ is infinite, then we say that $\varepsilon$ is \emph{mutation infinite}.

For a subset $J' \subset J$, the restriction $\varepsilon\vert_{J'}$ of $\varepsilon$ is the matrix obtained from $\varepsilon$ by restricting to the column set $J'$ and to the row set $J'\cap J_\uf$. Note that if $k\in J'$ is mutable, then we have $\mu_k(\varepsilon\vert_{J'}) = (\mu_k(\varepsilon))\vert_{J'}.$ That is, the restriction commutes with the mutation.

\begin{lemma} \label{lem:finiteinfinite}
Let $J=J_\uf \sqcup J_\fr$ be a finite set and  $\varepsilon \coloneqq (\varepsilon_{i,j})_{i\in J_\uf, j\in J}$ be an extended skew symmetrizable matrix. 
Assume that $\varepsilon$ is mutation infinite and  $\varepsilon\vert_{J_\uf}$ is mutation finite.
Then $\varepsilon$ satisfies~\eqref{eq:manyarrows}.
\begin{equation} \label{eq:manyarrows}
    \text{\parbox{35em}{
 for any $\ell \ge 0$, there exists a matrix $\varepsilon_{\ell}=(\varepsilon^{\ell}_{i,j})_{i\in J_\uf, j\in J}$ which is  mutation equivalent to $(\varepsilon_{i,j})_{i\in J_\uf, j\in J}$ such that  
$- \varepsilon^{\ell}_{r,s} \ge \ell$ for some $r\in J_\uf$ and $s\in J_\fr$.}}
\end{equation}
Note that if a matrix satisfies \eqref{eq:manyarrows}, then so does any matrix mutation equivalent to it. 
\end{lemma}
\begin{proof}
Let $\varepsilon=\varepsilon_0,\varepsilon_1,\varepsilon_2,\varepsilon_3,\ldots$ be an infinite sequence of pairwise different matrices in the mutation equivalence class of $\varepsilon$.
Since $\varepsilon|_{J_\uf}$ is mutation finite, 
 the set $\{ \varepsilon_n|_{J_\uf} \mid n \ge 0\}$ is finite.
It follows that there exists $m$ such that $\varepsilon_m|_{J_\uf}$ appears infinitely many times in the sequence $(\varepsilon_n|_{J_\uf} \mid n \ge 0)$. Hence there is an infinite sequence $m=p_0< p_1<p_2<\cdots$ such that $\varepsilon_{p_k}|_{J_\uf}=\varepsilon_m|_{J_\uf}$ for all $k\ge 0$.  
Set $M_k \coloneqq \max\{|\varepsilon^{p_k}_{r,s}| \mid  r\in J_\uf, s\in J_\fr   \}$. 
Then $M_k$ has no upper bound since $\varepsilon_m=\varepsilon_{p_0},\varepsilon_{p_1},\varepsilon_{p_2},\varepsilon_{p_3},\ldots$ is an infinite sequence of pairwise distinct matrices. We may assume that there are infinitely many $k$'s such that $- \varepsilon^{p_k}_{r,s} = M_k$ for some $r \in J_\uf, s\in J_\fr$. Otherwise, we can change the sign of $\varepsilon^{p_k}_{r,s}$ by mutating the exchange matrix $\varepsilon_{p_k}$ in the $r^\mathrm{th}$-direction. Hence  $\varepsilon_m$ satisfies \eqref{eq:manyarrows}.
\end{proof}

Recall that if $\varepsilon$ is an extended skew-symmetric matrix, then one can associate a quiver $Q=Q(\varepsilon)$ with vertices $J=J_\uf\sqcup J_\fr$ 
given by
\begin{equation*}
\text{there are $\varepsilon_{j,i}$-many arrows from $i$ to $j$ whenever  $\varepsilon_{j,i}\ge 0$}.
\end{equation*}

To prove Proposition~\ref{prop:ABCDEFG}, we make use of the classification results of Felikson--Tumarkin on mutation finite quivers. The following is a half of \cite[Theorem 1.11]{FT21} when $\varepsilon$ is an extended skew-symmetric matrix.

\begin{lemma} \label{lem:FTcrit}
Let $Q$ be a quiver associated with an extended skew-symmetric matrix. Assume that $J_\fr=\{f\}$ and  $Q|_{J_\uf}$ is mutation finite. If there exists $Q'$ mutation equivalent to $Q$ such that 
\begin{enumerate}
    \item $Q'$ contains a double arrow $v_1 \rightrightarrows v_2$, 
    \item  either $b_1\neq -b_2$ or $b_2<0$, where 
    $$b_i = (\text{the number of arrows from $v_i$ to $f$}) -  \, (\text{number of arrows from $f$ to $v_i$}), $$
\end{enumerate}
then $Q$ is mutation infinite.
\end{lemma}

\begin{remark}
Indeed, there is an analogous version of Lemma~\ref{lem:FTcrit} when $\varepsilon$ is extended \emph{skew-symmetrizable} matrix and $Q$ is the \emph{diagram} associated with  $\varepsilon$, see \cite[Section 10]{FT21}.
\end{remark}

\begin{lemma} 
\label{lem:Apqf} 
Let $Q$ be of type $\tilde A_{p,q}$, that is, the underlying graph is $(p+q)$-gon, there are $p$-many clockwise arrows and $q$-many counterclockwise arrows, and all the vertices of $Q$ are mutable. Assume that $p>0$ and $q>0$. Then the following hold.

\begin{enumerate}
    \item For any vertex $a$ of $Q$, there exists a sequence of mutations at vertices different from $a$ such that the resulting quiver contains a double arrow  $a \rightrightarrows v$ for some vertex $v$. 
    \item Let $Q'$ be the quiver obtained from $Q$ by adding a frozen vertex $f$ with  arrows connecting $f$  to a single vertex $a$ in $Q$.
Then $Q'$ is mutation-infinite. 
    \item For any $\ell\ge 0$, there is a quiver $Q^\ell$ mutation equivalent to $Q'$ such that there exists a mutable vertex $r_\ell$ with arrows more than $\ell$ connecting to $f$.
\end{enumerate}
\end{lemma}
\begin{proof}
It is well-known that $Q$ is mutation finite.

(1) Let $J$ be the set of vertices of $Q$. 
Then there is a sequence of  mutations  of $Q$ at vertices $J\setminus \{a\}$, all of which are mutations at a sink or a source to obtain the following quiver: 
$$
\xymatrix@R=1.2em
{
 & & v \ar[dl] \ar[r]& c_{p-1} \ar[dr]& & \\
 & b_{q-1} \ar[d]& & &c_{p-2}  \ar[dr]  & \\
 & & & & & c_{p-3}  \ar[d] \\
& & & & & \\
 & \vdots & & & &  \vdots \\
& b_3 \ar[d]& & & &  c_3  \ar[d] \\
& b_2 \ar[dr] & & & &  c_2  \ar[d]\\
& & b_1 \ar[dr] & & & c_1 \ar[dll] \\
 & & & a & & \\
}
$$
 That is, $a$ is a sink, $v$ is a source,  and there are $q$-many counterclockwise arrows (respectively $p$-many clockwise arrows)  between $a$ and $v$.
Then the sequence of mutations $$\left(\mu_{c_{p-1}} \circ\mu_{c_{p-2}} \circ \cdots \mu_{c_2} \circ \mu_{c_1}  \right) \circ\left(\mu_{b_{q-1}} \circ\mu_{b_{q-2}} \circ \cdots \mu_{b_2} \circ \mu_{b_1}  \right) $$
yields the desired quiver.

(2) Let $\mu$ be a mutation sequence in $(1)$. Then $\mu(Q')$ contains a double arrow  $a \rightrightarrows v$ and the vertex $a$ is the only vertex connected to $f$.
Hence we have $b_a>0$ and $b_v=0$ so that $b_a\neq -b_v$. Hence $Q'$ is mutation infinite by Lemma \ref{lem:FTcrit}.

(3) It follows from Lemma \ref{lem:finiteinfinite} that $Q'$ has the desired property since $Q'$ has only one frozen vertex.
\end{proof}

\begin{corollary}\label{cor:res}
    Let $Q$ be a quiver associated with an extended skew-symmetric matrix $\varepsilon$. Assume that  there is a full subquiver $Q'$ of $Q$ such that
\begin{enumerate}
    \item the mutable part of $Q'$ is of type $\widetilde A_{p,q}$ with $p,q>0$, and
    \item $Q'$ contains only one frozen vertex $s$ and there is only one unfrozen vertex connected to $s$.
\end{enumerate}
Then we have the property~\eqref{eq:manyarrows2}
\end{corollary}
\begin{proof}
As the restriction commutes with the mutation, \eqref{eq:manyarrows2} follows from Lemma~\ref{lem:Apqf} (3).
\end{proof}

We are ready to prove Proposition~\ref{prop:ABCDEFG}.

\begin{proof}[Proof of Proposition~\ref{prop:ABCDEFG}]
If the Dynkin diagram $D$ of $\g$ contains another Dynkin diagram $D'$ as a full subgraph, then the longest element $w_0$ of the Weyl group of $\g$ can be decomposed into 
$$w_0=vw^\prime \quad \text{ where $w^\prime$ is the longest element of the parabolic subgroup of type $D'$}.$$ 
Let $\underline{w}^\prime$ be a reduced expression of $w'$ and let $\underline{w}_0$ be a reduced expression of $w_0$ extending $\underline{w}^\prime$ with respect to the above decomposition. Then, by \eqref{eq:GLSseed}, the matrix attached to $\underline{w}^\prime$ is a restriction of the one attached to $\underline{w}_0$. Moreover, the frozen indices are preserved in the restriction. Hence one may assume that $\g$ is of type $A_5$,  $D_4$, $C_3$, $B_3$, or $G_2$. 
We deal with the five cases in Lemma~\ref{lemma_A5}, \ref{lemma_D4}, \ref{lemma_C3}, \ref{lemma_B3}, and \ref{lemma_G2}.

For type $A_5$ and $D_4$, it is enough to find a reduced expression $\underline{w}_0$ and a full subquiver of the initial quiver associated with $\underline{w}_0$ satisfying the conditions in Corollary  \ref{cor:res}. For type $C_3, B_3,$ and $G_2$, which are non-skew-symmetric types, we will find a reduced expression $\underline{w}_0$ and a submatrix of the initial matrix associated with $\underline{w}_0$ such that it is mutation infinite but its mutable part is mutation finite. 
Such examples can be found using \cite[Theorem 9.4]{FT21}, which is analogous to Lemma \ref{lem:FTcrit} in non-skew-symmetric cases. Then by Lemma \ref{lem:finiteinfinite},  we get the desired property \eqref{eq:manyarrows} and hence \eqref{eq:manyarrows2}. 
\end{proof}

\begin{lemma}\label{lemma_A5}
If $\g$ is of type $A_5$, then \eqref{eq:manyarrows2} holds.
\end{lemma}

\begin{proof}
For type $A_5$, take $$\underline{w}_0 =s_1s_2s_1s_3s_2s_1s_4s_3s_2s_1s_5s_4s_3s_2s_1$$
where the Dynkin diagram is 
$$\begin{tikzpicture}[scale=0.5]
\draw (-1,0) node[anchor=east] {$A_{5}$};
\draw (0 cm,0) -- (8 cm,0);
\draw[fill=white] (0 cm, 0 cm) circle (.25cm) node[below=4pt]{$1$};
\draw[fill=white] (2 cm, 0 cm) circle (.25cm) node[below=4pt]{$2$};
\draw[fill=white] (4 cm, 0 cm) circle (.25cm) node[below=4pt]{$3$};
\draw[fill=white] (6 cm, 0 cm) circle (.25cm) node[below=4pt]{$4$};
\draw[fill=white] (8 cm, 0 cm) circle (.25cm) node[below=4pt]{$5$};
\end{tikzpicture}
$$
The quiver attached to $\underline{w}_0$ is
$$
\xymatrix{
1 \ar[d]& 3 \ar[d]  \ar[l]& 6  \ar[d] \ar[l]& 10 \ar[d] \ar[l] &  \fbox{15}\ar[l] \\
2 \ar[d] \ar[ru]& 5 \ar[ru] \ar[d] \ar[l]& 9 \ar[ru] \ar[d] \ar[l] &  \fbox{14} \ar[l]& \\
4 \ar[d]\ar[ru]& 8 \ar[ru]\ar[d]  \ar[l]&  \fbox{13}\ar[l]&  & \\
7 \ar[ru]\ar[d]&  \fbox{12} \ar[l]& &  & \\
 \fbox{11} & & &  &
}
$$
and the full subquiver with  the vertices $14,9,8,4,2,3$ and $6$ satisfies the condition in Corollary \ref{cor:res}. Here the boxed entries are frozen vertices. \end{proof}

\begin{lemma}\label{lemma_D4}
If $\g$ is of type $D_4$, then \eqref{eq:manyarrows2} holds.
\end{lemma}

\begin{proof}
For the case $\g=D_4$, 
take
$$\underline{w}_0 =s_2s_4s_1s_2s_4s_3s_2s_4s_1s_2s_3s_4$$
where the Dynkin diagram is
$$\begin{tikzpicture}[scale=0.5]
\draw (-1,0) node[anchor=east] {$D_{4}$};
\draw (0 cm,0) -- (2 cm,0);
\draw (2 cm,0) -- (4 cm,0.7 cm);
\draw (2 cm,0) -- (4 cm,-0.7 cm);
\draw[fill=white] (0 cm, 0 cm) circle (.25cm) node[below=4pt]{$1$};
\draw[fill=white] (2 cm, 0 cm) circle (.25cm) node[below=4pt]{$2$};
\draw[fill=white] (4 cm, 0.7 cm) circle (.25cm) node[right=3pt]{$4$};
\draw[fill=white] (4 cm, -0.7 cm) circle (.25cm) node[right=3pt]{$3$};
\end{tikzpicture}$$
The quiver attached to $\underline{w}_0$ is 
$$
\xymatrix{
3 \ar[drr] & \fbox{9} \ar[l]&  \\
 1 \ar@/_1pc/[dd] \ar[u]&4  \ar@/^1pc/[dd] \ar[dl] \ar[l]& 7 \ar[l]   \ar@/_1pc/[dd] \ar[ul]& \fbox{10}  \ar[l] \\
6\ar[urrr]&  \fbox{11} \ar[l]  \\
2  \ar@/^2pc/[uur]& 5 \ar@/_1pc/[uur] \ar[l]& 8 \ar[uur] \ar[l]&  \fbox{12} \ar[l]
}
$$
and the full subquiver with the vertices 
$10, 7, 3,1,2$ and $5$ satisfies the condition in Corollary \ref{cor:res}.
\end{proof}

\begin{lemma}\label{lemma_C3}
If $\g$ is of type $C_3$, then \eqref{eq:manyarrows2} holds.
\end{lemma}

\begin{proof}
In the case $\g =C_3$, the Dynkin diagram and Cartan matrix are given by
$$
\begin{tikzpicture}[scale=0.5]
\draw (-1,0) node[anchor=east] {$C_{3}$};
\draw (0 cm,0) -- (2 cm,0);
\draw (2 cm, 0.1 cm) -- +(2 cm,0);
\draw (2 cm, -0.1 cm) -- +(2 cm,0);
\draw[shift={(2.8, 0)}, rotate=180] (135 : 0.45cm) -- (0,0) -- (-135 : 0.45cm);
\draw[fill=white] (0 cm, 0 cm) circle (.25cm) node[below=4pt]{$1$};
\draw[fill=white] (2 cm, 0 cm) circle (.25cm) node[below=4pt]{$2$};
\draw[fill=white] (4 cm, 0 cm) circle (.25cm) node[below=4pt]{$3$};
\end{tikzpicture}
\qquad
\left(\begin{array}{rrr}
2 & -1 & 0 \\
-1 & 2 & -2 \\
0 & -1 & 2
\end{array}\right).
$$
Take
$$\underline w_0=s_3s_2s_3s_2s_1s_2s_3s_2s_1$$

Then
$$\varepsilon=\left(\begin{array}{rrrrrrrrr}
0 & -2 & 1 & 0 & 0 & 0 & 0 & 0 & 0 \\
1 & 0 & -1 & 1 & 0 & 0 & 0 & 0 & 0 \\
-1 & 2 & 0 & 0 & 0 & -2 & 1 & 0 & 0 \\
0 & -1 & 0 & 0 & -1 & 1 & 0 & 0 & 0 \\
0 & 0 & 0 & 1 & 0 & 0 & 0 & -1 & 1 \\
0 & 0 & 1 & -1 & 0 & 0 & -1 & 1 & 0
\end{array}\right)$$
Let
$$J'=\{1,2,3,6,8\}$$
Then we have
\begin{equation}\label{eq:C3_restriction}
\varepsilon\vert_{J'}
=\left(\begin{array}{rrrrr}
0 & -2 & 1  &  0  &  0 \\
1 & 0  & -1 &  0  &  0 \\
-1& 2  & 0 &  -2  &  0 \\
0 & 0  & 1  &  0  &1
\end{array}\right).
\end{equation}
Then by taking a sequence of mutation 
we get
$$\mu_3\circ \mu_2 \circ \mu_6 (\varepsilon\vert_{J'}) = \left(\begin{array}{rrrrr}
0 & 0 & 1 & 0 & 0 \\
0 & 0 & -1 & 2 & 0 \\
-1 & 2 & 0 & -2 & 0 \\
0 & -2 & 1 & 0 & -2 
\end{array}\right)$$
Note that the corresponding \emph{diagram} (see \cite{FZ03}, \cite[Section 9.1]{FT21} for the definition of diagram associated with an extended skew-symmetrizable matrix with one frozen index) is 
\begin{equation} \label{eq:affineC3diagram}
\xymatrix{ & 2 \ar[d]_2& & \\
1 & 3 \ar[r]_2 \ar[l]& 6 \ar[r]   \ar@<-0.5ex>[ul]\ar@<0.5ex>[ul]  & \fbox{8}}. 
\end{equation}
The mutable part of the diagram is of type $\widetilde C_3$ (see \cite[Figure 9.2]{FT21}), which is mutation finite. By applying  \cite[Theorem 9.4]{FT21} (which is an analog of Lemma \ref{lem:FTcrit} for skew-symmetrizable cases), the matrix $\varepsilon\vert_{J'}$ is mutation infinite and hence the matrix $\varepsilon$ has the desired property by Lemma \ref{lem:finiteinfinite}.
\end{proof}

\begin{lemma}\label{lemma_B3}
If $\g$ is of type $B_3$, then \eqref{eq:manyarrows2} holds.
\end{lemma}

\begin{proof}
In the case $\g =B_3$, the Dynkin diagram and Cartan matrix are given by
$$\begin{tikzpicture}[scale=0.5]
\draw (-1,0) node[anchor=east] {$B_{3}$};
\draw (0 cm,0) -- (2 cm,0);
\draw (2 cm, 0.1 cm) -- +(2 cm,0);
\draw (2 cm, -0.1 cm) -- +(2 cm,0);
\draw[shift={(3.2, 0)}, rotate=0] (135 : 0.45cm) -- (0,0) -- (-135 : 0.45cm);
\draw[fill=white] (0 cm, 0 cm) circle (.25cm) node[below=4pt]{$1$};
\draw[fill=white] (2 cm, 0 cm) circle (.25cm) node[below=4pt]{$2$};
\draw[fill=white] (4 cm, 0 cm) circle (.25cm) node[below=4pt]{$3$};
\end{tikzpicture} \qquad
\left(\begin{array}{rrr}
2 & -1 & 0 \\
-1 & 2 & -1 \\
0 & -2 & 2
\end{array}\right)
$$
Take
$$\underline w_0=s_3s_2s_3s_2s_1s_2s_3s_2s_1.$$
Then we have

$$\varepsilon =\left(\begin{array}{rrrrrrrrr}
0 & -1 & 1 & 0 & 0 & 0 & 0 & 0 & 0 \\
2 & 0 & -2 & 1 & 0 & 0 & 0 & 0 & 0 \\
-1 & 1 & 0 & 0 & 0 & -1 & 1 & 0 & 0 \\
0 & -1 & 0 & 0 & -1 & 1 & 0 & 0 & 0 \\
0 & 0 & 0 & 1 & 0 & 0 & 0 & -1 & 1 \\
0 & 0 & 2 & -1 & 0 & 0 & -2 & 1 & 0
\end{array}\right)$$
Let
$$J'=\{1,2,3,6,8\}$$
Then we have
$$\varepsilon\vert_{J'}
=\left(\begin{array}{rrrrr}
0 & -1 & 1  &  0  &  0 \\
2 & 0  & -2 &  0  &  0 \\
-1& 1  & 0 &  -1  &  0 \\
0 & 0  & 2  &  0  & 1
\end{array}\right)
$$
and we have
$$\mu_3(\varepsilon\vert_J')
=\left(\begin{array}{rrrrr}
0 & 0 & -1 & 0 & 0 \\
0 & 0 & 2 & -2 & 0 \\
1 & -1 & 0 & 1 & 0 \\
0 & 2 & -2 & 0 & 1 \\
\end{array}\right)
$$
Note that the corresponding diagram is again \eqref{eq:affineC3diagram} so that we have the desired property.
\end{proof}

\begin{lemma}\label{lemma_G2}
If $\g$ is of type $G_2$, then \eqref{eq:manyarrows2} holds.
\end{lemma}

\begin{proof}
In the case of $G_2$ the Dynkin diagram and Cartan matrix are
$$\begin{tikzpicture}[scale=0.5]
\draw (-1,0) node[anchor=east] {$G_2$};
\draw (0,0) -- (2 cm,0);
\draw (0, 0.15 cm) -- +(2 cm,0);
\draw (0, -0.15 cm) -- +(2 cm,0);
\draw[shift={(0.8, 0)}, rotate=180] (135 : 0.45cm) -- (0,0) -- (-135 : 0.45cm);
\draw[fill=white] (0 cm, 0 cm) circle (.25cm) node[below=4pt]{$1$};
\draw[fill=white] (2 cm, 0 cm) circle (.25cm) node[below=4pt]{$2$};
\end{tikzpicture}\qquad
\left(\begin{array}{rr}
2 & -3 \\
-1 & 2
\end{array}\right)
$$
We take
$$\underline w_0=s_1s_2s_1s_2s_1s_2.$$
to obatin the exchange matrix
$$\varepsilon= \left(\begin{array}{rrrrrr}
0 & -1 & 1 & 0 & 0 & 0 \\
3 & 0 & -3 & 1 & 0 & 0 \\
-1 & 1 & 0 & -1 & 1 & 0 \\
0 & -1 & 3 & 0 & -3 & 1
\end{array}\right)$$
Taking the restriction to $J'=\{1,2,3,5\}$, we obtain the diagram
$$\xymatrix{
1 \ar[d]_3& 3  \ar[l] & \fbox{5} \ar[l] \\  
2 \ar[ru]^3&    & 
}$$
which is mutation equivalent to 
$$\xymatrix{
1 \ar[d]_3& 3  \ar@<-0.5ex>[l]\ar@<0.5ex>[l]  \ar[r] & \fbox{5}  \\  
2 \ar[ru]^3&    & 
}$$
by taking the composition of mutations $\mu_2\circ \mu_1\circ \mu_3 \circ \mu_2 \circ \mu_1$.
Since the mutable part of the above diagram is of type $\widetilde G_2$ (see \cite[Figure 9.2]{FT21}), which is mutation-finite, we conclude that it is mutation infinite by \cite[Theorem 9.4]{FT21}. Thus we obtain the desired property by Lemma \ref{lem:finiteinfinite}.
\end{proof}

\begin{remark}
    If $\g$ is of type $A_1,A_2,A_3,A_4$ or $B_2$, then the cluster algebra $\C[U^{-}_{w_0}]$ is finite type. Hence the exchange matrices of $\C[U^{-}_{w_0}]$ are mutation finite so that they do \emph{not} satisfy \eqref{eq:manyarrows2}
\end{remark}

\providecommand{\bysame}{\leavevmode\hbox to3em{\hrulefill}\thinspace}
\providecommand{\MR}{\relax\ifhmode\unskip\space\fi MR }
\providecommand{\MRhref}[2]{%
  \href{http://www.ams.org/mathscinet-getitem?mr=#1}{#2}
}
\providecommand{\href}[2]{#2}

\end{document}